\documentclass[12pt]{amsart}

\usepackage{amsmath,amsxtra,amssymb,amsthm,amsfonts}
\usepackage{mathrsfs}
\usepackage[T1]{fontenc}
\usepackage[utf8]{inputenc}
\usepackage{mathtools}   
\usepackage{tikz}
\usepackage{bbm}
\usetikzlibrary{arrows, automata}
\usepackage{color,hyperref}
\definecolor{dblue}{rgb}{0, 0, 0.72}

\allowdisplaybreaks

\numberwithin{equation}{section}
\newtheorem{lemma}{Lemma}[section]

\newtheorem{theorem}[lemma]{Theorem}

\newtheorem{proposition}{Proposition}[section]

\newtheorem{prob}[lemma]{Problem}
\newtheorem{rem}[lemma]{Remark}
\newtheorem{remark}[lemma]{Remark}

\newtheorem{example}[lemma]{Example}
\newtheorem{definition}[lemma]{Definition}
\newtheorem{corollary}[lemma]{Corollary}

\newcommand{\re}{\begin{rem}\rm}
	\newcommand{\mar}{\end{rem}}

\newcommand{\ee }{\mathrm{I}\!\!1}

\renewcommand{\for}{\begin{eqnarray*}}
	
	\newcommand{\mel}{\end{eqnarray*}}

\DeclareMathOperator{\sgn}{sgn}

\DeclareMathOperator{\tr}{tr}

\newcommand{\pl}{\hspace{.1cm}}

\newcommand{\qd}{\end{proof}\vspace{0.5ex}}

\newcommand{\id}{\iota_{\infty,2}^n}

\newcommand{\A}{{\mathcal A}}

\newcommand{\pf}{\begin{proof}}

\newcommand{\be}{\left|{\atop}}

\newcommand{\xspace}{\hbox{\kern-2.5pt}}
\newcommand{\xyspace}{\hbox{\kern-1.1pt}}

\newcommand{\ssubset} {\!\!\subset\! \!}

\allowdisplaybreaks

\definecolor{LightGray}{rgb}{0.94,0.94,0.94}
\definecolor{VeryLightBlue}{rgb}{0.9,0.9,1}
\definecolor{LightBlue}{rgb}{0.8,0.8,1}
\definecolor{DarkBlue}{rgb}{0,0,0.6}
\definecolor{LightGreen}{rgb}{0.88,1,0.88}
\definecolor{MidGreen}{rgb}{0.6,1,0.6}
\definecolor{DarkGreen}{rgb}{0,0.6,0}
\definecolor{DarkGrreen}{rgb}{0,0.8,0}

\definecolor{VeryLightYellow}{rgb}{1,1,0.9}
\definecolor{LightYellow}{rgb}{1,1,0.6}
\definecolor{MidYellow}{rgb}{1,1,0.5}
\definecolor{DarkYellow}{rgb}{0.8,1,0.3}
\definecolor{VeryLightRed}{rgb}{1,0.9,0.9}
\definecolor{LightRed}{rgb}{1,0.8,0.8}
\definecolor{DarkRed}{rgb}{0.8,0.2,0}
\definecolor{DarkRedb}{rgb}{0.6,0.2,0}
\definecolor{DarkLila}{rgb}{0.8,0,1}
\definecolor{Beige}{rgb}{0.96,0.96,0.86}
\definecolor{Gold}{rgb}{1.,0.84,0.}
\definecolor{Goldb}{rgb}{0.7,0.3,0.5}
\definecolor{MyYellow}{rgb}{1.,0.84,0.8}

\def\11{\mathbb{I}}

\usepackage{tikz-cd}

\usepackage{graphicx}

\DeclareRobustCommand\openone{\leavevmode\hbox{\small1\normalsize\kern-.33em1}}

\renewcommand{\id}{\rm{id}}
\renewcommand{\be}{\begin{equation}}
	\renewcommand{\ee}{\end{equation}}
\newcommand{\bea}{\begin{eqnarray}}
	\newcommand{\eea}{\end{eqnarray}}
\newcommand{\beas}{\begin{eqnarray*}}
	\newcommand{\eeas}{\end{eqnarray*}}

\usepackage{amsmath}
\usepackage{ams math}
\usepackage{amssymb}
\usepackage{amsthm}
\usepackage{dsfont}
\usepackage{upgreek}
\usepackage{enumitem}

\usepackage{array}
\usepackage{ams fonts}
\usepackage{graphics}
\usepackage[utf8]{inputenc}
\newtheorem*{theorem*}{Theorem}
\newtheorem*{remark*}{Remark}
\newtheorem*{lemma*}{Lemma}
\newtheorem*{notation*}{Notation}
\newtheorem*{cor*}{Corollary}
\newtheorem*{note*}{Note}
\newtheorem*{prop*}{Proposition}
\newtheorem*{example*}{Example}
\usepackage{tikz-cd}
\usepackage{amsthm}

\usepackage[margin=1.0 in]{geometry}

\title{Noncommutative Poisson Random Measure and Its Applications}
\begin{document}
\author[Y. Chen]{Yidong Chen}
\author[M. Junge]{Marius Junge}
\begin{abstract}
	We introduce a noncommutative Poisson random measure on a von Neumann algebra. This is a noncommutative generalization of the classical Poisson random measure. We call this construction \textbf{Poissonization}. Poissonization is a functor from the category of von Neumann algebras with normal semifinite faithful weights to the category of von Neumann algebras with normal faithful states. Poissonization is a natural adaptation of the second quantization \cite{KRP} to the context of von Neumann algebras. The construction is compatible with normal (weight-preserving) homomorphisms and unital normal completely positive (weight-preserving) maps. We present two main applications of Poissonization. First Poissonization provides a new framework to construct algebraic quantum field theories \cite{haag} that are not generalized free field theories. Second Poissonization permits straight-forward calculations of quantum relative entropies (and other quantum information quantities) in the case of type III von Neumann algebras.
\end{abstract}
	\maketitle
	\tableofcontents
	\section{Introduction}\label{section:introduction}
	In classical probability theory, the Poisson process and p-stable processes are important tools with applications in harmonic analysis and Banach space theory. By the Lévy-Khintchine formula, p-stable random variables can be obtained as a stochastic integral from the Poisson process \cite{ST}. Using p-stable random variables, Dacuhna-Castelle, Bretagnole and Kirvine \cite{BDK} showed that there exists an isometric isomorphism between $L_p(\mathbb{R}_+)$ and $L_q(0,1)$ for $0 < q<p\leq 2$. Using noncommutative generalizations of classical probability constructions \cite{KRP, Meyer}, it was shown in \cite{MJJP} that for any hyperfinite von Neumann algebra $M$, there exists a completely positive isometric embedding from $L_p(M)$ to $L_q(M')$ for $1\leq 1 < p \leq 2$. In particular, the von Neumann algebra $M'$ can be made hyperfinite. The preservation of hyperfiniteness relies crucially on a noncommutative version of Poisson random measure. The detailed construction was based on an unpublished work of Marius Junge. In this paper, we expand the original construction of noncommutative Poisson random measure to general von Neumann algebras with normal semifinite faithful weights. We study various functorial properties of this construction. From the perspective of quantization, the noncommutative Poisson random measure is a natural generalization of the second quantization to the setting of operator algebras. Hence we decide to name our construction \textbf{Poissonization}. 
	
	Recall the second quantization can be understood as a functor from the category of real Hilbert spaces with symplectic forms and contractions to the category of CCR algebras and unital completely positive maps \cite{BR1, BR2}. A CCR algebra comes equipped with distinguished representations by quasi-free states. These states are built out of the inner product on the underlying Hilbert space. From a physics perspective, the quasi-free states calculate the "vacuum" expectation of the multi-particle correlations. The vectors of the underlying Hilbert space represent physical states of a single particle species. An inherent drawback to this quantization approach is its inability to produce connected $n$-point correlation functions. Plainly, using only the inner product it is impossible to produce multilinear forms on tensor products of Hilbert space without breaking the multilinear form into the quasi-free form. Hence without perturbation theory, the second quantization (with the quasi-free states) can only produce the so-called generalized free field theories. 
	
	On the other hand, if the quantization starts from the category of operator algebras instead of Hilbert spaces, one can replace inner products with linear functionals. Then it is almost trivial to product connected $n$-point correlation functions. These functions will simply be produced by the linear functional acting on the product of $n$ operators. From this point of view, it is necessary to construct a generalization of the second quantization to the category of operator algebras with linear functionals.
	
	Moreover, any legitimate quantization procedure must produce unambiguous outcome. Therefore, when quantizing a common "one-particle" algebra, two different realizations of the quantization outcome must be isomorphic in a suitable sense. From a quantum probabilistic perspective, this question can be recast as a noncommutative analog of the classical Hamburger's moment problem \cite{wid}. It was shown in \cite{MAW} that the noncommutative moment problem has a unique solution if the linear functional satisfies certain growth upper bound. This upper bound can be understood as a noncommutative analog of the moments of classical Poisson random variables. Combining these discussions, it is natural to consider a general quantization procedure based on a noncommutative analog of classical Poisson random measure.
	
	To summarize the main results of our construction, we introduce the following notations. Let $\textbf{vNa}_w$ be the category where the objects are von Neumann algebras with a normal faithful semifinite weight $(N,\omega)$ and the morphisms are weight-preserving normal unital completely postive maps. Let $\text{vNa}_s$ be the category where the objects are von Neumann algebras with a normal faithful states and the morphisms are state-preserving normal unital completely positive maps. Then Poissonization can be summarized in the following theorem:
	\begin{theorem}\label{theorem:main}
		Poissonization \textbf{Poiss} is a functor:
		\begin{equation}
			\textbf{Poiss}:\textbf{vNa}_w\Rightarrow\textbf{vNa}_s: (N,\omega)\rightarrow (\mathbb{P}_\omega N, \phi_\omega)
		\end{equation}
		such that the following properties hold:
		\begin{enumerate}
			\item(Noncommutative Poisson moment formula) For self-adjoint affiliated operators $x = x^*\eta N$ such that $e^{ix} - 1\in m_\omega$ (the definition ideal of $\omega$), Poissonization lifts the unitary operator $e^{ix}\in N$ to a unitary operator $\Gamma(e^{ix})\in \mathbb{P}_\omega N$ such that:
			\begin{equation}
				\phi_\omega(\Gamma(e^{ix})) = \exp(\omega(e^{ix} - 1))\pl;
			\end{equation}
			\item(Compatibility of modular automorphism) Poissonization preserves the modular automorphism group:
			\begin{center}
				\begin{tikzcd}
					N \arrow[r, "\textbf{Poiss}"] \arrow[d, "\sigma^\omega_t"] & \mathbb{P}_\omega N \arrow[d, "\sigma^{\phi_\omega}_t = \textbf{Poiss}(\sigma^\omega_t)"] \\
					N \arrow[r, "\textbf{Poiss}"] & \mathbb{P}_\omega N \pl;
				\end{tikzcd}
			\end{center}
			\item The Haagerup $L_2$-space of the Poisson algebra is isomorphic to the symmetric Fock space of the $L_2$-space of the original von Neumann algebra:
			\begin{equation}
				L_2(\mathbb{P}_\omega N, \phi_\omega) \cong \mathcal{F}_s(L_2(N,\omega))\pl;
			\end{equation}
			\item(Generalization of classical Poisson random variables) If $(N,\omega) = (\mathbb{C}, \lambda)$ where the functional is simple multiplication by $\lambda > 0$, then $\mathbb{P}_{\tr}\mathbb{C} = \ell_\infty(\mathbb{N})$ and the state gives the Poisson distribution with intensity $\lambda$:
			\begin{equation}
				\phi_{\tr} (\delta_{k}) = \frac{e^{-\lambda}\lambda^k}{k!}
			\end{equation} 
			where $\delta_k\in\ell_\infty(\mathbb{N})$ is the characteristic function supported at $k\in\mathbb{N}$;
			\item The algebra $\mathbb{P}_{\tr}\mathbb{B}(\ell_2(\mathbb{Z}))$ is the hyperfinite II$_1$ factor. More generally, $\mathbb{P}_{\omega\otimes\tr}N\overline{\otimes}\mathbb{B}(\ell_2)$ is a factor and it is either type III or type II$_1$. The type is determined by the group generated by the Arveson spectrum $Sp(\Delta_\omega)$.
		\end{enumerate}
	\end{theorem}
	As applications, we use Poissonization to construction new examples of type III von Neumann factors from unitary principal series representations of real semisimple Lie groups \cite{Knapp, Knapp1}. It turns out that a collection of these constructions naturally forms a pre-cosheaf of von Neumann algebras that satisfies the axioms of algebraic quantum field theory (AQFT) \cite{haag}. In a separate paper, we will present an in-depth study of numerous other constructions of AQFT using Poissonization. These toy models are not generalized free field theories. Hence it is of great interest to see if these models can provide valid examples of interacting field theories in higher dimension ($\geq 2$). Finally, we show that Poissonization preserves quantum relative entropy in a suitable sense. Since Poissonization naturally provides examples of type III von Neumann algebras, our proposition provides a rigorous and straight-forward way to calculate the relative entropy when density matrices are not well-defined. In a separate paper, we will continue this line of research and study various other quantum information quantities (e.g. tripartite information, out-of-time-order correlation function) using Poissonization. Because of its probabilistic origin, Poissonization provides a natural framework to study the so-called quantum chaotic systems \cite{PCASA}.
	
	This paper is the first part of a series of papers on Poissonization and its connection to various questions in mathematical physics. In this paper, we will lay the mathematical foundation for future projects. The paper is organized as follows. In Section \ref{section:review}, we review the classical Poisson random measure and discuss some motivations of Poissonization. In Section \ref{section:construction}, we construct the Poissonization functor with increasing generality. We start by considering von Neumann algebras with fixed normal faithful functionals, and then generalize to normal weithful semifinite weights. Several equivalent constructions of Poissonization are presented in this section. Each construction emphasizes on a different aspect of Poissonization. After the construction, we study the functorial properties of Poissonization. In particular, we show that Poissonization is compatible with both normal $*$-homomorphism and normal unital completely positive maps. In Section \ref{section:properties}, we study some important properties of Poissonization. In Subsection \ref{subsection:type}, we determine the factoriality and the type of the von Neumann algebras constructed by Poissonization. In Subsection \ref{subsection:representation}, we use Poissonization to construct type III von Neumann algebras from unitary principal series representations and present a simple toy model of algebraic quantum field theory in the sens of \cite{haag}. In Subsection \ref{subsection:split}, we briefly study the structure of Poisson subfactors. It turns out that Poissonization naturally produces infinite index subfactors \cite{Jones1, Jones2, Longo3}. Although it is not yet clear how to define standard invariants for infinite index subfactors, we have enough control on Poissonization to make explicit calculations on the so-called centralizer algebras and the central vectors \cite{JP1}. In Subsection \ref{subsection:relEnt}, we use Poissonization to calculate the relative entropy of certain states in some type III von Neumann algebras.
	\section{Review of Classical Poisson Random Measure}\label{section:review}
	Recall the classical notion of a Poisson random measure associated to a measure space $(\Omega, \Sigma, \mu)$. Denote $\Sigma_{fin}\ssubset\Sigma$ to be the collection of finite measures. A Poisson random measure $\nu_\Omega$ is given by a family of Poisson random variables $(P_A)_{A\in \Sigma_{fin}}$ defined on a probability space $(X,\mathcal{F},\mathbb{P})$. The Poisson random variables are labeled by the measurable sets $A$ such that \cite{classical}:
	\begin{enumerate}
		\item For $A\in\Sigma_{fin}$, $\mathbb{P}(P_A = k) = \frac{e^{-\mu(A)}\mu(A)^k}{k!}$;
		\item If $\{A_1,...,A_n\}$ are mutually disjoint, then $\{P_{A_1},...,P_{A_n}\}$ are mutually independent;
		\item For disjoint sets $A_1,...,A_n\in\Sigma_{fin}$, $P_{\cup_j A_j} = \sum_j P_{A_j}$.
	\end{enumerate}
	Given a Poisson random measure $\nu_{\Omega}$ on the finite measure space $(\Omega,\Sigma,\mu)$ and an integrable function $f\in L^1(\Omega,\mu)$, the characteristic function is given by:
	\begin{equation}
			\mathbb{E}(e^{it \int_\Omega f d\nu_\Omega}) = \exp(\int_\Omega d\mu(x)(e^{itf(x)} - 1))\pl.
	\end{equation}

	We are now ready to lay out some basic requirements a noncommutative analog of Poisson random measure must satisfy. Following the general scheme of quantization, the measure space $(\Omega,\Sigma,\mu)$ is replaced by the noncommutative analog of $L^\infty(\Omega,\mu)$, namely a von Neumann algebra $N$ together with a normal weight $\omega$. Disjoint subsets of $\Omega$ are replaced by orthogonal projections in $N$. The family of noncommutative Poisson random variables should be constructed in a noncommutative probability space. Recall a noncommutative probability space \cite{VND, Meyer, MAW} $(\mathcal{A},\tau)$ is a unital $*$-algebra $\mathcal{A}$ over $\mathbb{C}$ together with a unital linear functional $\tau$. In addition, we recall the notion of strong independence in quantum probability theory. This is the correct analog of the classical independence. 
	\begin{definition}\label{definition:strongIndependence}
		Let $(\mathcal{A},\tau)$ be a noncommutative probability space, then two subalgebras $\mathcal{B}_1,\mathcal{B}_2\ssubset\mathcal{A}$ are strongly independent if for any $x\in\mathcal{B}_1, y\in\mathcal{B}_2$ we have:
		\begin{equation}
			[x,y] = 0 \text{ , }\tau(xy) = \tau(x)\tau(y)\pl.
		\end{equation}
	\end{definition}
	A noncommutative Poisson random measure should be a $*$-preserving linear map \footnote{Linearity ensures that the sum of independent Poisson random variables is another Poisson random variable }:
	\begin{equation}
		\lambda: (N,\omega)\rightarrow (\mathcal{A},\tau)
	\end{equation}
	such that:
	\begin{enumerate}
		\item(Noncommutative independence) For two orthogonal projections $e_1,e_2 \in N$, then $[\lambda(eNe), \lambda(fNf)] = 0$ and the algebras generated by $\lambda(eNe)$ and $\lambda(fNf)$ are strongly independent;
		\item(Noncommutative Poisson moment) $\tau(e^{i\lambda(x)})  = \exp(\omega(e^{ix} - 1))$ where $x =x^*\in N$ such that $e^{ix} -1$ is $\omega$-finite.
	\end{enumerate}
	
	We now briefly summarize the main idea of Poissonization of a hyperfinite matrix algebra $(\mathbb{B}(H),\tr)$. The measure space $(\Omega,\Sigma,\mu)$ is replaced by the projection lattice in $\mathbb{B}(H)$ and disjoint subsets are replaced by orthogonal projections. Fix a family of mutually orthogonal projections $\{e_i\}_{i\in\mathbb{N}}$ such that $f_n:=\sum_{1\leq i \leq n}e_i$ weakly converges to $1$. For finite dimensional matrix algebra, it turns out that the Poissonization of $(M_n(\mathbb{C}),\tr)$ is nothing but the infinite symmetric tensor product of $M_n(\mathbb{C})$. The quantization map $\lambda$ is given by:
	\begin{equation}
		\lambda: M_n(\mathbb{C})\rightarrow M_s(M_n(\mathbb{C})): x\mapsto (\sum_{1\leq j \leq n}\pi_j(x))_n
	\end{equation}
	where $||x|| < 1$ and $\pi_j(x) := 1\otimes...\otimes x\otimes...\otimes 1 \in \bigotimes^nM_n(\mathbb{C})$ with $x$ in the $j$-th component. And the state $\phi_{\tr}$ can be calculated explicitly:
	\begin{equation}
		\phi_{\tr}(\Gamma(e^{ix})) = \exp(\tr(e^{ix} - 1))
	\end{equation}
	where $x$ is a self-adjoint contraction in $M_n(\mathbb{C})$ and $\Gamma(e^{ix})  :=  \exp(i\lambda(x)) = ((e^{ix})^{\otimes n})_{n\geq 0}\in\mathbb{P}_{\tr}M_n(\mathbb{C})\cong M_s(M_n(\mathbb{C}))$. The strong independence is an easy consequence of the following calculation:\footnote{This calculation will be repeated numerous times in the subsequence constructions.}
	\begin{align}
		\begin{split}
			\phi_{\tr}(\Gamma(e^{ix})\Gamma(e^{iy})) &= \exp(\tr(e^{ix}e^{iy} - 1)) = \exp(\tr \big((e^{ix} - 1)(e^{iy} - 1) + (e^{ix} -1 ) + (e^{iy} -1 )\big))
			\\
			&=\exp(\tr(e^{ix} -1 ))\exp(\tr(e^{iy} - 1)) = \phi_{\tr}(\Gamma(e^{ix}))\phi_{\tr}(\Gamma(e^{iy}))
		\end{split}
	\end{align}
	where $x, e^{ix}\in eM_n(\mathbb{C})e$ and $y, e^{iy}\in fM_n(\mathbb{C})f$, and $e\perp f$ in $M_n(\mathbb{C})$. The commutation requirement follows from the fact that $\lambda$ is in fact a Lie algebra homomorphism:
	\begin{equation}
		\lambda([x,y]) = [\lambda(x), \lambda(y)]\pl.
	\end{equation}
	And hence for $x \in e M_n(\mathbb{C}) e, y\in fM_n(\mathbb{C})f$ we have $[\Gamma(e^{ix}), \Gamma(e^{iy})] = 0$.
	
	Then for infinite dimensional type I factor $\mathbb{B}(H)$, its Poissonization is constructed as an inductive limit. The key observation is that the following $*$-homomorphism preserves the Poisson state:
	\begin{equation}
		\pi_{n,k}:\mathbb{P}_{\tr}(f_k\mathbb{B}(H)f_k)\rightarrow \mathbb{P}_{\tr}(f_n\mathbb{B}(H)f_n): \Gamma(x)\mapsto \Gamma(x + f_n - f_k)
	\end{equation}
	where we have:
	\begin{equation}
		\phi_{\tr_n}(\pi_{n,k}\Gamma(x)) = \exp(\tr_n(x + f_n - f_k - f_n)) = \exp(\tr_k(x - f_k)) = \phi_{\tr_k}(\Gamma(x))\pl.
	\end{equation}
	Here the trace $\tr_n$ is the restriction of the trace to the corner $f_n\mathbb{B}(H)f_n$. Therefore we can consider the inductive limit of the sequence $\{\mathbb{P}_{\tr}\mathbb{B}(f_nHf_n), \pi_{n,k}\}_{n,k\in\mathbb{N}}$ and the weak closure of the inductive limit is the Poisson algebra $\mathbb{P}_{\tr}\mathbb{B}(H)$. It is easy to check that the basic requirements of noncommutative Poisson random measure are satisfied in this case as well. We will elaborate on the details of these constructions later.
	\section{Constructions of Poisson Algebra}\label{section:construction}
	In this first section, we construct the Poisson algebra from a von Neumann algebra and a normal faithful weight. A complete characterization of a normal faithful weight is due to Haagerup \cite{T2} and it implies that a weight can be thought of as a limit of positive functionals. For this reason, we first construct the Poisson algebra from the initial data of a von Neumann algebra with a normal faithful state. Then we present the full construction for a general normal faithful weight. 
	
	The construction of the Poisson algebra from a von Neumann algebra with a normal faithful weight is a functor where the target category is the category of von Neumann algebra with a normal faithful state. Henceforth, we will call this functor \textit{Poissonization} and we will call the lifted state \textit{the Poisson state}. All normal $*$-homomorphisms that preserve the weight can be lifted to normal $*$-homomorphisms between two Poisson algebras. The lifted $*$-homomorphisms preserve the Poisson states. In particular, the Poissonization preserves the modular automorphism groups, thereby preserving the intrinsic dynamics of von Neumann algebras. More precisely, the modular automorphism of the initial state (or weight) is lifted to the modular automorphism of the Poisson state via the functorial property of the Poissonization. Later in this section, we will extend the functorial property of the Poissonization and discuss how normal unital completely positive maps (or quantum channels) can be lifted to quantum channels between Poisson algebras. 
	
	Finally, Poissonization constructs the Poisson algebra in a standard form \cite{T}. It canonically acts on the Haagerup $L_2$-space of the Poisson state. It turns out that the Haagerup $L_2$-space of the Poisson algebra is isometrically isomorphic to the symmetric Fock space of the $L_2$-space of the original von Neumann algebra. However, in the Poisson algebra, there is \textit{no} simple way to construct a creation / annhilation operator. The natural candidate for the creation / annhilation operator does \textit{not} act on the Fock basis in the expected way.
	
	\subsection{Poisson Algebra for Normal Faithful Positive Functionals}
	In this subsection, we fix a von Neumann algebra $N$ and a normal faithful positive functional $\omega\in N^+_*$.  $\omega$ may not be normalized. We denote the normalized state to be $\omega_n(x):=\frac{\omega(x)}{\omega(1)}$. We need the following $*$-homomorphisms:
	\begin{equation}
		\pi_j: N\rightarrow N^{\otimes n}: x\mapsto 1\otimes...\otimes x\otimes...\otimes 1
	\end{equation}
	where $x$ is in the $j$-th position. Then we can define the key "quantization" map:
	\begin{equation}\label{equation:quanatization}
		\lambda: N \rightarrow \bigoplus_{n\geq 0} N^{\otimes n}: x\mapsto (\sum_{1\leq j \leq n}\pi_j(x))_{n\geq 0}
	\end{equation}
	where for the 0th component, we define $\lambda(x)_0 := 0$.
	\begin{definition}\label{definition:poissonalgebra(state)}
		The Poisson algebra of $(N,\omega)$ is the ultra-weak closure of $\{\exp(i\lambda(x)): x \in N_{s.a.}\}$ in the von Neumann algebra $\bigoplus_{n\geq 0} N^{\otimes n}$. 
	\end{definition}
	On $\bigoplus_{n \geq 0}N^{\otimes n}$, we can define the positive linear functional:
	\begin{equation}
		\phi_\omega((x_n)_{n\geq 0}):= e^{-\omega(1)}\sum_{n\geq 0 }\frac{\omega^{\otimes n}(x_n)}{n!}\pl.
	\end{equation}
	The exponential factor normalizes $\phi_\omega$. When restricted to the Poisson algebra, $\phi_\omega$ defines the Poisson state that we are looking for.
	\begin{definition}
		Let $(N,\omega)$ be a von Neumann algebra with a fixed normal faithful positive functional. Then Poissonization \textbf{Poiss} constructs the Poisson algebra with a fixed normal faithful state: $(\mathbb{P}_\omega N, \phi_\omega)$.
	\end{definition}
	One of the goals of this section is to extract a few results on the Poisson state $\phi_\omega$ and the corresponding $L_2$-space: $L_2(\mathbb{P}_\omega N, \phi_\omega)$. First of all, we explain why the state $\phi_\omega$ deserves to be called the Poisson state:
	\begin{lemma}\label{lemma:nameexplainPoissonState}
		Let $\Gamma: N\rightarrow \bigoplus_{n\geq 0}N^{\otimes n}: x\mapsto (x^{\otimes n})_n$ be the $*$-homomorphism of (symmetric) quantization. The formula $\Gamma(x) = (x^{\otimes n})_n$ is well-defined when $x$ is a contraction in $N$. Then we have:
		\begin{enumerate}
			\item  The following formula for the generators of $\mathbb{P}_\omega N$ holds:
			\begin{equation}
				\exp(i\lambda(x)) = \Gamma(\exp(ix))
			\end{equation}
			where $x \in N_{s.a.}$.
			\item On a single generator, the Poisson state is given by:
			\begin{equation}\label{equation:Poissoncharacteristic}
				\phi_\omega(\exp(i\lambda(x))) = \exp(\omega(\exp(ix) - 1))\pl.
			\end{equation}
		\end{enumerate}
	\end{lemma}
	\begin{proof}
		For the first claim, for every $n \geq 1$, we have:
		\begin{equation}
			\exp(i\sum_{1\leq j \leq n}\pi_j(x)) = \prod_{1\leq j \leq n}\exp(i \pi_j(x)) = \big(\exp(ix)\big)^{\otimes n}\pl.
		\end{equation}
		Thus, we have:
		\begin{align}
			\begin{split}
				\exp(i\lambda(x)) &= \bigoplus_{n\geq  0}\exp(i\sum_{1\leq j \leq n}\pi_j(x)) = \bigoplus_{n\geq 0}\big(\exp(ix)\big)^{\otimes n} = \Gamma(\exp(ix))\pl.
			\end{split}
		\end{align}
		For the second claim, we have following calculation:
		\begin{align}
			\begin{split}
				\phi_\omega(\exp(i\lambda(x)))&= e^{-\omega(1)}\sum_{n\geq 0}\frac{1}{n!}\omega^{\otimes n}\big((e^{i\lambda(x)})_n\big)
				= e^{-\omega(1)}\sum_{n\geq 0}\frac{1}{n!}\big(\omega(e^{ix})\big)^n
				\\
				& = \exp(\omega(e^{ix}  - 1))\pl.\qedhere
			\end{split}
		\end{align}
	\end{proof}
	\begin{remark}\label{remark:classicalPoisson}
		The classical Poisson random measure associated to a finite measure space $(\Omega, \Sigma, \mu)$ is given by a family of Poisson random variables $(P_A)_{A\in \Sigma}$ labeled by the measurable sets $A$ such that \cite{classical}:
		\begin{align*}
			\mathbb{P}(P_A = k) = \frac{e^{-\mu(A)}\mu(A)^k}{k!}\pl.
		\end{align*}
		In addition, if the measurable sets $A_1,...,A_n$ are disjoint, then the Poisson random variables $P_{A_1},...,P_{A_n}$ are independent and $P_{\cup_j A_j} = \sum_j P_{A_j}$. 
		
		Given a Poisson random measure $\nu_{\Omega}$ on the finite measure space $(\Omega,\Sigma,\mu)$ and an integrable function $f\in L^1(\Omega,\mu)$, the characteristic function is given by:
		\begin{equation*}
			\mathbb{E}(e^{it \int_\Omega f d\nu_\Omega}) = \exp(\int_\Omega d\mu(x)(e^{itf(x)} - 1))\pl.
		\end{equation*}
		Compare the classical characteristic function with Equation \ref{equation:Poissoncharacteristic}, it is clear that $\phi_\omega$ is a direct generalization of the classical Poisson random measure.
	\end{remark}
	An immediate consequence of Remark \ref{remark:classicalPoisson} is the following result of the (noncommutative) moment formula for the Poisson state $\phi_\omega$:
	\begin{lemma}\label{lemma:momentformulaPoisson}
		Let $x_1,...,x_n$ be contractions in $N$, then we have:
		\begin{equation}\label{equation:momentformula}
			\phi_\omega(\lambda(x_1)...\lambda(x_n)) = \sum_{\sigma\in \mathcal{P}(n)}\prod_{A\in\sigma}\phi_\omega(\overrightarrow{\prod}_{i\in A}x_i)
		\end{equation} 
		where $\mathcal{P}_n$ is the set of partitions of $[n]:=\{1,...,n\}$ and $\overrightarrow{\prod}_{i\in A}$ denotes the ordered product over the ordered set $A\ssubset [n]$. The ordering of $A$ is inherited from the canonical ordering on $[n]$. 
		
		Moreover, there exists a universal constant $C > 0$ such that the norm:
		\begin{equation}\label{equation:normdefinition}
			|||x||| := \max\{||x||, \omega(x^*x)^{1/2}, \omega(xx^*)^{1/2}, |\omega(x)|\}
		\end{equation}
		satisfies the following estimates:
		\begin{equation}\label{equation:growth}
			|\phi_\omega(\lambda(x_1)...\lambda(x_n))| \leq C^n n^n \prod_{1\leq i \leq n}|||x_i|||\pl.
		\end{equation}
	\end{lemma}
	\begin{proof}
		For a given $m\in \mathbb{N}$, the $m$-th component of $\lambda(x_1)...\lambda(x_n)$ is given by:
		\begin{align*}
			\begin{split}
				\big(\lambda(x_1)...\lambda(x_n)\big)_m &= \big(\sum_{1\leq j \leq m}\pi_j(x_1)\big)...\big(\sum_{1\leq j \leq m}\pi_j(x_n)\big)  = \sum_{j_1, ..., j_n}\pi_{j_1}(x_1)...\pi_{j_n}(x_n) \\&= \sum_{\substack{f:[n]\rightarrow[m]}} \bigotimes_{1\leq j \leq m} \big(\overrightarrow{\prod}_{i \in f^{-1}(j)} x_i\big)
			\end{split}
		\end{align*}
		where the summation is over all possible set maps $f:[n]\rightarrow [m]$ and the product $\overrightarrow{\prod}$ is the ordered product over the subset $f^{-1}(j)$. The order structure on $f^{-1}(j)\ssubset [n]$ is inherited from the canonical ordering on $[n]$. If the set $f^{-1}(j)$ is empty, the product is interpretted to give $1\in N$. Using this formula, we have:
		\begin{equation*}
			\omega^{\otimes m}\big((\lambda(x_1)...\lambda(x_n))_m\big) = \sum_{f:[n]\rightarrow [m]}\prod_{1\leq j \leq m}\omega(\overrightarrow{\prod}_{i\in f^{-1}(j)}x_i)\pl.
		\end{equation*}
		For each set function $f:[n]\rightarrow[m]$, the collection of non-empty subsets $\{f^{-1}(j)\neq \emptyset: j\in[m]\}$ forms a partition of $[n]$. For a fixed partition $\sigma\in\mathcal{P}(n)$, let $|\sigma|$ be the number of subsets in this partition. Then we have:
		\begin{equation*}
			\omega^{\otimes m}\big((\lambda(x_1)...\lambda(x_n))_m\big) = \sum_{\substack{\sigma\in\mathcal{P}(n)\\ |\sigma| \leq m}}|\sigma|!\binom{m}{|\sigma|} \omega(1)^{m - |\sigma|}\prod_{A\in \sigma}\omega(\overrightarrow{\prod}_{i\in A}x_i)
		\end{equation*}
		where the combinatorial factor counts the number of set functions $f:[n]\rightarrow[m]$ that gives the same partition $\sigma\in \mathcal{P}(n)$.
		Therefore we have:
		\begin{align*}
			\begin{split}
				\phi_\omega(\lambda(x_1)...\lambda(x_n)) &= e^{-\omega(1)}\sum_{m\geq 0}\frac{1}{m!}\omega^{\otimes m}\big((\lambda(x_1)...\lambda(x_n))_m\big)\\& = e^{-\omega(1)}\sum_{m\geq 0}\frac{|\sigma|!}{m!}\binom{m}{|\sigma|}\sum_{\substack{\sigma\in\mathcal{P}(n)\\ |\sigma| \leq m}}\omega(1)^{m - |\sigma|}\prod_{A\in\sigma}\omega(\overrightarrow{\prod}_{i\in A}x_i)\\
				&=e^{-\omega(1)}\sum_{\sigma\in\mathcal{P}(n)}\sum_{m \geq |\sigma|}\frac{\omega(1)^{m - |\sigma|}}{(m - |\sigma|)!}\prod_{A\in\sigma}\omega(\overrightarrow{\prod}_{i\in A}x_i) = \sum_{\sigma\in\mathcal{P}(n)} \prod_{A\in\sigma}\omega(\overrightarrow{\prod}_{i\in A}x_i)\pl.
			\end{split}
		\end{align*}
		To prove the upper bound on the moment formula, we first observe that the following estimate holds by Cauchy-Schwarz inequality:
		\begin{equation*}
			|\omega(x_1...x_n)| \leq \omega(x_1^*x_1)^{\frac{1}{2}}\omega(x_n^*...x_2^*x_2...x_n) \leq \omega(x_1^*x_1)^{\frac{1}{2}}(\prod_{2\leq i \leq n-1}||x_i||) \omega(x_n^*x_n)^{\frac{1}{2}} \leq \prod_{1\leq i \leq n}|||x_i|||
		\end{equation*}
		Using this observation, we have:
		\begin{align*}
			\begin{split}
				|\phi_\omega(\lambda(x_1)...\lambda(x_n))| &\leq \sum_{\sigma\in\mathcal{P}(n)} \prod_{A\in\sigma}|\omega(\overrightarrow{\prod}_{i\in A}x_i)| \leq \sum_{\sigma\in \mathcal{P}(n)}\prod_{A\in\sigma} \prod_{i\in A}|||x_i||| = |\mathcal{P}(n)|\prod_{1\leq i \leq n}|||x_i|||\pl.
			\end{split}
		\end{align*}
		The number of partitions of $[n]$ is given by the sum of Stirling numbers of the second kind:
		\begin{align*}
			\begin{split}
				|\mathcal{P}(n)| &= \sum_{1\leq k \leq n} S(n,k) = \mathbb{E}(X^n)
			\end{split}
		\end{align*}
		where $X\sim \text{Poiss}(1)$ is a Poisson random variable with intensity $1$. Then by the classical estimates of the moments of a Poisson random variable \cite{classical}, we have:
		\begin{equation*}
			|\phi_\omega(\lambda(x_1)...\lambda(x_n))| \leq C^n n^n \prod_{1\leq i \leq n}|||x_i|||\pl.\qedhere
		\end{equation*}		
	\end{proof}
	\begin{remark}
		The moment formula of the Poisson state $\phi_\omega$ is a direct generalization of the classical moment formula for a Poisson random measure.
	\end{remark}
	Before we study the structure of the Haagerup space: $L_2(\mathbb{P}_\omega N, \phi_\omega)$, we give an alternative description of $\mathbb{P}_\omega N$.
	\begin{proposition}\label{proposition:symmetrictensor}
		Let $\bigotimes^n_s N$ be the $n$-fold symmetric tensor product of $N$ and let $M_s(N) :=\bigoplus_{n\geq 0}\bigotimes_s^n N$ be the infinite symmetric tensor product of $N$. Then the following statements are true:
		\begin{enumerate}
			\item $\bigotimes_s^n N = \overline{\text{span}\{x^{\otimes n}:x\in N\}}^{\text{ultra-weak}}$;
			\item Let $\lambda_n(x) = \sum_{1\leq j \leq n}\pi_j(x)$ be the n-th component of $\lambda(x)$. Define the following elements in $\bigotimes_s^n N$:
			\begin{equation*}
				\lambda_{n,k}(x) := \sum_{\substack{f:[k]\rightarrow[n]}}\bigotimes_{1\leq i \leq n}x^{|f^{-1}(i)|}
			\end{equation*}
			where $f:[k]\rightarrow[n]$ is a set map and by definition $\sum_{1\leq i \leq n}|f^{-1}(i)| =k $. Then $\lambda_{n,k}(x) \in \text{span}\{\lambda_n(x^{p_1})...\lambda_n(x^{p_j}): j\leq k , p_1,...,p_j \in \mathbb{N}\}$;
			\item $\mathbb{P}_\omega N\cong M_s(N)$.
		\end{enumerate}
	\end{proposition}
	\begin{proof}
		To prove the first claim, for each $n\in\mathbb{N}$ let $E_n$ be the conditional expectation from $N^{\otimes n}$ to the $n$-fold symmetric tensor product: $\bigotimes_s^n N$. Since $\text{span}\{x_1\otimes...\otimes x_n: x_1,...,x_n\in N\}$ is ultra-weakly dense in $N^{\otimes n}$, we only need to show that $E_n(x_1\otimes...\otimes x_n)$ is in $\text{span}\{x^{\otimes n}:x\in N\}$. For this we consider the contour integral:
		\begin{align*}
			\begin{split}
				\frac{1}{(2\pi i)^n}\oint_{|z_i| = 1} \frac{dz_1}{z_1^2}...\frac{dz_n}{z_n^2}(\sum_{1\leq j \leq n}x_jz_j)^{\otimes n} &=\sum_{j_1,...j_n = 1}^n x_{j_1}\otimes...\otimes x_{j_n}\prod_{1\leq i \leq n}(\frac{1}{2\pi i}\oint_{|z| = 1}\frac{dz}{z^2}z^{|k: j_k = i|})
				\\& = \sum_{\pi \in S_n}x_{\pi(1)}\otimes...\otimes x_{\pi(n)} = n!E_n(x_1\otimes...\otimes x_n)
			\end{split}
		\end{align*}
		where in the last equation, we used the fact that the only non-zero terms come from subscripts $\{(j_1,...,j_n): |k:j_k = i| = 1 \text{ for all i}\}$. For all such tuples of subscripts, there exists a permutation $\pi\in S_n$ such that $j_i = \pi(i)$ for all $1\leq i \leq n$. Therefore we have $E_n(x_1\otimes...\otimes x_n)\in \text{span}\{x^{\otimes n}: x\in N\}$.
		
		For the second claim, we use induction on $k$. For $k = 1$, we have $\lambda_{n,1}(x) = \lambda_n(x)$. The claim is trivial. For $k+1$, it is not difficult to see that we have:
		\begin{equation}
			\lambda_{n,k}(x)\lambda_n(x) = \lambda_{n,k+1}(x)
		\end{equation}
		since each map $f:[k+1]\rightarrow [n]$ is given by the restriction $f|_{[k]}$ and a map $g:[1]\rightarrow [n]$ that specifies the assignment of $k+1$.
		Then by simple induction, we have proved the claim.
		
		In particular, since $\lambda_{n,n}(x) = x^{\otimes n}$, we have shown that $x^{\otimes n}\in \text{span}\{\lambda_n(x^{p_1})...\lambda_n(x^{p_j}): j\leq n, p_1,...,p_j\in\mathbb{N}\}$. Therefore, since $\bigotimes_s^n N$ is the ultra-weak closure of $x^{\otimes n}$, $\bigotimes_s^n N$ is the ultra-weak closure of the algebra generated by $\lambda_n(x)$:
		\begin{equation}
			\bigotimes_s^n N = \{\lambda_n(x): x\in N\}''\pl.
		\end{equation}
		To prove the last claim, we use the number operator $\lambda(1)$. More precisely, let $f(t):= e^{-t^2/2}$, we have:
		\begin{equation}
			f(\lambda(1)) = \frac{1}{2\pi}\int_\mathbb{R} \widehat{f}(t)e^{it\lambda(1)}dt
		\end{equation}
		where $\hat{f}(t)$ is the Fourier transform of the Gaussian function $f$. Since by definition $e^{it\lambda(1)}\in \mathbb{P}_\omega N$, we have $f(\lambda(1))\in\mathbb{P}_\omega N$. Since $f(\lambda(1)) = \sum_{n\geq 0 }e^{-n^2/2}e_n$ and $e_n$ is the projection onto $\bigotimes_s^n N$, then $e_n\in \mathbb{P}_\omega N$ for all $n\in\mathbb{N}$. Then we have the following differentiation:
		\begin{equation*}
			i\lambda_n(x) = \frac{d}{dt}|_{t = 0}e_n e^{it\lambda(x)}e_n
		\end{equation*}
		where the differentiation is well-defined in the norm topology. Therefore $\lambda_n(x)\in\mathbb{P}_\omega N$. And by the previous two results, we have: $\mathbb{P}_\omega N\cong M_s(N)$.
	\end{proof}
	\subsection{Noncommutative Bernoulli Approximation and Poisson Algebra}
	So far, Poissonization seems to be completely analogous to the second quantization. However, we will now see that the Hilbert space $L_2(\mathbb{P}_\omega N, \phi_\omega)$ has a grading structure that is different from the grading structure of $M_s(N)$. First, we use a noncommutative analog of the classical Bernoulli approximation to the Poisson process to give yet another description of $\mathbb{P}_\omega N$.
	\begin{proposition}\label{proposition:grading}
		Let $w$ be a free ultrafilter on $\mathbb{N}$. For $n \geq \omega(1)$, define:
		\begin{equation}
			\widetilde{\lambda}_n(x):=\sum_{1\leq j \leq n}\pi_j(x\otimes \chi_{[0, 1/n]}) \in \bigotimes^n N\overline{\otimes} L^\infty([0,\omega(1)^{-1}])\pl.
		\end{equation}
		Otherwise, define $\widetilde{\lambda}_n(x):=\sum_{1\leq j \leq n}\pi_j(x\otimes 1)\in \bigotimes^n N\overline{\otimes}L^\infty([0,\omega(1)^{-1}])$. Consider the operator $\lambda^w(x):= (\widetilde{\lambda}_n(x))^\bullet$ affiliated with the (Ocneanu-)ultraproduct \cite{AH}:
		\begin{equation*}
			\prod^w(N\overline{\otimes} L^\infty([0,\omega(1)^{-1}]))^{\otimes n}\pl.
		\end{equation*}Then the von Neumann algebra $\mathcal{N}^\infty_\omega$ generated by $\{\exp(i\lambda^w(x)):x\in N_{s.a.}\}$ is isomorphic to $\mathbb{P}_\omega(N)$.
	\end{proposition}
	\begin{proof}
		On $N\overline{\otimes} L^\infty([0,\omega(1)^{-1}])\cong L^\infty([0,\omega(1)^{-1}], N)$, define the normal faithful state:
		\begin{equation*}
			\omega_\infty(f(t)) := \int_0^{\omega(1)^{-1}} dt \omega(f(t))\pl.
		\end{equation*} Then on the ultraproduct, the family of states $(\omega_\infty^{\otimes n})_n$ induces a normal state $\phi^w_\omega$ such that:
		\begin{equation*}
			\phi^w_\omega((x_n)_n) := \lim_{n\rightarrow\infty}\omega_\infty^{\otimes n}(x_n)\pl.
		\end{equation*}
		Let $e_\omega\in \prod^w(N\overline{\otimes}L^\infty([0,\omega(1)^{-1}]))^{\otimes n}$ be the support of $\phi^w_{\omega}$. Then on the corner (the Groh-Raynaud ultraproduct) $\mathcal{N}^w$, the state $\phi^w_\omega$ is faithful. 
		
		In addition, using the same calculation as in Lemma \ref{lemma:momentformulaPoisson}, we can check that the induced state $\phi^w_\omega$ satisfies the same moment formula as $\phi_\omega$ on the Poisson algebra $\mathbb{P}_\omega N$. To see this, we first derive the following formula for the tensor state when $n \geq \omega(1)$:
		\begin{align}
			\begin{split}
				\omega_\infty^{\otimes n}(\widetilde{\lambda}_n(x_1)...\widetilde{\lambda}_n(x_m)) &= \sum_{f:[m]\rightarrow[n]}\prod_{1\leq j \leq n}\omega_\infty(\overrightarrow{\prod}_{i\in f^{-1}(j)}x_i\otimes \chi_{[0, \omega(1)/n]})
				\\
				&=\sum_{\substack{\sigma\in\mathcal{P}(m)\\ |\sigma| \leq n}}|\sigma|!\binom{n}{|\sigma|}\big(\prod_{A\in \sigma}\omega(\overrightarrow{\prod}_{i\in A}x_i)\big)(\frac{1}{n})^{|\sigma|}\pl.
			\end{split}
		\end{align}
		Since $\lim_{n\rightarrow \infty}\frac{|\sigma|!\binom{n}{|\sigma|}}{n^{|\sigma|}} = 1$, we have the following moment formula:
		\begin{equation}
			\phi^w_\omega(\lambda^w(x_1)...\lambda^w(x_n)) = \lim_{m\rightarrow \infty}\omega_\infty^{\otimes m}(\widetilde{\lambda}_m(x_1)...\widetilde{\lambda}_m(x_n)) = \sum_{\sigma\in\mathcal{P}(n)}\prod_{A\in\sigma}\omega(\overrightarrow{\prod}_{i\in A}x_i)\pl.
		\end{equation}
		Following the results in \cite{MAW}, we consider the unital $*$-algebra $\mathcal{A}$ generated by abstract symbols $\{\lambda'(x): x\in N^\omega_{ana}\ssubset N\}$ where \begin{equation*}
			N^\omega_{ana} :=\{x\in N: |||\sigma_z(x)||| \leq C\exp(cIm(z)) \text{ for some positive constants C, c}\}\pl.
		\end{equation*}  
		$N^\omega_{ana}$ is the set of analytic elements where the modular automorphism group admits an analytic continuation in a strip in the complex plane. In particular, since $\omega$ is a normal positive functional, $1\in N^\omega_{ana}$. Let $\phi'$ be a positive linear functional on $\mathcal{A}$ such that:
		\begin{equation*}
			\phi'(\lambda'(x_1)...\lambda'(x_n)) := \sum_{\sigma\in\mathcal{P}(n)}\prod_{A\in\sigma}\omega(\overrightarrow{\prod}_{i\in A}x_i)\pl.
		\end{equation*}
		Here a functional is positive if $\phi'(a^*a) \geq 0$ and $\phi'(a^*) = \overline{\phi'(a)}$ for $a\in\mathcal{A}$. Then both $L_2(\mathbb{P}_\omega N, \phi_\omega)$ and $L_2(\mathcal{N}^w, \phi^w_\omega)$ are representations of the abstract $*$-probability space $(\mathcal{A},\phi')$. Since the operators $\lambda^w(x)$ are affiliated to the corner $\mathcal{N}^w$, the von Neumann algebra $\mathcal{N}^\infty_\omega$ is a subalgebra in $\mathcal{N}^w$. Then by a result in \cite{MAW}, there exists a normal $*$-isomorphism: $\varphi: \mathbb{P}_\omega N\rightarrow \mathcal{N}^\infty_\omega$ such that the states are preserved:
		\begin{equation*}
			\phi^w_\omega\circ\varphi = \phi_\omega\pl.\qedhere
		\end{equation*}
	\end{proof}
	\begin{corollary}\label{corollary:grading}
		Using the same notation as Proposition \ref{proposition:grading}, there exists a normal conditional expectation:
		\begin{equation*}
			E: \mathcal{N}^w\rightarrow \mathbb{P}_\omega N
		\end{equation*}
		where $\mathcal{N}^w$ is the (Groh-Raynaud) ultraproduct of $\big(N\overline{\otimes}L^\infty([0,\omega(1)^{-1}])^{\otimes n}, \omega_\infty^{\otimes n} \big)$. The normal conditional expectation preserves the state: $\phi^w_\omega = \phi_\omega\circ E$.
	\end{corollary}
	\begin{proof}
		Recall $N^\omega_{ana}$ is the set of analytic elements in $N$. It is obvious that $\lambda^w(\sigma_t^\omega(x)) = \sigma_t^{\phi^w_\omega}(\lambda^w(x))$ for $x\in N_{ana}^\omega$. Therefore by Takasaki's theorem \cite{T}, the normal conditional expectation exists and it preserves the state. By Proposition \ref{proposition:grading}, the image of this conditional expectation is isomorphic to $\mathbb{P}_\omega N$.
	\end{proof}
	Now we are ready to study the grading structure of $L_2(\mathbb{P}_\omega N, \phi_\omega)$. Let $\xi_\omega\in L_2(\mathbb{P}_\omega N, \phi_\omega)$ be the vector representing $\phi_\omega$ and let $\xi^w_\omega\in L_2(\mathcal{N}^w, \phi^w_\omega)$ be the vector representing $\phi^w_\omega$. Then by Corollary \ref{corollary:grading}, the normal conditional expectation $E$ extends to an $L_2$-projection such that: $E\xi^w_\omega = \xi_\omega$. To state the structural lemma on $L_2(\mathbb{P}_\omega N, \phi_\omega)$, we first introduct the following definition:
	\begin{definition}\label{definition:grading}
		Let $\sigma\in\mathcal{P}(n)$ be a partition of $[n]$. For every $m\geq |\sigma|$, a set function $f:[n]\rightarrow [m]$ is subordinated to $\sigma$ if the partition induced by $f$ equals to $\sigma$:
		\begin{equation*}
			f(i) = f(j) \iff \exists A\in\sigma \text{ such that } i,j\in A\pl.
		\end{equation*}
 	\end{definition}
	Let $F_{n,m}(\sigma)$ be the set of $\sigma$-subordinated functions from $[n]$ to $[m]$. If $m < |\sigma|$, then we define $F_{n,m}(\sigma)$ to be the full set of set functions from $[n]$ to $[m]$. Then we can define the following element in the ultraproduct:
	\begin{equation*}
		\lambda^w_\sigma(x_1,...,x_n):=\big(\sum_{f\in F_{n,m}(\sigma)}\pi_{f(1)}(x_1\otimes \chi_{[0,1/m]})...\pi_{f(n)}(x_n\otimes \chi_{[0,1/m]})\big)^\bullet\pl.
	\end{equation*}
	The $m$-th component of $\lambda_\sigma^w$ is denoted as $\lambda_{\sigma,m}^w$. Using these notations, we can state the following lemma:
	\begin{lemma}\label{lemma:emptybasis}
		Let $\sigma_\emptyset\in\mathcal{P}(n)$ be the singleton partition: $\{\{1\},...,\{n\}\}$. Consider the vectors in $L_2(\mathbb{P}_\omega N, \phi_\omega)$:
		\begin{equation}\label{equation:emptyset}
			\lambda_\emptyset(x_1,...,x_n)\xi_\omega := E\lambda_{\sigma_\emptyset}^w(x_1,...,x_n)\xi_\omega^w\pl.
		\end{equation}
		Then the following statements are true:
		\begin{enumerate}
			\item For any partition $\sigma = \{A_1,...,A_{|\sigma|}\}\in\mathcal{P}(n)$, we have:
			\begin{equation}
				E\lambda^w_\sigma(x_1,...,x_n)\xi^w_\omega = \lambda_\emptyset(\overrightarrow{\prod}_{i_1\in A_1}x_{i_1}, ..., \overrightarrow{\prod}_{i_{|\sigma|}\in A_{|\sigma|}}x_{i_{|\sigma|}})\xi_\omega\pl.
			\end{equation}
			\item For $x\in N_{ana}^\omega$, we have:
			\begin{equation}\label{equation:emptysetaction}
				\lambda(x)\lambda_\emptyset(y_1,...,y_n)\xi_\omega = \lambda_\emptyset(x,y_1,...,y_n)\xi_\omega + \sum_{1\leq i \leq n}\lambda_\emptyset(y_1,...,xy_i,...,y_n)\xi_\omega\pl.
			\end{equation}
			\item Each vector of the form $\lambda(x_1)...\lambda(x_n)\xi_\omega$ can be written as a linear combination of $\lambda_\emptyset$-vectors:
			\begin{equation}
				\lambda(x_1)...\lambda(x_n)\xi_\omega = \lambda_\emptyset(x_1,...,x_n)\xi_\omega+ \sum_{m < n}\alpha_{x'} \lambda_\emptyset(x'_1,..,x'_m)\xi_\omega
			\end{equation}
			where $\alpha_{x'}\in\mathbb{C}$'s are coefficients.
			\item The space $K:=\text{span}\{\lambda_\emptyset(x_1,...,x_n): x_i \in N^\omega_{ana}\}$ is dense in $L_2(\mathbb{P}_\omega N, \phi_\omega)$.
			\item The inner product between two $\lambda_\emptyset$-vectors satisfies:
			\begin{align}\label{equation:emptysetInnerProduct}
				\begin{split}
					\langle&\lambda_\emptyset(x_1,...,x_n)\xi_\omega,\lambda_\emptyset(y_1,...,y_m)\xi_\omega\rangle \\&= \sum_{0\leq p \leq \min\{m,n\}}\sum_{\substack{A_n \sqcup B_n = [n] \\ A_m \sqcup B_m = [m] \\ |B_n| = |B_m| = p}}\prod_{i\in B_n, j \in B_m}\omega(x_i^*y_j)\prod_{i\in A_n}\omega(x_i^*)\prod_{j\in A_m}\omega(y_j)
				\end{split}
			\end{align}
			\item The norm of these vectors can be bounded above:
			\begin{equation}
				||\lambda_\emptyset(x_1,...,x_n)\xi_\omega||^2_2 \leq Cn^{2n}\prod_{1\leq i \leq n}|||x_i|||^2
			\end{equation}
			where $C > 0$ is an absolute constant.
		\end{enumerate}
	\end{lemma}
	\begin{proof}
		The first claim follows directly from a straight-forward calculation:\begin{align*}
			\begin{split}
				\lambda^w_\sigma(x_1,...,x_n) &= (\lambda^w_{\sigma,m}(x_1,...,x_n))^\bullet = (\sum_{f\in F_{n,m}(\sigma)}\pi_{f(1)}(x_1\otimes \chi_{[0,1/m]})...\pi_{f(n)}(x_n\otimes \chi_{[0,1/m]}))^\bullet
				\\
				&=\big(\sum_{i_1\neq i_2\neq...\neq i_{|\sigma|}\in[m]}\pi_{i_1}(\overrightarrow{\prod}_{j_1\in A_1}x_{j_1}\otimes\chi_{[0,1/m]})...\pi_{i_{|\sigma|}}(\overrightarrow{\prod}_{j_{|\sigma|}\in A_{|\sigma|}}x_{j_{|\sigma|}}\otimes\chi_{[0,1/m]})\big)^\bullet\\
				&=\lambda^w_{\emptyset}(\overrightarrow{\prod}_{i_1\in A_1}x_{i_1},...,\overrightarrow{\prod}_{i_{|\sigma|}\in A_{|\sigma|}}x_{i_{|\sigma|}})\pl.
			\end{split}
		\end{align*}
		The first claim follows directly from this calculation.
		
		For the second claim, since $\lambda(x)$ is affiliated with $\mathbb{P}_\omega N$, we have:
		\begin{align}\label{equation:combinatorial}
			\begin{split}
				\lambda(x)&\lambda_\emptyset(y_1,...,y_n)\xi_\omega = E\lambda^w(x)\lambda^w_{\sigma_\emptyset}(y_1,...,y_n)\xi_\omega^w = E\big(\lambda^w_m(x)\lambda^w_{\sigma_\emptyset, m}(y_1,...,y_n)\big)^\bullet\xi^w_\omega\\
				& = E\big(\sum_{1\leq j \leq m}\pi_j(x\otimes \chi_{[0,\frac{1}{m}]})\sum_{f\in F_{n,m}(\sigma_\emptyset)}\pi_{f(1)}(y_1\otimes \chi_{[0,\frac{1}{m}]})...\pi_{f(n)}(y_n\otimes\chi_{[0,\frac{1}{m}]})\big)^\bullet\xi^w_\omega\\
				&= E\big(\sum_{g\in F_{n+1,m}(\sigma_\emptyset)}\pi_{g(1)}(x\otimes \chi_{[0,\frac{1}{m}]})\pi_{g(2)}(y_1\otimes\chi_{[0,\frac{1}{m}]})...\pi_{g(n+1)}(y_n\otimes\chi_{[0,\frac{1}{m}]})\big)^\bullet\xi_\omega^w\\
				& + E\big(\sum_{\substack{1\leq j \leq n \\ f\in F_{n,m}(\sigma)}}\pi_{f(1)}(y_1\otimes\chi_{[0,\frac{1}{m}]})...\pi_{f(j)}(xy_j\otimes\chi_{[0,\frac{1}{m}]})...\pi_{f(n)}(y_n\otimes\chi_{[0,\frac{1}{m}]})\big)^\bullet\xi^w_\omega\\
				&= \lambda_\emptyset(x,y_1,...,y_n)\xi_\omega + \sum_{1\leq j \leq n}\lambda_\emptyset(y_1,...,xy_j,...,y_n)\xi_\omega\pl.
			\end{split}
		\end{align}
		In the first equation, we have used the fact that $\lambda(x) = \lambda^w(x)$ is affiliated with $\mathbb{P}_\omega N$. 
		
		For the third claim, we use Equation \ref{equation:combinatorial} and induction on $n$. For $n = 1$, we have $\lambda_\emptyset(x)\xi_\omega = \lambda(x)\xi_\omega$. Assume the claim is true upto $k$. For $k + 1$, Equation \ref{equation:combinatorial} implies:
		\begin{equation}\label{equation:combinatorialReverse}
			\lambda(x_1)\lambda_\emptyset(x_2,...,x_{k+1})\xi_\omega  =\lambda_\emptyset(x_1,...,x_{k+1})\xi_\omega + \sum_{2\leq j \leq k+1}\lambda_\emptyset(x_2,...,x_1x_j,...,x_{k+1})\xi_\omega\pl.
		\end{equation}
		By inductive hypothesis, $\lambda_\emptyset(x_2,...,x_{k+1})\xi_\omega = \lambda(x_2)...\lambda(x_{k+1})\xi_{\omega} - \sum_{k' < k}\alpha_{x'}\lambda_\emptyset(x'_1,...,x'_{k'})\xi_\omega$ for some constants $\alpha_{x'}\in\mathbb{C}$. Thus Equation \ref{equation:combinatorialReverse} gives a formula for $\lambda(x_1)...\lambda(x_{k+1})\xi_\omega$ in terms of $\lambda_\emptyset$-vectors where $\lambda_\emptyset(x_1,...x_{k+1})\xi_\omega$ appears in the linear combination only once. Therefore by induction, the third claim is proved.
		
		The fourth claim is a direct consequence of the third claim. Since by construction $\overline{\text{span}\{\lambda(x_1)...\lambda(x_n)\xi_\omega: x_i \in N^\omega_{ana}\}} = L_2(\mathbb{P}_\omega N, \phi_\omega)$, the combinatorial formula shows that $\lambda(x_1)...\lambda(x_n)\xi_\omega\in K$ for all tuples of analytic elements $x_1,...,x_n \in N^\omega_{ana}$. Therefore $\overline{K}= L_2(\mathbb{P}_\omega N, \phi_\omega)$.
		
		To prove the formula of the inner product, we use the ultraproduct model:
		\begin{align*}
			\begin{split}
				\langle&\lambda_\emptyset(x_1,...,x_n)\xi_\omega,\lambda_\emptyset(y_1,...,y_m)\xi_\omega\rangle = \langle\lambda^w_{\sigma_\emptyset}(x_1,...,x_n)\xi^w_\omega, \lambda^w_{\sigma_\emptyset}(y_1,...,y_m)\xi^w_\omega\rangle\\
				&=\lim_{l\rightarrow \infty}\omega_\infty^{\otimes l}\bigg(\big(\sum_{f\in F_{n,l}(\sigma_\emptyset)}\pi_{f(n)}(x^*_n\otimes\chi_{[0, \frac{1}{l}]})...\pi_{f(1)}(x^*_1\otimes\chi_{[0,\frac{1}{l}]})\big)^\bullet\big(\sum_{g\in F_{m,l}(\sigma_\emptyset)}\pi_{g(1)}(y_1\otimes\chi_{[0,\frac{1}{l}]})...\\&\times\pi_{g(m)}(y_m\otimes\chi_{[0,\frac{1}{l}]})\big)^\bullet\bigg)\\
				& = \lim_{l \rightarrow\infty}\omega_\infty^{\otimes l}\bigg(\sum_{\substack{0\leq p \leq \min\{n,m\} \\ A\subset [l], |A| = p \\ f\in F_{n,l}(\sigma_\emptyset), g\in F_{m,l}(\sigma_\emptyset)}}\prod_{i\in A}\pi_i(x^*_{f^{-1}(i)}y_{g^{-1}(i)}\otimes \chi_{[0,\frac{1}{l}]})\prod_{i\in Im(f) - A}\pi_i(x^*_{f^{-1}(i)}\otimes\chi_{[0,\frac{1}{l}]})\\&\times\prod_{i\in Im(g) - A}\pi_i(y_{g^{-1}(i)}\otimes\chi_{[0,\frac{1}{l}]})\bigg)
				\\
				&=\lim_{l\rightarrow \infty}\sum_{0\leq p \leq \min\{m,n\}}\sum_{\substack{A_n \sqcup B_n = [n] \\ A_m \sqcup B_m = [m] \\ |B_n| = |B_m| = p}}\frac{\omega(1)^{-l + n + m - p}}{l^{n+m - p}}\binom{l}{n+m - p}\binom{n+m - p}{p}p!(n-p)!(m-p)! \\&\times\prod_{i\in B_n, j \in B_m}\omega(x_i^*y_j)\prod_{i\in A_n}\omega(x_i^*)\prod_{j\in A_m}\omega(y_j)\times \omega(1)^{l-n-m+p}
				\\
				&=\sum_{0\leq p \leq \min\{m,n\}}\sum_{\substack{A_n \sqcup B_n = [n] \\ A_m \sqcup B_m = [m] \\ |B_n| = |B_m| = p}}\prod_{i\in B_n, j \in B_m}\omega(x_i^*y_j)\prod_{i\in A_n}\omega(x_i^*)\prod_{j\in A_m}\omega(y_j)
			\end{split}
		\end{align*}
		where we used the fact that any set function in $ F_{n,l}(\sigma_\emptyset)$ is injective and hence $f^{-1}$ is well defined. 
		
		Finally, using Equation \ref{equation:emptysetInnerProduct}, we have the following upper bound on the norm of $\lambda_\emptyset$-vectors:
		\begin{align*}
			\begin{split}
				||\lambda_\emptyset(x_1,...,x_n)\xi_\omega||^2_2 &\leq \sum_{0\leq p\leq n}\sum_{\substack{A_n\sqcup B_n = [n] \\ A_m\sqcup B_m = [n] \\ |B_n| = |B_m| = p}} |\prod_{i\in B_n, j\in B_m}\omega(x^*_i x_j)\prod_{i\in A_n}\omega(x^*_i)\prod_{j\in A_m}\omega(x_j)|
				\\
				& \leq \big(\sum_{0\leq p \leq n}p!\binom{n}{p}^2\big)\prod_{1\leq i \leq n}|||x_i|||^2 \leq \big(\sum_{0\leq p \leq n}\frac{n^{2p}}{p!}\big)\prod_{1\leq i \leq n}|||x_i|||^2\\&\leq C n^{2n}\prod_{1\leq i \leq n}|||x_i|||^2
			\end{split}
		\end{align*}
		where $C > 0$ is some absolute constant. 
	\end{proof}
	\begin{corollary}\label{corollary:meanzero}
		Let $x_i, y_j\in N^\omega_{ana}$ be mean zero analytic elements (i.e. $\omega(x_i) = \omega(y_j) = 0$), then we have:
		\begin{equation}\label{equation:meanzero}
			\langle\lambda_\emptyset(x_1,...,x_n)\xi_\omega, \lambda_\emptyset(y_1,...,y_m)\xi_\omega\rangle = \delta_{n,m}\sum_{\pi \in S_n}\prod_{1\leq i \leq n}\omega(x_i^*y_{\pi(i)})			
		\end{equation}
	\end{corollary}
	\begin{proof}
		The proof is a direct consequence of Equation \ref{equation:emptysetInnerProduct}
	\end{proof}
	The inner product in Corollary \ref{corollary:meanzero} has the same form as the familiar quasi-free state on CCR algebra. This suggests a connection between $L_2(\mathbb{P}_\omega N, \phi_\omega)$ and the symmetric Fock space $\mathcal{F}(L_2(N, \omega))$. Indeed we will see later that these two Hilbert spaces are isomorphic. However, to state the result, we need yet another basis for $L_2(\mathbb{P}_\omega N, \phi_\omega)$ because it is difficult to study the grading structure using $\lambda_\emptyset$-basis. Before we construct the Fock basis, we use $\lambda_\emptyset$-basis to construct the lift of a conditional expectation. This construction requires several functorial properties of Poissonization, which we will study in the next subsection.
	\subsection{Functorial Properties of Poissonization}
	In this subsection, we collect some functorial properties of Poissonization. These results will be used later to construct the isometric isomorphism between $L_2(\mathbb{P}_\omega N, \phi_\omega)$ and the symmetric Fock space $\mathcal{F}(L_2(N,\omega))$.
	\begin{lemma}\label{lemma:predualstructure}
		Let $d_\phi \in L_1(\mathbb{P}_\omega N, \phi_\omega)$ be the density corresponding to the state $\phi_\omega\in \mathbb{P}_\omega N_*$. The subspace $\text{span}\{\lambda_\emptyset(x_1,..,x_n)d_\phi\}$ is dense in $L_1(\mathbb{P}_\omega N, \phi_\omega)$.
	\end{lemma}
	\begin{proof}
		By Lemma \ref{lemma:emptybasis}, the subspace $K = \text{span}\{\lambda_\emptyset(x_1,...,x_n)\xi_\omega\}$ is dense in $L_2(\mathbb{P}_\omega N, \phi_\omega)$. Since $L_2(\mathbb{P}_\omega N, \phi_\omega)$ is dense in $L_1(\mathbb{P}_\omega N, \phi_\omega)$, $\text{span}\{\lambda_\emptyset(x_1,...,x_n)d_\phi\}$ is dense in the predual.
	\end{proof}
	\begin{lemma}\label{lemma:ultraweakdense}
		The unital $*$-subalgebra: $\mathcal{A}:=\{\Gamma(a) = \sum_{n\geq 0}a^{\otimes n}: ||a|| \leq 1\}$ is ultra-weakly dense in $\mathbb{P}_\omega N$.
	\end{lemma}
	\begin{proof}
		By Proposition \ref{proposition:symmetrictensor}, it suffices to prove that the ultra-weak closure of $\mathcal{A}$ contains $\{x^{\otimes n}: n\in \mathbb{N}\}$. For all $a\in N$, the element $\Gamma(ta) = \sum_{n\geq 0}t^n a^{\otimes n}$ is well-defined in $\mathbb{P}_\omega N\cong M_s(N)$ for all $|t| < \frac{1}{||a||}$. $\lim_{t\rightarrow 0^+}\Gamma(ta)$ converges in the ultra-weak topology and is equal to the vacuum vector. Hence the vacuum vector is in the ultra-weak closure of $\mathcal{A}$. By a simple induction and the formula:
		\begin{equation*}
			\lim_{t\rightarrow 0^+}\frac{1}{t^{k+1}}(\Gamma(ta) - \sum_{1\leq j \leq k}t^ja^{\otimes j}) = a^{\otimes (k+1)}
		\end{equation*}
		we have that $a^{\otimes n}\in \overline{\mathcal{A}}^{\text{ultra-weak}}$. Therefore the claim follows from Proposition \ref{proposition:symmetrictensor}.
	\end{proof}
	\begin{proposition}\label{proposition:contractionLift}
		Let $(N,\omega_N)$ and $(M,\omega_M)$ be two von Neumann algebras with fixed normal faithful positive functionals. And let $T:N\rightarrow M$ be a normal unital contraction such that $\omega_M\circ T = \omega_N$. Assume there exists a normal unital linear map:
		\begin{equation*}
			\hat{T}: M\rightarrow N
		\end{equation*}
		such that $(T,\hat{T})$ forms a dual pair: $\omega_M(T(x)y) = \omega_N(x\hat{T}(y))$. Then there exists a normal contraction:
		\begin{equation}
			\Gamma(T):\mathbb{P}_\omega N\rightarrow \mathbb{P}_\omega M: \Gamma(a)\mapsto \Gamma(Ta)
		\end{equation}
		such that $\phi_{\omega_M}\circ\Gamma(T) = \phi_{\omega_N}$ and the following formula holds in the predual:
		\begin{equation}\label{equation:contractionEmptyBasis}
			\Gamma(T)^*\lambda_\emptyset(x_1,...,x_n)d_{\phi_M} = \lambda_\emptyset(\hat{T}x_1,...,\hat{T}x_n)d_{\phi_N}
		\end{equation}
		where $\Gamma(T)^*$ is the dual of $\Gamma(T)$ and $d_{\phi_M}$ is the density associated with thet state $\phi_{\omega_M}$. Here we identify the predual with the Haagerup $L_1$-space: $L_1(\mathbb{P}_{\omega_N}N,\phi_{\omega_N}) \cong \mathbb{P}_{\omega_N}N_*$.
	\end{proposition}	
	\begin{proof}
		For all $s < 1$, we have: $||\Gamma(sTa)|| = \frac{1}{1 - ||sTa||} \leq \frac{1}{1 - ||sa||} = ||\Gamma(sa)||$ where the inequality holds because $T$ is a contraction. Therefore $\Gamma(T)$ is well-defined and is a contraction on the ultra-weakly dense subalgebra $\mathcal{A}=\{\Gamma(a):||a||\leq 1\}$. To see $\Gamma(T)$ is normal, we consider its dual $\Gamma(T)^*$ restricted to the predual $\mathbb{P}_{\omega_M}M_*$. Using the dual map $\hat{T}$, we have:
		\begin{align}\label{equation:keyidentity}
			\begin{split}
				\langle\Gamma(T)^*\lambda_\emptyset(x_1,...,x_m)d_{\phi_M}, \Gamma(y)\rangle & = \langle\lambda_\emptyset(x_1,...,x_m)d_{\phi_M}, \Gamma(Ty)\rangle\\
				&=\exp(\omega_M(Ty - 1))\prod_{1\leq i \leq m}\omega_M((Ty)x_i) \\&= \exp(\omega_N(y - 1))\prod_{1\leq i \leq m}\omega_N(y\hat{T}(x_i))
				\\&=\langle\lambda_\emptyset(\hat{T}x_1,...,\hat{T}x_m)d_{\phi_N},\Gamma(y)\rangle
			\end{split}
		\end{align} 
		where $x_i\in M, y\in N$ are all contractions and $\langle,\rangle$ denotes the canonical pairing between the predual and the algebra. Since the algebra generated by contraction is ultra-weakly dense in the Poisson algebra, Equation \ref{equation:contractionEmptyBasis} holds. Since the dual map $\Gamma(T)^*$ defines a contraction on the predual, $\Gamma(T)$ is a normal contraction. 
		
		In addition, on the ultra-weakly dense subalgebra generated by contractions, we have:
		\begin{equation*}
			\phi_{\omega_M}(\Gamma(T)\Gamma(x)) = \exp(\omega_M(Tx - 1)) = \exp(\omega_N(x - 1)) = \phi_{\omega_N}(\Gamma(x))\pl.
		\end{equation*}
		Hence $\phi_{\omega_M}\circ\Gamma(T) = \phi_{\omega_N}$.  
	\end{proof}
	\begin{corollary}\label{corollary:homomorphismLift}
		Using the same notation as Lemma \ref{proposition:contractionLift}, let $\pi:N \rightarrow M$ be normal $*$-homomorphism that preserves the faithful functionals: $\omega_M\circ\pi = \omega_N$. In addition, assume the subalgebra $\pi(N)$ is invariant under the modular automorphism group $\{\sigma^M_t\}$. Then there exists a normal $*$-homomorphism:
		\begin{equation}
			\Gamma(\pi):\mathbb{P}_\omega N\rightarrow \mathbb{P}_\omega M: \Gamma(a)\mapsto \Gamma(\pi(a))
		\end{equation}
		such that $\phi_{\omega_M}\circ\Gamma(\pi) = \phi_{\omega_N}$.
	\end{corollary}
	\begin{proof}
		If $\exists x\neq 0 \in N$ such that $\pi(x) = 0$, then $0 = \omega_M\circ\pi(x) = \omega_N(x)$. Since $\omega_N$ and $\omega_M$ are both faithful, $\pi$ is necessarily faithful. Hence we can identify $N$ with its image in $M$ and regard $N$ as a subalgebra in $M$. Since $\pi(N)$ is invariant under the modular automorphism and $\omega_M$ restricted to $\pi(N)$ equals to the normal functional $\omega_N$, by Takesaki's theorem \cite{T} there exists a normal conditional expectation:
		\begin{equation*}
			E:M\rightarrow \pi(N)
		\end{equation*}
		such that $\omega_N\circ E = \omega_M$. In addition, $\omega_N(x\pi^{-1}E(y)) = \omega_M(\pi(x)E(y)) = \omega_M(E(\pi(x)y)) = \omega_M(\pi(x)y)$. Therefore by Proposition \ref{proposition:contractionLift}, $\Gamma(\pi)$ is a well-defined normal contraction that preserves the Poisson states. 
		
		Since the unital $*$-algebra $\mathcal{A} = \{\Gamma(x):||x||\leq 1\}$ is ultra-weakly dense in the Poisson algebra (c.f. Lemma \ref{lemma:ultraweakdense}), we only need to show that $\Gamma(\pi)$ is a well-defined bounded $*$-homomorphism on $\mathcal{A}$. Since $\Gamma(1) = 1$ in the Poisson algebra, we have $\Gamma(\pi)\Gamma(1) = \Gamma(\pi(1)) = \Gamma(1)$. In addition, $\Gamma(\pi)$ is multiplicative: $\Gamma(\pi)\big(\Gamma(x)\Gamma(y)\big) = \Gamma(\pi)\Gamma(xy) = \Gamma(\pi(xy)) = \Gamma(\pi(x))\Gamma(\pi(y))$. It is clear that $\Gamma(\pi)$ is $*$-preserving: $\Gamma(\pi)\Gamma(x^*) = \Gamma(\pi(x^*)) = \Gamma(\pi(x))^*$. Finally, since $\pi$ is injective, it preserves the norm. Hence for all $t < 1$, we have: $||\Gamma(\pi(tx))|| = \frac{1}{1 - ||\pi(tx)||} = \frac{1}{1 - ||tx||} = ||\Gamma(tx)||$. Therefore by density and linearity, $\Gamma(\pi)$ on the entire Poisson algebra is a normal $*$-homomorphism. 
	\end{proof}
	\begin{corollary}\label{corollary:modularautoLift}
		Using the same notation as Corollary \ref{corollary:homomorphismLift}, the modular automorphism group $\{\sigma^\omega_t\}$ admits a strongly continuous lift: $\Gamma(\sigma^\omega_t):\mathbb{P}_\omega N\rightarrow \mathbb{P}_\omega N$ that preserves the Poisson state. In addition, the lift is the modular automorphism of the Poisson state:
		\begin{equation}
			\Gamma(\sigma^\omega_t) =\sigma^{\phi_\omega}_t\pl.
		\end{equation}
	\end{corollary}
	\begin{proof}
		For each $t\in \mathbb{R}$, the modular automorphism $\sigma^\omega_t$ satisfies the assumptions of Corollary \ref{corollary:homomorphismLift}. Thus $\Gamma(\sigma^\omega_t)$ is a well-defined normal $*$-homomorphism that preserves the Poisson state. In addition, we can check the KMS-condition holds on the ultra-weakly dense subalgebra $\mathcal{A}_{ana} = \{\Gamma(a):||a||\leq 1, a \in N^\omega_{ana}\} = \mathcal{A}\cap \big(\mathbb{P}_\omega(N)\big)^{\phi_\omega}_{ana}$:
		\begin{align}\label{equation:KMSonPoisson}
			\begin{split}
				\phi_\omega\bigg(\Gamma(\sigma^\omega_{t+i})\big(\Gamma(x)\big)\Gamma(y)\bigg) &= \phi_\omega\bigg(\Gamma\big(\sigma^\omega_{t+i}(x)\big)\Gamma(y)\bigg) = \exp(\omega(\sigma^\omega_{t+i}(x)y - 1)) \\&=\exp(\omega(y\sigma^\omega_t(x) - 1)) = \phi_\omega\bigg(\Gamma(y)\Gamma(\sigma^\omega_t)\big(\Gamma(x)\big)\bigg)\pl.
			\end{split}
		\end{align}
		Then by the uniqueness of the one-parameter group that satisfies the KMS condition \cite{T}, $\Gamma(\sigma^\omega_t) = \sigma^{\phi_\omega}_t$. In particular, the one-parameter group $\{\Gamma(\sigma^\omega_t)\}$ is strongly continuous. 
	\end{proof}
	\begin{corollary}\label{corollary:conditionalExpLift}
		Let $N'\ssubset N$ be a von Neumann subalgebra such that the restricted normal positive functional $\omega|_{N'}$ is semifinite. In addition, assume $N'$ is invariant under the modular automorphism $\sigma_t^\omega(N') = N'$. Then $\mathbb{P}_\omega N'\ssubset\mathbb{P}_\omega N$ and there exists a normal conditional expectation:
		\begin{equation}
			\mathcal{E}:\mathbb{P}_\omega N\rightarrow \mathbb{P}_\omega N'
		\end{equation}
		such that: $\phi_\omega\circ\mathcal{E} = \phi_\omega$. Moreover, $\mathcal{E} = \Gamma(E)$ where $E:N\rightarrow N'$ is the conditional expectation and $\Gamma(E)$ is the lift constructed in Proposition \ref{proposition:contractionLift}.
	\end{corollary}
	\begin{proof}
		By Corollary \ref{corollary:modularautoLift}, the Poisson algebra $\mathbb{P}_\omega N'$ is preserved under the modular automorphism of the Poisson state:
		\begin{equation*}
			\sigma_t^{\phi_\omega}\big(\mathbb{P}_\omega N'\big) = \Gamma(\sigma^\omega_t)\mathbb{P}_\omega N' = \mathbb{P}_\omega \sigma^\omega_t(N') = \mathbb{P}_\omega N'\pl.
		\end{equation*}	
		Since the Poisson state restricted to the subalgebra $\mathbb{P}_\omega N'$ is semifinite, by Takesaki's theorem \cite{T} there exists a normal conditional expectation:
		\begin{equation*}
			\mathcal{E}:\mathbb{P}_\omega N\rightarrow\mathbb{P}_\omega N'
		\end{equation*}
		that preserves the Poisson state. To see $\mathcal{E} = \Gamma(E)$, we again consider the ultra-weakly dense subalgebra $\mathcal{A}=\{\Gamma(a):||a||\leq 1\}$ and check that $\Gamma(E)$ gives the correct projection on the Hilbert space $L_2(\mathbb{P}_\omega N, \phi_\omega)$. For contractions $x\in N'$ and $y\in N$, we have:
		\begin{align*}
			\begin{split}
				\langle\Gamma(x)d^{1/2}_{\phi_\omega}, \Gamma(y)d^{1/2}_{\phi_\omega}\rangle &= \phi_\omega(\Gamma(x^*)\mathcal{E}\Gamma(y)) = \phi_\omega(\Gamma(x^*)\Gamma(y)) =\exp(\omega(x^*y -1)) \\&= \exp(\omega(x^*E(y) - 1))=\phi_\omega(\Gamma(x^*)\Gamma(Ey)) =\langle\Gamma(x)d_{\phi_\omega}^{1/2}, \Gamma(E)\Gamma(y)d_{\phi_\omega}^{1/2}\rangle\pl.
			\end{split}
		\end{align*}
		Since $\mathcal{A}$ is a subalgebra in the Poisson algebra and it is ultra-weakly dense, $\text{span}\{\Gamma(x)d_{\phi_\omega}^{1/2}: \Gamma(x)\in\mathcal{A}\}$ is dense in the Hilbert space $L_2(\mathbb{P}_\omega N, \phi_\omega)$. Hence $\Gamma(E)$ implements the orthogonal projection induced by the conditional expectation $\mathcal{E}$. Therefore by density and continuity, we have: $\mathcal{E} = \Gamma(E)$ on the Poisson algebra.
	\end{proof}
	\begin{corollary}\label{corollary:conditionalExpCalc}
		Using the same notation as Corollary \ref{corollary:conditionalExpLift}, the lifted conditional expectation induces an orthogonal projection on the Hilbert space $L_2(\mathbb{P}_\omega N, \phi_\omega)$ and its action on the $\lambda_\emptyset$-basis is given by:
		\begin{equation}\label{equation:conditionalExpCalc}
			\mathcal{E}\lambda_\emptyset(x_1,...,x_n)d_{\phi_\omega}^{1/2} = \lambda_\emptyset(Ex_1,...,Ex_n)d_{\phi_\omega}^{1/2}\pl.
		\end{equation}
	\end{corollary}
	\begin{proof}
		This is a direct consequence of Proposition \ref{proposition:contractionLift} and Corollary \ref{corollary:conditionalExpLift}. The conditional expectation $E:N\rightarrow N'$ is dual to the inclusion $N'\ssubset N$ in the sense that for all $x\in N'$ and $y\in N$:
		\begin{equation*}
			\omega(x(Ey)) = \omega(E(xy)) =\omega(xy)\pl.
		\end{equation*}
		Therefore the $L_2$-version of the calculation in the proof of Proposition \ref{proposition:contractionLift} implies the following:
		\begin{align*}
			\begin{split}
				\langle&\mathcal{E}\lambda_\emptyset(x_1,...,x_n)d_{\phi_\omega}^{1/2}, \Gamma(y)d_{\phi_\omega}^{1/2}\rangle = \langle\Gamma(E)\lambda_\emptyset(x_1,...,x_n)d_{\phi_\omega}^{1/2}, \Gamma(y)d_{\phi_\omega}^{1/2}\rangle \\&= \langle\lambda_\emptyset(x_1,...,x_n)d_{\phi_\omega}^{1/2}, \Gamma(y)d_{\phi_\omega}^{1/2}\rangle
				=\exp(\omega(y - 1))\prod_{1\leq i \leq m}\omega(yx_i) \\
				&=\exp(\omega(y -1))\prod_{1\leq i \leq m}\omega(y(Ex_i))=\langle\lambda_\emptyset(Ex_1,...,Ex_n)d_{\phi_\omega}^{1/2}, \Gamma(y)d_{\phi_\omega}^{1/2}\rangle
			\end{split}
		\end{align*}
		where $y\in N'$ is a contraction. Then by density, we have the desired equation.
	\end{proof}
	In addition, we also need to consider lifting a normal projection onto a corner to the Poisson algebra. The result is given by the following corollary:
	\begin{corollary}\label{corollary:corner}
		Let $e\in N_\omega$ be a projection in the centralizer of $\omega$ (i.e. $\sigma^\omega_t(e) = e$) and let $N_e := eNe + \mathbb{C}(1 -e)$ be the subalgebra in $N$ generated by the corner $eNe$ and the projection $1 - e$. Then the conditional expectation $E_e:N\rightarrow N_e$ can be lifted to a normal conditional expectation on the Poisson algebra: $\Gamma(E_e):\mathbb{P}_\omega N\rightarrow \mathbb{P}_\omega N_e$. 
		
		Moreover, there exists an isometric embedding:
		\begin{equation}
			\Gamma(\iota_e):L_2(\mathbb{P}_{\omega_e}(eNe), \phi_{\omega_e}) \rightarrow L_2(\mathbb{P}_\omega N, \phi_\omega)\pl.
		\end{equation}
		And the conditional expectation $\Gamma(E_e)$ induces an orthogonal projection that projects onto the subspace: $\Gamma(\iota_e)\big(L_2(\mathbb{P}_{\omega_e}(eNe), \phi_{\omega_e})\big)$. In terms of the $\lambda_\emptyset$-basis, we have the following explicit description $\Gamma(E_e)$ projection:
		\begin{equation}
			\Gamma(E_e)\lambda_\emptyset(x_1,...,x_n)d^{1/2}_{\phi_\omega} = \sum_{A\subset [n]}\prod_{i\notin A}\omega((1-e)x_i)\lambda_\emptyset(ex_je: j\in A)d^{1/2}_{\phi_{\omega_e}}\pl.
		\end{equation}
		Here the affiliated operator $\lambda_\emptyset(ex_je: j\in A)$ is generated by $x_j$'s where $j\in A$.
	\end{corollary}
	\begin{proof}
		Since $e$ is in the centralizer, the subalgebra $N_e$ is invariant under the modular automorphism. Therefore there exists a normal conditional expectation: $E_e:N\rightarrow N_e$. 
		
		The inclusion $N_e\ssubset N$ induces a normal $*$-homomorphism:
		\begin{equation}
			\Gamma(\iota_e):\mathbb{P}_{\omega_e}N_e \rightarrow \mathbb{P}_\omega N: \Gamma(a)\mapsto \Gamma(a + (1-e))
		\end{equation}
		where $\omega_e$ is the restriction of $\omega$ onto the corner $N_e$. In particular, $\Gamma(\iota_e)$ preserves the Poisson state and thus induces an isometric embedding of the Haagerup $L_2$-spaces:
		\begin{equation}
			\phi_\omega\big( \Gamma(\iota_e)\Gamma(a)\big) = \exp(\omega(a + (1- e) -1)) = \exp(\omega(a - e)) = \exp(\omega_e(a - 1)) = \phi_{\omega_e}(\Gamma(a))\pl.
		\end{equation} 
		Moreover for any contraction $y\in eNe$ and any contraction $x\in N$, we have:
		\begin{align}
			\begin{split}
				\langle&\Gamma(x)d_{\phi_\omega}^{1/2}, \Gamma(\iota_e)\Gamma(y)d_{\phi_{\omega_e}}^{1/2}\rangle = \phi_\omega(\Gamma(x^*(y + 1- e))) = \exp(\omega(x^*(y + 1 -e)))
				\\& = \exp(\omega(x^*(1-e)))\phi_{\omega_e}(\Gamma(ex^*e)\Gamma(y)) = \exp(\omega(x^*(1-e)))\langle\Gamma(exe)d_{\phi_{\omega_e}}^{1/2}, \Gamma(y)d_{\phi_{\omega_e}}^{1/2}\rangle
			\end{split}
		\end{align}
		Combining these observations, for any contraction $a$ in $eNe$ we have:
		\begin{align}
			\begin{split}
				\langle&\Gamma(E_e)\lambda_\emptyset(x_1,...,x_n)d_{\phi_\omega}^{1/2}, \Gamma(a)d_{\phi_{\omega_e}}^{1/2}\rangle =\langle\lambda_\emptyset(x_1,...,x_n)d_{\phi_\omega}^{1/2}, \Gamma(a + 1 -e)d_{\phi_\omega}^{1/2}\rangle
				\\
				&=\exp(\omega(a - e))\prod_{1\leq i \leq n}\omega((a + 1 - e)x_i) = \exp(\omega(a - e))\sum_{A\subset [n]}\prod_{i\in A}\omega(ax_i)\prod_{i\notin A}\omega((1-e)x_i)
				\\
				&=\sum_{A\subset [n]}\prod_{i\notin A}\omega((1-e)x_i)\bigg(\exp(\omega_e(a - 1))\prod_{i\in A}\omega_e(a(ex_ie))\bigg) \\&= \sum_{A\subset [n]}\prod_{i\in A}\omega((1-e)x_i)\langle\lambda_\emptyset(ex_ie: i\notin A)d_{\phi_{\omega_e}}^{1/2}, \Gamma(a)d_{\phi_{\omega_e}}^{1/2}\rangle\pl.
			\end{split}
		\end{align}
		Therefore by density of the subalgebra $\{\Gamma(a):||a||\leq 1\}$, we have the desired formula.
	\end{proof}
	\subsection{Standard Form of Poisson Algebra and Symmetric Fock Space}
	After the preparations, we are now ready to construct the Fock basis and prove that $L_2(\mathbb{P}_\omega N, \phi_\omega)$ is isomorphic to the symmetric Fock space $\mathcal{F}(L_2(N, \omega))$. For this purpose, we consider the embedding:
	\begin{equation}
		\iota: N\rightarrow N\bigoplus N\ssubset N\overline{\otimes}M_2(\mathbb{C}): x\mapsto x\otimes e_{11}
	\end{equation}
	where $e_{11}$ denotes the matrix unit at the $(1,1)$-th entry. On $N\bigoplus N$, we consider the diagonal state: $\omega_2:=\omega\otimes tr_2 = \omega\oplus\omega$ where $tr_2$ denotes the canonical trace on the $2\times 2$ matrix algebra. Then $\iota$ preserves the state: $\omega_2\circ\iota = \omega$. We can factorize $\iota(x)$ into a mean zero component and the remainder: $\iota(x) = x\otimes e_{11} = x\otimes \epsilon + x\otimes e_{22}$ where $\epsilon:=e_{11} - e_{22}\in M_2(\mathbb{C})$. Let $E:N\bigoplus N\rightarrow N: (x,y)\mapsto x$ be the normal projection onto the corner. The embedding $\iota$ induces an isometric embedding:
	\begin{equation}
		\iota:L_2(\mathbb{P}_\omega N, \phi_\omega)\rightarrow L_2(\mathbb{P}_{\omega_2}(N\oplus N), \phi_{\omega_2}): \lambda_\emptyset(x_1,...,x_n)d_{\phi_\omega}^{1/2}\mapsto \lambda_\emptyset(\iota(x_1),...,\iota(x_n))d_{\phi_{\omega_2}}^{1/2}\pl.
	\end{equation}
	To see that this is an isometry, we apply Lemma \ref{lemma:emptybasis} and the fact that $\omega_2(\iota(x_i^*)\iota(x_j)) = \omega(x_i^*x_j), \text{ }\omega_2(\iota(x)) = \omega(x)$. Then by multi-linearity, we have:
	\begin{equation}
		\lambda_\emptyset(\iota(x_1),...\iota(x_m))d_{\phi_{\omega_2}}^{1/2} = \sum_{A\subset [n]}\lambda_{\emptyset}\big(\{x_i\otimes \epsilon: i\in A\}, \{x_j \otimes e_{22}:j\notin A\}\big)d_{\phi_{\omega_2}}^{1/2}\pl.
	\end{equation}
	where the affiliated opeartor $\lambda_\emptyset\big(\{x_i\otimes \epsilon: i\in A\}, \{x_j \otimes e_{22}:j\notin A\}\big)$ is generated by $x_i\otimes \epsilon$ for $i\in A$ and $x_j\otimes e_{22}$ for $j\notin A$. 
	\begin{definition}\label{definition:FockBasis}
		The Fock basis is the set of vectors:
		\begin{equation}
			\lambda_{\emptyset\emptyset}(x_1,...,x_n)d_{\phi_\omega}^{1/2}:= \Gamma(E)\lambda_\emptyset(x_1\otimes\epsilon,...,x_n\otimes \epsilon)d_{\phi_{\omega_2}}^{1/2}\pl.
		\end{equation}
	\end{definition}
	We first give an alternative combinatorial formula for the Fock basis vectors:
	\begin{lemma}\label{lemma:FockTransform}
		By Corollary \ref{corollary:corner}, the following identity holds:
		\begin{equation}\label{equation:FockTransform}
			\lambda_{\emptyset\emptyset}(x_1,...,x_n)d_{\phi_\omega}^{1/2} = \sum_{0\leq i \leq n}(-1)^i\sum_{\substack{A\subset [n] \\ |A| = i}}\prod_{j\in A}\omega(x_j) \lambda_\emptyset(\{x_i : i\notin A\})d_{\phi_\omega}^{1/2}\pl.
		\end{equation}
		In particular, $\lambda_{\emptyset\emptyset}(x_1,...,x_n)d_{\phi_\omega}^{1/2}\in \text{span}\{\lambda_\emptyset(y_1,...,y_m)d_{\phi_\omega}^{1/2}: m \leq n\}$. Moreover for $m < n$, we have:
		\begin{equation}
			\langle\lambda_{\emptyset\emptyset}(x_1,...,x_n)d_{\phi_\omega}^{1/2}, \lambda_\emptyset(y_1,...,y_m)d_{\phi_\omega}^{1/2}\rangle = 0\pl.
		\end{equation}
	\end{lemma}
	\begin{proof}
		This follows from a direct application of Corollary \ref{corollary:corner}. 
		\begin{align*}
			\begin{split}
				\lambda_{\emptyset\emptyset}(x_1,...,x_n)d_{\phi_\omega}^{1/2} &= \Gamma(E)\lambda_\emptyset(x_1\otimes \epsilon,...,x_n\otimes\epsilon)d_{\phi_{\omega_2}}^{1/2}\\
				&= \sum_{A\subset [n]}\prod_{i\in A}\omega_2((1-e_{11})x_i\otimes \epsilon) \lambda_\emptyset(\{e_{11}x_j\otimes \epsilon: j\notin A\})d_{\phi_\omega}^{1/2}
				\\
				&= \sum_{0\leq i \leq n}\sum_{\substack{A\subset [n] \\ |A| = i}}\prod_{i\in A}\omega(-x_i)\lambda_\emptyset(\{x_i: i\notin A\})d_{\phi_\omega}^{1/2}\pl.
			\end{split}
		\end{align*}
		It is clear that the Fock basis $\lambda_{\emptyset\emptyset}(x_1,...,x_n)d_{\phi_\omega}^{1/2}$ is in the subspace spanned by $\lambda_\emptyset$-basis generated no more than $n$-elements. Finally, we use Lemma \ref{lemma:emptybasis} to calculate the inner product:
		\begin{align}
			\begin{split}
				\langle&\lambda_{\emptyset\emptyset}(x_1,...,x_n)d_{\phi_\omega}^{\frac{1}{2}}, \lambda_{\emptyset}(y_1,...,y_m)d_{\phi_\omega}^{\frac{1}{2}}\rangle \\&=\langle\lambda_\emptyset(x_1\otimes \epsilon,...,x_n\otimes\epsilon)d_{\phi_{\omega_2}}^{\frac{1}{2}}, \lambda_{\emptyset}(y_1\otimes e_{11},...,y_m\otimes e_{11})d_{\phi_{\omega_2}}^{\frac{1}{2}}\rangle\\
				&= \sum_{0\leq p \leq m}\sum_{\substack{A_n \sqcup B_n = [n] \\ A_m\sqcup B_m = [m] \\ |B_n| = |B_m| = p}}\prod_{i\in B_n, j\in B_m}\omega_2(x^*_iy_j \otimes e_{11})\prod_{i\in A_n}\omega_2(x^*_i\otimes \epsilon)\prod_{j\in A_m}\omega_2(y_j\otimes e_{11})\pl.
			\end{split}
		\end{align}
		Since $\omega_2(x\otimes \epsilon) = 0$, the summands on the right hand side is non-zero only if $|A_n| = 0$. However, since $ m < n$, $|A_n| > 0$ for all summands. Therefore the sum is zero.
	\end{proof}
	\begin{corollary}\label{corollary:Fock}
		Using the same notation as Lemma \ref{lemma:FockTransform}, the Fock basis satisfies the quasi-free relation:
		\begin{align}\label{equation:quasifree}
			\begin{split}
				\langle\lambda_{\emptyset\emptyset}(x_1,...,x_n)d_{\phi_\omega}^{1/2}, \lambda_{\emptyset\emptyset}(y_1,...,y_m)d_{\phi_\omega}^{1/2}\rangle&=\delta_{n,m}\sum_{\pi \in S_n}\prod_{1\leq i \leq n}\omega(x_i^*y_{\pi(i)})\pl.
			\end{split} 
		\end{align}
	\end{corollary}
	\begin{proof}
		Without loss of generality, assume $m\leq n$. Then we have:
		\begin{align*}
			\begin{split}
				\langle&\lambda_{\emptyset\emptyset}(x_1,...,x_n)d_{\phi_\omega}^{1/2}, \lambda_{\emptyset\emptyset}(y_1,...,y_m)d_{\phi_\omega}^{1/2}\rangle\\
				&=\sum_{0\leq \alpha\leq m}(-1)^\alpha\sum_{\substack{A\subset [m] \\ |A| = \alpha}}\prod_{i\in A}\omega(y_i)\langle\lambda_{\emptyset\emptyset}(x_1,...,x_n)d_{\phi_\omega}^{1/2}, \lambda_{\emptyset}(\{y_i: i\notin A\})d_{\phi_\omega}^{1/2}\rangle 
				\\
				& = \delta_{m,n}\langle\lambda_{\emptyset\emptyset}(x_1,...,x_n)d_{\phi_\omega}^{1/2}, \lambda_\emptyset(y_1,...,y_n)d_{\phi_\omega}^{1/2}\rangle
				\\
				&=\delta_{n,m}\langle\lambda_\emptyset(x_1\otimes\epsilon,...,x_m\otimes\epsilon)d_{\phi_{\omega_2}}^{1/2}, \lambda_\emptyset(y_1\otimes e_{11}, ...,y_n\otimes e_{11})d_{\phi_{\omega_2}}^{1/2}\rangle
				\\
				&=\delta_{n,m}\sum_{\pi\in S_n}\prod_{1\leq i\leq n}\omega(x^*_i y_{\pi(i)})
			\end{split}
		\end{align*}
		where the last equality follows from the fact that $x_i\otimes \epsilon$'s have zero mean and hence the only non-zero contribution comes from the complete pairing between $x_i$'s and $y_j$'s.
	\end{proof}
	\begin{corollary}\label{corollary:FockIso}
		Using the same notation as Lemma \ref{lemma:FockTransform}, there exists an isometric isomorphism:
		\begin{equation}
			\iota: L_2(\mathbb{P}_\omega N, \phi_\omega)\rightarrow \mathcal{F}(L_2(N,\omega)): \lambda_{\emptyset\emptyset}(x_1,...,x_n)d_{\phi_\omega}^{1/2}\mapsto \otimes^s_{1\leq i \leq n}x_i
		\end{equation}
		where $\otimes^s_{1\leq i \leq n}x_i$ is the symmetric tensor product of $x_i$'s.
	\end{corollary}
	\begin{proof}
		Using the combinatorial formula in Lemma \ref{lemma:FockTransform}, we have:
		\begin{equation}
			\text{span}\{\lambda_{\emptyset\emptyset}(x_1,...,x_n)d_{\phi_\omega}^{1/2}\} = \text{span}\{\lambda_{\emptyset}(y_1,...,y_m)d_{\phi_\omega}^{1/2}\}\pl.
		\end{equation}
		Since the $\lambda_\emptyset$-vectors span a dense subspace in the Haagerup $L_2$-space, the Fock basis also span a dense subspace. Hence the formula $\iota(\lambda_{\emptyset\emptyset}(x_1,...,x_n)d_{\phi_\omega}^{1.2}) = \otimes^s_{1\leq i \leq n}x_i$ gives a well-defined linear map on a dense subspace of $L_2(\mathbb{P}_\omega N, \phi_\omega)$. By Corollary \ref{corollary:Fock}, $\iota$ is an isometry. Hence it extends to an isometry to the entire $L_2(\mathbb{P}_\omega N, \phi_\omega)$. Since the symmetric tensors form a basis for the symmetric Fock space, it is clear that $\iota$ is an isomorphism.
	\end{proof}
	With the Fock space isomorphism, it is reasonable to ask whether $\lambda(x)$ acts on the Fock basis by the familiar creation/annhiliation action. However this is not the case.
	\begin{lemma}\label{lemma:noQuasiParticle}
		$\lambda(x)$ acts on the Fock basis in the following way:
		\begin{align}
			\begin{split}
				\lambda(x)&\lambda_{\emptyset\emptyset}(y_1,...,y_n)d_{\phi_\omega}^{1/2} = \lambda_{\emptyset\emptyset}(x,y_1,...,y_n)d_{\phi_\omega}^{1/2} + \sum_{1\leq i \leq n}\lambda_{\emptyset\emptyset}(y_1,...,xy_i,...,y_n)d_{\phi_\omega}^{1/2}  \\&- \omega(x)\lambda_{\emptyset\emptyset}(y_1,...,y_n)d_{\phi_\omega}^{1/2} - \sum_{1\leq i \leq n}\omega(xy_i)\lambda_{\emptyset\emptyset}(y_1,...,y_{i-1},y_{i+1},...,y_n)d_{\phi_\omega}^{1/2}
			\end{split} 
		\end{align}
	\end{lemma}
	\begin{proof}
		Using the same notation as Lemma \ref{lemma:FockTransform}, we have:\begin{align*}
			\begin{split}
				\lambda(x)&\lambda_{\emptyset\emptyset}(y_1,...,y_n)d_{\phi_\omega}^{1/2}\\&= \Gamma(E)\bigg(\lambda(x\otimes e_{11})\lambda_\emptyset(y_1\otimes \epsilon, ..., y_n\otimes\epsilon)d_{\phi_{\omega_2}}^{1/2}\bigg)
				\\
				&=\Gamma(E)\bigg(\lambda_\emptyset(x\otimes e_{11}, y_1\otimes \epsilon,...,y_n\otimes\epsilon)d_{\phi_{\omega_2}}^{1/2} + \sum_{1\leq i \leq n}\lambda_\emptyset(y_1\otimes\epsilon,...,xy_i\otimes e_{11},...,y_n\otimes\epsilon)d_{\phi_{\omega_2}}^{1/2}\bigg)
				\\
				&=\lambda_{\emptyset\emptyset}(x,y_1,...,y_n)d_{\phi_\omega}^{1/2} + \sum_{1\leq i \leq n}\lambda_{\emptyset\emptyset}(y_1,...,xy_i,...,y_n)d_{\phi_\omega}^{1/2}  - \omega(x)\lambda_{\emptyset\emptyset}(y_1,...,y_n)d_{\phi_\omega}^{1/2} \\&- \sum_{1\leq i \leq n}\omega(xy_i)\lambda_{\emptyset\emptyset}(y_1,...,y_{i-1},y_{i+1},...,y_n)d_{\phi_\omega}^{1/2}\pl.\qedhere
			\end{split}
		\end{align*}
	\end{proof}
	\begin{remark}
		If we consider the CCR algebra and its representation on a symmetric Fock space, the infinitesimal generators of the CCR algebra are the familiar creation / annhilation operators. In that case, the usual number operator is well-defined and generates a strongly continuous semi-group. A vector of the form: $\lambda(x_1)...\lambda(x_n)\xi_\omega$ is an eigenvector of the number operator and it can be interpreted as the state of a ensemble of $n$ quasi-particles. From the construction of the Fock basis and Lemma \ref{lemma:noQuasiParticle}, the usual number operator does not give the correct grading structure on $L_2(\mathbb{P}_\omega N, \phi_\omega)$ and the affiliated generators $\lambda(x)$'s cannot be interpreted as creation / annhilation operators. Therefore heuristically, there is no well-defined notion of "quasi-particles" in the Poisson algebra.
	\end{remark}
	\subsection{Poisson Algebra for Normal Faithful Semifinite Weight}
	In this section, we generalize the construction of Poisson algebra in the previous subsection to allow $\omega$ to be a normal faithful semifinite weight (n.s.f weight). This generalization is necessary for potential applications. For example, one of the applications of Poissonization is to construct local algebras of observables in the sense of algebraic quantum field theory (AQFT) \cite{haag}. There the modular automorphism group often comes from certain unitary representations of noncompact Lie groups and the corresponding functionals are typically unbounded. 
	
	We provide two constructions of Poisson algebra starting from a von Neumann algebra and a fixed n.s.f. weight. Both constructions utilize a spatial approach: Instead of constructing the Poisson algebra directly, we first construct the correct Hilbert space and then the Poisson algebra will be a subalgebra generated by the appropriate unitary operators. The resulting Poisson algebra will be in a standard form and the Poisson state will be represented by a "vacuum" vector in the Hilbert space. The inner product on this Hilbert space is defined using the moment formula and the uniform estimates in Lemma \ref{lemma:momentformulaPoisson} are crucial in the construction. The first construction relies on the fact that any n.s.f. weight can be written as a weak limit of normal positive functionals. Hence the Hilbert space of the Poisson algebra can be realized as a closed subspace of the ultraproduct of a net of Haagerup $L_2$-spaces. The second construction takes a more abstract algebraic approach and relies heavily on the combinatorial formulae derived in the previous subsection. The key observation here is that the combinatorial formulae make sense on a formal level. Hence it makes sense to use these formulae as definitions for the inner product and basis vectors. This approach is inspired by the field of quantum probability \cite{quantStoch}. After the constructions, we show that both constructions give isomorphic Poisson algebras.
	\subsubsection{Poisson Algebra for N.S.F. Weight as a Limit of States}
	We refer to \cite{T} for more information on n.s.f. weights. Recall that normal weights can be characterized as the supremum of dominated bounded positive functionals: $\omega(x)=\sup\{\varphi(x):\varphi \leq \omega\}$. The goal is to find a net of dominated normal positive functionals $\{\varphi_s\}_s$ such that the net weakly converges to $\omega$. Recall each normal faithful weight defines a left Hilbert algebra with a dense subalgebra given by the following left ideal:
	\begin{equation}
		n_\omega := \{x\in N: \omega(x^*x)< \infty\}\pl.
	\end{equation}
	In addition, recall the definition ideal of a weight is a two-sided ideal:
	\begin{equation}
		m_\omega := n_\omega^*n_\omega = \{\sum_{1\leq i \leq n}x_i^*y_i:x_i, y_i\in n_\omega\}\pl.
	\end{equation}
	The left ideal $n_\omega$ forms a dense subspace of the Haagerup $L_2$-space of $\omega$:
	\begin{equation}
		\iota: n_\omega\rightarrow L_2(N,\omega): x\mapsto \hat{x}
	\end{equation} 
	where the inner product on $\iota(n_\omega)$ is finite: $\langle\hat{x},\hat{x}\rangle = \omega(x^*x) < \infty$. In order for the Poissonization to be compatible with modular automorphism group, it is necessary to consider the analytic elements inside the definition ideal $m_\omega$. This is given by the well-known Tomita algebra \cite{T}. Recall that the Tomita algebra of $\omega$ is an ultra-weakly dense $*$-subalgebra of $N$ defined as:
	\begin{equation}
		\mathcal{U}^\omega_{ana} := \{x\in n_\omega: \hat{x}\in \cap_{n\geq 0}D(\Delta_\omega^n), \text{ }\forall n\in \mathbb{N}\}
	\end{equation}
	where $\Delta_\omega$ is the modular operator of $\omega$. If $x\in \mathcal{U}^\omega_{ana}$, the $x\in m_\omega$ and there exists constants $c,C>0$ such that:
	\begin{equation*}
		||\sigma_z^\omega(x)|| \leq C\exp(c|Im(z)|)\pl.
	\end{equation*}
	For later use, we define the following $L_1$-norm on $m_\omega$:
	\begin{equation}
		||x||_1 := \inf\{\sum_{1\leq i \leq n}\omega(x_i^*x_i)^{1/2}\omega(y_i^*y_i)^{1/2}: x = \sum_{1\leq i \leq n}x_i^*y_i\}\pl.
	\end{equation}
	It is clear that if $||x||_1 < \infty$ then $x\in m_\omega$ and $\omega(x) \leq \sum_{1\leq i\leq n}\omega(x_i^*x_i)^{1/2}\omega(y_i^*y_i)^{1/2} < \infty$. Heuristically, the $||\cdot||_1$-norm is the infimum of all possible $L_2$-factorization of an $L_1$-element. A uniform control on the moment formula is crucial to our construction (c.f. Lemma \ref{lemma:momentformulaPoisson}). This can be achieved by introducing the following norm on the Tomita algebra:
	\begin{equation}\label{equation:tripleNorm}
		|||x||| := \max\{||x||, \omega(x^*x)^{1/2}, \omega(xx^*)^{1/2}, ||x||_1\}\pl.
	\end{equation}
	With these notations and definitions, we are now ready to present the construction of Poissonization using an approximation of n.s.f weight by bounded positive functionals. 
	\begin{lemma}\label{lemma:limitWeight}
		Let $\omega$ be a n.s.f. weight, there exists a net of normal positive functionals $\{\varphi_s\}_s$ such that the net converges weakly to $\omega$: $\lim_s\varphi_s(x) = \omega(x)$. In addition, if $N$ has a separable predual, we can restrict to a sub-net where all $\varphi_s$'s are faithful. 
	\end{lemma}
	\begin{proof}
		Recall from \cite{T} that there exists an opposite weight $\omega'$ on the commutant $\pi_\omega(N)'$. And $\omega'$ is a n.s.f. weight such that the positive unit ball of its definition ideal gives all Radon-Nikodym derivatives of positive functionals that are subordinated to $\omega$ \cite{T}:
		\begin{equation}
			\{h_\varphi: \varphi(\cdot) = \omega(h_\varphi\cdot), \text{ }\varphi \leq \omega\} = \{x\in m_{\omega'}: ||x|| \leq 1, \text{ }x \geq 0\} \ssubset m_{\omega'}^+\pl.
		\end{equation}
		Since $m_{\omega'}$ is a hereditary ideal, we can fix a net of approximate unit for the open positive unit ball: $\{h_s\}_s\ssubset m_{\omega'}^+$. Then we have:
		\begin{equation*}
			\lim_s \varphi_s(x): = \lim_s \omega(h_s x) = \sup_s\varphi_s(x)  = \omega(x)\pl.
		\end{equation*}
		So far the construction does not guarantee that all positive functionals are faithful. However, if $N$ has a separable predual, there exists a countable subset $\{x_i\}_{i\in I}\ssubset m_\omega$ such that $\text{span}\{\omega(x_i\cdot)\in N_*\}$ is $L_1$-dense in the predual. In addition, the commutant $\pi_\omega(N)'$ also has a separable predual. Hence we can take the net $\{h_s\}_s\ssubset m_{\omega'}^+$ to be a sequence. In addition, for each $x_i$, since $\lim_s \varphi_s(x_i) = \omega(x_i)$, there exists a subsequence of functionals $\{\varphi_{s_i}\}_{s,i}$ that are faithful on $x_i$ and converge weakly to $\omega$. Then by a diagonal argument, we can replace $\{\varphi_s\}_s$ by a subsequence of faithful functionals that converges weakly to $\omega$.
	\end{proof}
	\begin{proposition}\label{proposition:PoissonForWeight1}
		Let $(N,\omega)$ be a von Neumann algebra with separable predual and a fixed normal faithful semifinite weight. Let $\mathcal{U}^\omega_{ana}$ be the Tomita algebra of $\omega$ with the norm $|||\cdot|||$ (c.f. Equation \ref{equation:tripleNorm}). Then the following statements are true:
		\begin{enumerate}
			\item There exists a Hilbert space $\mathcal{H}_\omega \ssubset \prod_s^w L_2(\mathbb{P}_{\varphi_s}N, \phi_s)$ where $w$ is a free ultrafilter on $\mathbb{N}$ and $\{\varphi_s\}_s$ is a sequence of normal faithful positive functionals weakly converging to $\omega$. $\mathbb{P}_{\varphi_s}N$ is the Poisson algebra induced by the functional $\varphi_s$ and $\phi_s$ is the corresponding Poisson state. The Hilbert space $H_\omega$ has basis vectors $\{\widetilde{\lambda}_\emptyset(x_1,...,x_n): x_i \in \mathcal{U}^\omega_{ana}\}$ where the basis vectors are represented by an ultraproduct of $\lambda_\emptyset$-basis vectors:
			\begin{equation}
				\widetilde{\lambda}_\emptyset(x_1,...,x_n) := (\lambda_\emptyset(x_1,...,x_n)d_s^{1/2})^\bullet
			\end{equation}
			where $d_s$ is the density of the Poisson state $\phi_s$.
			\item For $x\in \{x\in\mathcal{U}^\omega_{ana} : x^* = x\}$, the ultraproduct $(\Gamma(e^{ix}))^\bullet $ is a well-defined unitary operator on $H_\omega$:
			\begin{equation}
				(\Gamma(e^{ix}))^\bullet\widetilde{\lambda}_{\emptyset}(y_1,...,y_n) := (\Gamma_s(e^{ix})\lambda_\emptyset(y_1,...,y_n)d_s^{1/2})^\bullet\pl.
			\end{equation} 
			\item Define the Poisson algebra $\mathbb{P}_\omega N$ to be the von Neumann algebra generated by the unitaries: $\{(\Gamma(e^{ix}))^\bullet: x\in \mathcal{U}^\omega_{ana}, \text{ }x^* = x\}$ in $\mathbb{B}(H_\omega)$. The Poisson state $\phi_\omega$ is represented by the vector $(d_s^{1/2})^\bullet$ and $\phi_\omega$ is a normal faithful state. In addition, we have the isomorphism:
			\begin{equation}
				L_2(\mathbb{P}_\omega N, \phi_\omega) \cong H_\omega :\lambda_\emptyset(x_1,...,x_n)d_{\phi_\omega}^{1/2}\mapsto \widetilde{\lambda}_\emptyset(x_1,...x_n)
			\end{equation}
			where $d_{\phi_\omega}$ is the density of the Poisson state.
		\end{enumerate}
	\end{proposition}
	\begin{proof}
		By Lemma \ref{lemma:momentformulaPoisson}, the norms of the family of vectors $\lambda_\emptyset(x_1,...,x_n)d_s^{1/2}$ are uniformly bounded by:
		\begin{equation}
			||\lambda_\emptyset(x_1,...,x_n)d_s^{1/2}||^2_2 \leq C n^{2n}\prod_{1\leq i \leq n}|||x_i|||^2
		\end{equation}
		where $C>0$ is an absolute constant. Therefore $\widetilde{\lambda}_\emptyset(x_1,...,x_n)$ is well-defined in the ultraproduct $\prod^w_s L_2(\mathbb{P}_{\varphi_s} N,\varphi_s)$. Hence the Hilbert space $H_\omega$ is well-defined and is spanned by $\{\widetilde{\lambda}_\emptyset(x_1,...,x_n): x_i\in \mathcal{U}^\omega_{ana}\}$. 
		
		To see the action is well-defined, we use the following uniform estimates:
		\begin{align}\label{equation:uniformAction}
			\begin{split}
				||\Gamma(e^{itx})\lambda(y_1)...\lambda(y_n)d_s^{1/2}||_2 &\leq \sum_{k\geq 0}\frac{t^k}{k!}||\lambda(x)^k\lambda(y_1)...\lambda(y_n)d_s^{1/2}||_2
				\\
				&\leq \sum_{k\geq 0}\frac{C^{n+k}t^k(n+k)^{n+k}|||x|||^k}{k!}\prod_{1\leq i \leq n}|||y_i|||
				\\
				&\leq \sum_{k \geq 0}(1 + \frac{n}{k})^k C^n(Cet|||x|||)^k\prod_{1\leq i \leq n}|||y_i|||
			\end{split}
		\end{align}
		where the second inequality follows from Lemma \ref{lemma:momentformulaPoisson}. Since the constant $C$ is a universal constant, for $|t| < \frac{1}{C|||x|||e}$, the operator is uniformly bounded. Since $\lambda$-vectors are finite linear combinations of $\lambda_\emptyset$-vectors, then for $|t| < \frac{1}{C|||x|||e}$, $(\Gamma(e^{itx})\lambda(y_1)...\lambda(y_n)d_s^{1/2})^\cdot$ defines a vector in $H_\omega$. In addition, for $t \geq \frac{1}{C|||x|||e}$, we can decompose $t = k\frac{1}{2C|||x|||e} + \delta t$ where $k\in \mathbb{Z}$ and $|\delta t| < \frac{1}{C|||x|||e}$. Therefore by decomposing $\Gamma(e^{itx}) = \prod_{1\leq i \leq |k|}\Gamma(e^{i\sgn{k}\frac{1}{2C|||x|||e}}x)\Gamma(e^{i\delta t x})$, we see that for all $t\in \mathbb{R}$, $(\Gamma(e^{itx}))^\bullet$ is a well-defined unitary operator on $H_\omega$. Moreover, by the differentiation formula given in the proof of Proposition \ref{proposition:symmetrictensor}, $\overline{\text{span}\{(\Gamma(e^{ix_1})...\Gamma(e^{ix_n})d_s^{1/2})^\bullet\}} = \overline{\text{span}\{(\lambda_\emptyset(x_1,...,x_n)d_s^{1/2})^\bullet\}} = \overline{\text{span}\{(\lambda(x_1)...\lambda(x_n)d_s^{1/2})^\bullet\}}$ where the last equation follows from the fact that $\lambda_\emptyset$-vector is a finite linear combination of $\lambda$-vectors and vice versa.
		
		Since the Poisson states $\phi_s$ are normalized for all $s$, $||d_s^{1/2}|| = 1$	and $\xi_\omega:=(d_s^{1/2})^\bullet$ is a well-defined vector in $H_\omega$. We need to verify the state defined by: $\langle \xi_\omega, \cdot \xi_\omega\rangle$ is a normal faithful state and satisfies the moment formula of the Poisson state. It is clear that this state is normal and is normalized. In addition, by definition $L_2(\mathbb{P}_\omega N, \phi_\omega) = \overline{\text{span}\{(\Gamma(e^{ix_1})...\Gamma(e^{ix_n})d_s^{1/2})^\bullet\}}$. As a vector space, this is isomorphic to $H_\omega$. To see this is an isometry, we need to verify the moment formula for the vector state $\langle\xi_\omega,\cdot\xi_\omega\rangle$:
		\begin{align}
			\begin{split}
				||(\lambda(x_1)...\lambda(x_n)d_s^{1/2})^\bullet||^2_2 &= \lim_s ||\lambda(x_1)...\lambda(x_n)d_s^{1/2}||^2_2 
				\\
				&= \lim_s \sum_{\sigma\in\mathcal{P}(n)}\prod_{A\in\sigma}\varphi_s(\overrightarrow{\prod}_{i\in A}x_i)
				\\
				&=  \sum_{\sigma\in\mathcal{P}(n)}\prod_{A\in\sigma}\omega(\overrightarrow{\prod}_{i\in A}x_i)
			\end{split}
		\end{align}
		where the limit exists by weak convergence and the fact that $x_i$'s are in the Tomita algebra. Therefore we have the desired isometry. In particular, the Poisson state is faithful.
	\end{proof}
	\begin{corollary}\label{corollary:modularAutoWeight}
		Using the same notation as Proposition \ref{proposition:PoissonForWeight1}, the modular automorphism of the Poisson state $\phi_\omega$ satisfies:
		\begin{equation}
			\sigma^{\phi_\omega}_t((\Gamma(e^{ix}))^\bullet) = (\sigma_t^{\varphi_s}\Gamma(e^{ix}))^\bullet =  (\Gamma(\sigma^\omega_t(e^{ix})))^\bullet\pl.			
		\end{equation}
		In particular, the modular operator acts on $L_2(\mathbb{P}_\omega N, \phi_\omega)$ by:
		\begin{equation}
			\Delta_{\phi_\omega}^{it}\lambda_\emptyset(x_1,...,x_n)\xi_\omega = \lambda_\emptyset(\sigma^\omega_t(x_1),...,\sigma^\omega_t(x_n))\xi_\omega\pl.
		\end{equation}
	\end{corollary}
	\begin{proof}
		We define a one-parameter unitary group:
		\begin{equation}
			U_t: L_2(\mathbb{P}_\omega N, \phi_\omega) \rightarrow L_2(\mathbb{P}_\omega N, \phi_\omega): \lambda(x_1)...\lambda(x_n)\xi_\omega \mapsto \lambda(\sigma^\omega_t(x_1))...\lambda(\sigma^\omega_t(x_n))\xi_\omega\pl.
		\end{equation}
		Using the moment formula and the fact that modular automorphism preserves the weight, it is easy to see that $U_t$ is unitary:
		\begin{align}
			\begin{split}
				\langle U_t\lambda(x_1)...\lambda(x_n)\xi_\omega, U_t\lambda(x_1)...\lambda(x_n)\xi_\omega\rangle &= \sum_{\sigma\in\mathcal{P}(2n)}\sum_{\substack{A\in \sigma \\ B_A\sqcup B_A^* = A}}\omega(\overleftarrow{\prod}_{(i+n)\in B_A^*}\sigma^\omega_t(x_i^*)\overrightarrow{\prod}_{i\in B_A}\sigma^\omega_t(x_i))
				\\
				&=\sum_{\sigma\in\mathcal{P}(2n)}\sum_{\substack{A\in \sigma \\ B_A\sqcup B_A^* = A}}\omega(\sigma^\omega_t(\overleftarrow{\prod}_{(i+n)\in B_A^*}x_i^*\overrightarrow{\prod}_{i\in B_A}x_i))
				\\
				&=\langle\lambda(x_1)...\lambda(x_n)\xi_\omega, \lambda(x_1)...\lambda(x_n)\xi_\omega\rangle
			\end{split}
		\end{align}	
		where the overhead arrows indicate the order of multiplication. The unitary $U_t$ can also be represented as the ultraproduct: $U_t = (\Delta_s^{is})^\bullet$ where $\Delta_s$ is the modular operator of the Poisson state $\phi_s$. To see this, we calculate:
		\begin{align}
			\begin{split}
				\langle\lambda(x_1)...\lambda(x_n)\xi_\omega,& (\Delta_s^{it})^\bullet \lambda(y_1)...\lambda(y_m)\xi_\omega\rangle = \lim_s \langle\lambda(x_1)...\lambda(x_n)d_s^{1/2}, \Delta_s^{it}\lambda(y_1)...\lambda(y_m)d_s^{1/2}\rangle
				\\
				&=\lim_s\sum_{\sigma\in\mathcal{P}(m+n)}\sum_{\substack{A\in\sigma \\ B_A = A\cap [m]\\  B_A^* = A \cap B_A^c}}\varphi_s(\overleftarrow{\prod}_{(i+m)\in B_A^*}x_i^*\overrightarrow{\prod}_{i\in B_A}\sigma^{\varphi_s}_t(y_i))
				\\
				&=\lim_s\sum_{\sigma\in\mathcal{P}(n+m)}\sum_{\substack{A\in\sigma \\ B_A = A\cap [m]\\  B_A^* = A \cap B_A^c}}\varphi_s(\overleftarrow{\prod}_{(i+m)\in B_A^*}x_i^*\overrightarrow{\prod}_{i\in B_A}\sigma^\omega_t(y_i))
				\\
				& = \langle\lambda(x_1)...\lambda(x_n)\xi_\omega, \lambda(\sigma^\omega_t(y_1))...\lambda(\sigma^\omega_t(y_m))\xi_\omega\rangle\pl.
			\end{split}
		\end{align}
		where we used the fact that $\varphi_s(\cdot)= \omega(h_s\cdot)$ for some $h_s$ in the commutant $\pi_\omega(N)'$ and hence $\sigma^{\varphi_s}_t(\cdot) = h_s^{it/2}\sigma^\omega_t(\cdot) h_s^{-it/2}$. In particular since the Radon-Nikodym derivative $h_s$ is in the commutant, $\varphi_s(x^*\sigma^{\varphi_s}_t(y)) = \omega(h_sx^*h_s^{it/2}\sigma^\omega_t(y)h_s^{-it/2}) = \omega(h_sx^*\sigma^\omega_t(y)) = \varphi_s(x^*\sigma^\omega_t(y))$. 
		
		The adjoint action of this unitary group defines a strongly continuous one-parameter automorphism on $\mathbb{P}_\omega N$:
		\begin{equation}
			\alpha_t:\mathbb{P}_\omega N\rightarrow \mathbb{P}_\omega N: (\Gamma(e^{ix}))^\bullet\mapsto U_t^*(\Gamma(e^{ix}))^\bullet U_t = (\Delta_s^{it}\Gamma(e^{ix})\Delta_s^{-it})^\bullet = (\sigma^{\varphi_s}_t(\Gamma(e^{ix})))^\bullet\pl.
		\end{equation}
		To see $\{\alpha_t\}_t$ is the modular automorphism group of the Poisson state, we verify that $\alpha_t$ preserves $\phi_\omega$ and $\alpha_t$ satisfies the KMS condition. 
		
		To see that $\alpha_t$ preserves the Poisson state, we calculate:
		\begin{align}
			\begin{split}
				\phi_\omega\bigg(\alpha_t\big((\Gamma(e^{ix}))^\bullet\big)\bigg) &= \langle\xi_\omega, U_t^*(\Gamma(e^{ix}))^\bullet U_t\xi_\omega\rangle = \langle\xi_\omega, (\Gamma(e^{ix}))^\bullet\xi_\omega\rangle = \phi_\omega((\Gamma(e^{ix}))^\bullet)\pl.
			\end{split}
		\end{align}
		Since the Poisson algebra is generated by $\Gamma(e^{ix})$'s and $\alpha_t$ is an automorphism, the Poisson state is invariant under the $\alpha_t$-action.
		
		To check the KMS condition, we consider a slightly different generating set of the Poisson algebra. If $x - 1$ is in the Tomita algebra, then $\Gamma(e^{2\pi ix}) = \Gamma(e^{2 \pi i(x-1)})$ and hence $(\Gamma(e^{2\pi i x}))^\bullet$ is well-defined and forms a generating set of the Poisson algebra. We check the KMS condition on these generating unitaries:
		\begin{align}
			\begin{split}
				\phi_\omega\bigg((\Gamma(e^{iy}))^\bullet\alpha_t(\Gamma(e^{ix}))^\bullet\bigg) &= \lim_s \phi_s\bigg(\Gamma(e^{iy})\sigma^{\varphi_s}_t\Gamma(e^{ix})\bigg) = \lim_s\exp(\varphi_s(y\sigma^{\varphi_s}_t(x) - 1))
				\\
				&=\lim_s\exp(\varphi_s((y-1)(\sigma^{\varphi_s}_t(x) - 1)) + \varphi_s(y - 1) + \varphi_s(\sigma^{\varphi_s}_t(x) - 1))
				\\
				&=\lim_s\exp(\varphi_s((\sigma^{\varphi_s}_{t+i}(x) - 1)(y-1)) + \varphi_s(y - 1) + \varphi_s(\sigma^{\varphi_s}_{t+i}(x)- 1))
				\\
				&=\lim_s\exp(\varphi_s(\sigma^{\varphi_s}_{t+i}(x)y - 1)) = \phi_\omega\bigg(\alpha_{t+i}(\Gamma(e^{ix}))^\bullet(\Gamma(e^{iy}))^\bullet\bigg)
			\end{split}
		\end{align}
		where we have used the fact that $x \in \mathcal{U}^\omega_{ana}$ if and only if $x\in\mathcal{U}^{\varphi_s}_{ana}$ since $\varphi(\cdot) = \omega(h_s\cdot)$ and $h_s$ is in the commutant $\pi_\omega(N)'$. 
		
		Combining these calculations, $\alpha_t = \sigma^{\phi_\omega}_t$ by the uniqueness of modular automorphism group \cite{T}.
	\end{proof}
	In order to define the lift of unital completely positive maps, it is useful to record the following formula:
	\begin{corollary}\label{corollary:GammaAction}
		Consider the set of elements: $S := \{x \in N: ||a|| \leq 1, \text{ } a-1\in m_\omega\}$, the following equation holds for $y\in S$:
		\begin{align}
			\begin{split}
				\langle&\lambda_\emptyset(x_1,...,x_n)\xi_\omega, \Gamma(y)\lambda_\emptyset(z_1,...,z_m)\xi_\omega\rangle \\&= \exp(\omega(y-1))\sum_{0\leq p \leq \min\{m,n\}}\sum_{\substack{A_n \sqcup B_n = [n] \\ A_m\sqcup B_m = [m]\\ |B_n| = |B_m| = p}}\prod_{i\in B_n, j\in B_m}\omega(x^*_i yz_j)\prod_{i\in A_n}\omega(x^*_i)\prod_{j\in A_m}\omega(yz_j)\pl.
			\end{split}
		\end{align}
	\end{corollary}
	\begin{proof}
		By definition $\omega(y - 1) < \infty$, the remaining calculation is analogous to the proof of Lemma \ref{lemma:emptybasis}. 
	\end{proof}
	\subsubsection{Abstract Algebraic Construction of Poisson Algebra}
	In this subsection, we provide an alternative construction of Poissonization based on an abstract algebraic approach. The key observation is that the Hilbert space $L_2(\mathbb{P}_\omega N, \phi_\omega)$ can be defined using purely algebraic constructions and the moment formula. In particular, the positivity of the inner product follows from the positivity of certain combinatorial formulae on $N$. It does not depend on other analytic information. Moreover, we will use the same generators as the previous construction to generate the unitary group of the Poisson algebra. These unbounded generators must be essentially self-adjoint. This can be checked using the classical result that a symmetric operator is essentially self-adjoint if and only if $\pm i$ are not its eigenvalues. Again this condition can be checked purely algebraically using the moment formula. 
	
	The main object in the abstract construction is the Tomita algebra $\mathcal{U}^\omega_{ana}$. Since it is an associative $*$-algebra, it has a canonical Lie algebra structure. Recall given a Lie algebra $\mathfrak{g}$ we can always consider its universal enveloping algebra: $\mathbb{U}(\mathfrak{g}) = \bigoplus_{n\geq 0}\mathfrak{g}^{\otimes n} / I$ where $I$ is the two-sided ideal generated by $\{x\otimes y - y\otimes x - [x,y]:x,y\in \mathfrak{g}\}$. Moreover, there exists a Lie algebra homomorphism: $\lambda:\mathfrak{g}\rightarrow \mathbb{U}(\mathfrak{g})$. We apply this canonical construction to the Tomita algebra and consider the universal enveloping algebra $\mathbb{U}(\mathcal{U}^\omega_{ana})$. The Lie algebra homomorphism $\lambda:\mathcal{U}^\omega_{ana}\rightarrow \mathbb{U}(\mathcal{U}_{ana}^\omega)$ is the analog of the quantization map (c.f. Equation \ref{equation:quanatization}).
	
	Since the Tomita algebra is a $*$-algebra, we can also define an anti-linear map on the universal enveloping algebra $\mathbb{U}(\mathcal{U}^\omega_{ana})$:	\begin{align}
		\begin{split}
			&*:\mathbb{U}(\mathcal{U}^\omega_{ana})\rightarrow \mathbb{U}(\mathcal{U}^\omega_{ana}):
			(\frac{1}{n!}\sum_{\sigma\in S_n}
			\lambda(x_1)...\lambda(x_n))^* \mapsto \frac{1}{n!}\sum_{\sigma\in S_n}\lambda(x_n^*)...\lambda(x_1^*)\pl.
		\end{split}
	\end{align}
	Now we use the following combinatorial formulae to define abstract $\lambda_\emptyset$-basis and $\lambda_{\emptyset\emptyset}$-basis.
	\begin{definition}
		For $n \geq 2$, inductively define the following elements in $\mathbb{U}(\mathcal{U}^\omega_{ana})$:
		\begin{align}\label{equation:abstractEmptyBasis}
			\begin{split}
				&\lambda_\emptyset(x_1) := \lambda(x_1)
				\\&\lambda_\emptyset(x_1,...,x_n):=\lambda(x_1)\lambda_\emptyset(x_2,...,x_n) - \sum_{2\leq i \leq n}\lambda_\emptyset(x_2,...,x_1x_i,...,x_n)\pl.
			\end{split}
		\end{align}
		Using the $\lambda_\emptyset$-vectors, we define the $\lambda_{\emptyset\emptyset}$-vectors:
		\begin{align}\label{equation:abstractEmptyEmptyBasis}
			\begin{split}
				&\lambda_{\emptyset\emptyset}(x_1) :=\lambda_\emptyset(x_1) - \omega(x_1)\\
				&
				\lambda_{\emptyset\emptyset}(x_1,...,x_n) := \sum_{0\leq i \leq n}\sum_{\substack{A\subset[n]\\ |A| = i}}\prod_{i\in A}\omega(-x_i)\lambda_\emptyset(\{x_i:i\notin A\})
			\end{split}
		\end{align}
		where $\omega(x_1)$ is in $\mathbb{C}\ssubset\mathbb{U}(\mathcal{U}^\omega_{ana})$.
	\end{definition}
	Both $\lambda_\emptyset$-basis and $\lambda_{\emptyset\emptyset}$-basis are well-defined in $\mathbb{U}(\mathcal{U}^\omega_{ana})$. Using the abstract $\lambda_{\emptyset\emptyset}$-vectors, we can define the inner product:
	\begin{definition}
		$\langle\cdot,\cdot\rangle:\mathbb{U}(\mathcal{U}^\omega_{ana})\times\mathbb{U}(\mathcal{U}^\omega_{ana})\rightarrow \mathbb{C}$ is a sesquilinear form such that:
		\begin{equation}\label{equation:abstractQuasiFree}
			\langle\lambda_{\emptyset\emptyset}(x_1,...,x_n),\lambda_{\emptyset\emptyset}(y_1,...,y_m)\rangle = \delta_{n,m}\sum_{\pi\in S_n}\prod_{1\leq i \leq n}\omega(x^*_i y_{\pi(i)})\pl.
		\end{equation}
	\end{definition}
	\begin{lemma}\label{lemma:abstractHilbert}
		The sesquilinear form is positive and non-degenerate. And the map:
		\begin{equation}
			\iota: \mathbb{U}(\mathcal{U}^\omega_{ana})\rightarrow\mathcal{F}(L_2(N,\omega)):\lambda_{\emptyset\emptyset}(x_1,...,x_n)\mapsto \otimes_{1\leq i \leq n}\widehat{x_i}
		\end{equation}
		is an isometric embedding and the image is dense in the Fock space. Here $\widehat{x}\in L_2(N,\omega)$ is the usual embedding: $\widehat{\cdot}:n_\omega\rightarrow L_2(N,\omega): x\mapsto \hat{x}$ where $n_\omega$ is the left ideal of the right bounded vectors of the weight $\omega$.
	\end{lemma}
	\begin{proof}
		By Equation \ref{equation:abstractQuasiFree}, it is clear that the linear map is an isometry. In particular, the basis of the Fock space $\{\otimes_{1\leq i \leq n}\widehat{x_i}: n\in\mathbb{N}\}$ is in the image of $\iota$ and hence $\iota(\mathbb{U}(\mathcal{U}^\omega_{ana}))$ is a dense subspace of the Fock space. 
	\end{proof}
	The closure of $\mathbb{U}(\mathcal{U}^\omega_{ana})$ under the sesquilinear form $\langle\cdot,\cdot\rangle$ is isometrically isomorphic to the symmetric Fock space $\mathcal{F}(L_2(N,\omega))$. And we will denote this closure as $\mathcal{H}_\omega$.
	\begin{corollary}\label{corollary:abstractEmptyBasis}
		$\lambda_\emptyset$-vectors form another basis for $H_\omega$ and we have:
		\begin{equation}\label{equation:abstractEmpty}
			\langle\lambda_\emptyset(x_1,...,x_n),\lambda_\emptyset(y_1,...,y_m)\rangle = \sum_{\substack{0\leq p \leq \min\{m,n\}\\ A_n\sqcup B_n = [n]\\A_m\sqcup B_m = [m]\\|B_n| = |B_m| = p}}\prod_{i\in B_n, j\in B_m}\omega(x_i^*y_j)\prod_{i\in A_n}\omega(x_i^*)\prod_{j\in A_m}\omega(y_j)\pl.
		\end{equation}
		In addition we have the upper bound:
		\begin{align}\label{equation:abstractUpperBound}
			\begin{split}
				|\langle\lambda_\emptyset(x_1,...,x_n),\lambda_\emptyset(y_1,...,y_m)\rangle| \leq Cn^nm^m\prod_{1\leq i \leq n}|||x_i|||^2\prod_{1\leq j \leq m}|||y_j|||^2\pl.
			\end{split}
		\end{align}
	\end{corollary}
	\begin{proof}
		By the definition of $\lambda_{\emptyset\emptyset}$-vectors, it is clear that for all $n\in\mathbb{N}$, we have $\lambda_{\emptyset\emptyset}(x_1,...,x_n)\in \text{span}\{\lambda_\emptyset(y_1,...,y_m): y_i \in \mathcal{U}^\omega_{ana}\}$. Thus $\lambda_{\emptyset}$-vectors form a basis for $H_\omega$. 
		
		To prove Equation \ref{equation:abstractEmpty}, it suffices to prove the corresponding formula for $||\lambda_\emptyset(x_1,...,x_n)||^2$ and use the polarization identity. The formula for $||\lambda_\emptyset(x_1,...,x_n)||^2$ can be proved by reversing Equation \ref{equation:abstractEmptyEmptyBasis}. By the inclusion-exclusion principle, we have:
		\begin{align}\label{equation:abstractEmptyInverse}
			\begin{split}
				\lambda_\emptyset(x_1,...,x_{n}) &= \sum_{0\leq i \leq n}\sum_{\substack{A\subset[n]\\ |A| = i}}\prod_{i\in A}\omega(x_i)\lambda_{\emptyset\emptyset}(\{x_i:i\notin A\})
			\end{split}
		\end{align}
		Hence we have:
		\begin{align*}
			\begin{split}
				||\lambda_\emptyset(x_1,...,x_{n})||^2 &= \sum_{0\leq i \leq n}\sum_{\substack{A, B\subset [n] \\ |A| = |B| = i}}\prod_{i\in A}\omega(x^*_i)\prod_{j\in B}\omega(x_j)\langle\lambda_{\emptyset\emptyset}(\{x_i: i\notin A\}), \lambda_{\emptyset\emptyset}(\{x_j: j\notin B\})\rangle
				\\
				&=\sum_{\substack{0\leq i \leq n \\ A,B\subset[n]\\|A| = |B| = i}}\prod_{i\notin A, j\notin B}\omega(x^*_ix_j)\prod_{i\in A}\omega(x_i^*)\prod_{j\in B}\omega(x_j)\pl.
			\end{split}
		\end{align*}
		This is the equation we want to show. Finally, the derivation of the upper bound follows the exact same calculation as Lemma \ref{lemma:emptybasis}.
	\end{proof}
	We will construct the Poisson algebra in $\mathbb{B}(\mathcal{H}_\omega)$. The Poisson algebra will be generated by a set of unitaries whose generators are unbounded affiliated operators. These generators have controlled actions on the dense subspace $\mathbb{U}(\mathcal{U}^\omega_{ana})$.
	\begin{lemma}\label{lemma:unbounded}
		For $y\in\mathcal{U}^\omega_{ana}$, the left multiplication by $\lambda(y)$ on the Hilbert space $\mathcal{H}_\omega$ is a densely defined operator and its domain includes the universal enveloping algebra. $\mathbb{U}(\mathcal{U}^\omega_a)\ssubset D(\lambda(y))$. 
	\end{lemma}
	\begin{proof}
		For any $x_1,...,x_n\in \mathcal{U}^\omega_{ana}$, we have:
		\begin{align*}
			\begin{split}
				||\lambda(y)\lambda_\emptyset(x_1,...,x_n)||^2&= ||\lambda_\emptyset(y,x_1,...,x_n) + \sum_{1\leq i \leq n}\lambda_{\emptyset}(x_1,...,yx_i,...,x_n)||^2
			\end{split}
		\end{align*}
		There are $\binom{n+1}{2}$ inner products in the summation. Each summand has the form of Equation \ref{equation:abstractEmpty}. Therefore each term has an upper bound given by Equation \ref{equation:abstractUpperBound}. Hence we have:
		\begin{align}\label{equation:abstractActionNorm}
			\begin{split}
				||\lambda(y)\lambda_\emptyset(x_1,...,x_n)||^2 &\leq C\binom{n+1}{2}(n+1)^{2(n+1)}|||y|||^2\prod_{1\leq i \leq n}|||x_i|||^2\pl.
			\end{split}
		\end{align}
		This upper bound holds for all $x_i\in \mathcal{U}^\omega_{ana}$. Therefore $\text{span}\{\lambda_\emptyset(x_1,...,x_n):x_i\in\mathcal{U}^\omega_{ana}\} \ssubset D(\lambda(y))$. $\lambda(y)$ is a densely defined operator on $\mathcal{H}_\omega$ with $\mathbb{U}(\mathcal{U}^\omega_{ana})$ inside its domain.
	\end{proof}
	For positive $y\in\mathcal{U}^\omega_{ana}$, $\lambda(y)$ is a symmetric operator. In fact, it is essentially self-adjoint:
	\begin{lemma}\label{lemma:abstractEssentialSA}
		For non-zero positive elements in the Tomita algebra: $y^* = y\in \mathcal{U}^\omega_a$, the left multiplication by $\lambda(y)$ is an essentially self-adjoint operator on $\mathcal{H}_\omega$. 
	\end{lemma}
	\begin{proof}
		By the recursive definition of $\lambda_\emptyset$-vectors (c.f. Equation \ref{equation:abstractEmptyBasis}), $\text{span}\{\lambda_\emptyset(x_1,...,x_n)\} = \text{span}\{\lambda(x_1)...\lambda(x_n)\}$. Hence to see $\lambda(y)$ is essentially self-adjoint, it suffices to work with the $\lambda$-basis. 
		
		We first derive some lengthly formulae describing the action of $\lambda(y)$. First, we have:
		\begin{align}\label{equation:abstractActionSA}
			\begin{split}
				\langle&\lambda(x_1)...\lambda(x_n), \lambda(y)\lambda(x_1)...\lambda(x_n)\rangle = \sum_{\substack{\sigma\in P(2n+1) \\ \{y\}\in\sigma}}\sum_{\substack{A\in \sigma \\ B_A\cup B_A^* = A}}\omega(y)\omega(\overleftarrow{\prod_{(i+n)\in B_A^*}}x_i^*\overrightarrow{\prod_{i\in B_A}}x_i) \\&+ \sum_{\substack{\sigma\in P(2n+1) \\ \{y\}\notin \sigma}}\bigg[\big[\sum_{\substack{y\notin A\in \sigma \\ B_A\cup B_A^* = A}}\omega(\overleftarrow{\prod_{(i+n)\in B_A^*}}x_i^*\overrightarrow{\prod_{i\in B_A}}x_i)\big]\times\big[\sum_{\substack{y\in A_0 \\ B_{A_0}\cup B_{A_0}^* = A_0}}\omega(\overleftarrow{\prod_{(i+n)\in B_{A_0}^*}}x_i^*\cdot y \cdot\overrightarrow{\prod_{i\in B_{A_0}}}x_i)\big]\bigg]
			\end{split}
		\end{align}
		where the first summation is over the partitions that include $\{y\}$ as a singleton and the second summation is over the rest of the partitions. Since $\{y\}$ is a singleton in the first term, we can rewrite this term as:
		\begin{align*}
			\begin{split}
				\sum_{\substack{\sigma\in P(2n+1) \\ \{y\}\in\sigma}}\sum_{\substack{A\in \sigma \\ B_A\cup B_A^* = A}}\omega(y)\omega(\overleftarrow{\prod_{(i+n)\in B_A^*}}x_i^*\overrightarrow{\prod_{i\in B_A}}x_i) &= \omega(y)\sum_{\sigma\in P(2n)}\sum_{\substack{A\in\sigma \\ B_A\cup B_A^* = A}}\omega(\overleftarrow{\prod_{(i+n)\in B_A^*}}x_i^*\overrightarrow{\prod_{i\in B_A}}x_i)\\&=\omega(y)||\lambda(x_1)...\lambda(x_n)||^2 \pl.
			\end{split}
		\end{align*}
		In particular, since $\omega(y) \geq 0$, this term is non-negative and it is zero if and only if $y = 0$.
		
		For the second term, we have the following identity:
		\begin{align}
			\begin{split}
				\sum_{\substack{\sigma\in P(2n+1) \\ \{y\}\notin \sigma}}&\bigg[\big[\sum_{\substack{y\notin A\in \sigma \\ B_A\cup B_A^* = A}}\omega(\overleftarrow{\prod_{(i+n)\in B_A^*}}x_i^*\overrightarrow{\prod_{i\in B_A}}x_i)\big]\times\big[\sum_{\substack{y\in A_0 \\ B_{A_0}\cup B_{A_0}^* = A_0}}\omega(\overleftarrow{\prod_{(i+n)\in B_{A_0}^*}}x_i^*\cdot y \cdot\overrightarrow{\prod_{i\in B_{A_0}}}x_i)\big]\bigg] \\&= ||\sum_{1\leq j \leq n}\lambda(x_1)...\lambda(y^{1/2}x_j)...\lambda(x_n)||^2 \geq 0 \pl.
			\end{split}
		\end{align}
		Therefore we have: $\langle\lambda(x_1)...\lambda(x_n), \lambda(y)\lambda(x_1)...\lambda(x_n)\rangle = ||\lambda(y^{1/2})\lambda(x_1)...\lambda(x_n)||^2$. By polarization identity, it follows that $\lambda(y)$ is a symmetric operator.
		
		More generally, for any finite linear combination of the form $\xi:=\sum_{1\leq j \leq m}c_j \lambda(x^{(j)}_1)...\lambda(x^{(j)}_{n_j})$ where $x^{(j)}_1,...,x^{(j)}_{n_j}\in\mathcal{U}^\omega_{ana}$ and $c_j\in\mathbb{C}$, we have:
		\begin{equation}\label{equation:abstractActionLinearComb}
			\langle\xi, \lambda(y)\xi\rangle = \omega(y)||\xi||^2 + ||\sum_{1\leq j \leq m}c_j \sum_{1\leq k \leq n_j}\lambda(x^{(j)}_1)...\lambda(y^{1/2}x^{(j)}_k)...\lambda(x^{(j)}_{n_j})||^2\pl.
		\end{equation}
		Using Equation \ref{equation:abstractActionSA} and polarization identity, Equation \ref{equation:abstractActionLinearComb} can be derived by a simple induction on the index $j$. 
		
		If $\xi\in ker(\lambda(y)\pm i)$, then $\langle\xi,\lambda(y)\xi\rangle = \pm i ||\xi||^2_{\phi_\omega}\in i\mathbb{R}$. Hence by equation \ref{equation:abstractActionLinearComb}, it is necessary that $\omega(y)||\xi||^2 = 0$. Hence $\omega(y) = 0$ and by faithfulness of $\omega$, this implies $y = 0$. Therefore for non-zero positive operator $y\in\mathcal{U}^\omega_{ana}$, $\lambda(y)$ is an essentially self-adjoint operator.
	\end{proof}
	\begin{corollary}\label{corollary:affiliatedGenerators}
		Using the same notation as Lemma \ref{lemma:abstractEssentialSA}, $\{\exp(it\lambda(y))\}_{t\in\mathbb{R}}$ is a strongly continuous one-parameter unitary subgroup in $\mathbb{B}(\mathcal{H}_\omega)$.
	\end{corollary}
	\begin{proof}
		The claim follows directly from Stone's theorem \cite{T2}.
	\end{proof}
	\begin{definition}\label{definition:abstractPoisson}
		The (abstract) Poisson algebra $\mathbb{P}^{ab}_\omega(N)\ssubset \mathbb{B}(\mathcal{H}_\omega)$ is the von Neumann algebra generated by the unitaries $\{\exp(it\lambda(x)): x\geq 0,\text{ }x\in\mathcal{U}^\omega_{ana}\}$. The Poisson state is given by $\langle\xi_0, \cdot\xi_0\rangle$ where $\xi_0\in\mathbb{C}\ssubset\mathbb{U}(\mathcal{U}_{ana}^\omega)$ is the unit vector (unique upto a phase factor) in the constant component of the universal enveloping algebra.
	\end{definition}
	\begin{proposition}\label{proposition:identicalPoisson}
		There exists a normal $*$-isomorphism:
		\begin{equation}
			\pi:\mathbb{P}_\omega N\rightarrow \mathbb{P}^{ab}_\omega N: x\mapsto U^*xU
		\end{equation}
		which is implemented by the unitary isomorphism:
		\begin{equation}
			U: \mathcal{H}_\omega \rightarrow L_2(\mathbb{P}_\omega N, \phi_\omega):\lambda(x_1)...\lambda(x_n)\mapsto \lambda(x_1)...\lambda(x_1)d_{\phi_\omega}^{1/2}\pl.
		\end{equation}	
		Therefore the two constructions of Poisson algebra coincide. 
	\end{proposition}
	\begin{proof}
		This is a direct consequence of the isomorphism result in \cite[Section 5]{MAW}. The tuple $(\mathbb{U}(\mathcal{U}^\omega_{ana}), \langle\xi_0, \cdot\xi_0\rangle)$ is a noncommutative $*$-probability space (c.f. \cite[Section 5]{MAW}). Let $S:=\{x^* = x: x\in \mathcal{U}^\omega_{ana}\}\ssubset\mathbb{U}(\mathcal{U}^\omega_{ana})$ be a generating set for representations of the $*$-probability space. Then consider the linear map:
		\begin{equation*}
			\alpha:\mathbb{U}(\mathcal{U}^\omega_{ana})\rightarrow L_2(\mathbb{P}_\omega N, \phi_\omega): \lambda(x_1)...\lambda(x_n)\mapsto \lambda(x_1)...\lambda(x_n)d_{\phi_\omega}^{1/2}\pl.
		\end{equation*}
		And consider the $*$-representation of the generating set:
		\begin{equation*}
			\pi:S\rightarrow \mathcal{L}_{s.a.}(L_2(\mathbb{P}_\omega N, \phi_\omega)): x\mapsto \lambda(x)
		\end{equation*}
		where $\mathcal{L}_{s.a.}(\mathcal{H})$ denotes the set of self-adjoint densely defined operators on $\mathcal{H}$. Then using the language of \cite[Section 5]{MAW}, it is clear that $(\pi,\alpha, d_{\phi_\omega}^{1/2}, L_2(\mathbb{P}_\omega N, \phi_\omega))$ is a representation of the noncommutative $*$-probability space $(\mathbb{U}(\mathcal{U}^\omega_{ana}), \langle\xi_0, \cdot\xi_0\rangle)$. On the other hand, it is clear that the identity maps:
		\begin{equation*}
			\iota_\mathcal{A}: \mathbb{U}(\mathcal{U}^\omega_{ana})\rightarrow \mathcal{H}_\omega: \lambda(x_1)...\lambda(x_n)\mapsto\lambda(x_1)...\lambda(x_n)\xi_0
		\end{equation*}
		together with identity representation:
		\begin{equation*}
			\iota_S:S\rightarrow \mathcal{L}_{s.a.}(\mathcal{H}_\omega): \lambda(x)\mapsto\lambda(x)
		\end{equation*}
		form quadruple $(\iota_S, \iota_\mathcal{A}, \xi_0, \mathcal{H}_\omega)$ that is also a representation of the probability space $\mathbb{U}(\mathcal{U}^\omega_{ana})$. 
		
		Since the moment formula satisfies the growth condition (c.f. Lemma \ref{lemma:momentformulaPoisson} and \cite[Section 5]{MAW}), then there exists a normal $*$-homomorphism: $\pi:\mathbb{P}_\omega N\rightarrow \mathbb{P}^{ab}_\omega N$. Since the Poisson state on $\mathbb{P}_\omega N$ and the state $\langle\xi_0, \cdot\xi_0\rangle$ on $\mathbb{P}^{ab}_\omega N$ are both faithful, then the map $\pi$ is an isomorphism. Moreover it can be implemented by a unitary isomorphism:
		\begin{equation}
			U: \mathcal{H}_\omega\rightarrow L_2(\mathbb{P}_\omega N, \phi_\omega)\pl.
		\end{equation}
		Therefore both constructions of Poisson algebra coincide.  
	\end{proof}
	\subsection{Functorial Properties of Poissonization with N.S.F. Weight}
	In this subsection, we collect some important functorial properties of Poissonization with n.s.f. weights. These properties have been proved in previous sections for Poissonization with bounded functionals.
	\begin{lemma}\label{lemma:GammaLiftWeight}
		Let $(N,\omega)$ be a von Neumann algebra along with a n.s.f. weight and let $\mathcal{U}^\omega_{ana}$ be the Tomita algebra of $\omega$. Moreover, let $S^\omega_{ana}:=\{x\in \mathcal{U}^\omega_{ana}: x - 1\in m_\omega, \text{ }||x||\leq 1\}$. This is a multiplicative subgroup of $N$ that is invariant under conjugation.
		
		Consider the map:
		\begin{equation}\label{equation:GammaUnitary}
			\Gamma: \{\exp(ix)\in N: x^* = x, \text{ }x \in \mathcal{U}^\omega_{ana}\}\rightarrow \{\exp(i\lambda(x))\in\mathbb{P}_\omega N: \lambda(x)\eta\mathbb{P}_\omega N\}
		\end{equation}
		where $\lambda(x)\eta\mathbb{P}_\omega N$ indicates that $\lambda(x)$ is an affiliated operator to $\mathbb{P}_\omega N$. 
		Then this map extends to a $*$-group homomorphism: $\Gamma: S^\omega_{ana}\rightarrow \mathbb{P}_\omega N^\times$. In addition, the following familiar formula holds:
		\begin{equation}\label{equation:GammaLiftWeight}
			\phi_\omega(\Gamma(a)\lambda_\emptyset(x_1,...,x_n)) = \exp(\omega(a - 1))\prod_{1\leq i \leq n}\omega(ax_i)\pl.
		\end{equation}
	\end{lemma}
	\begin{proof}
		By the abstract construction of Poisson algebra, the map $\Gamma$ is well-defined on the set of unitaries generated by self-adjoint analytic elements. For any contraction $x\in S^\omega_{ana}$, define $\Gamma(x)$ to be the ultraproduct $(\Gamma_s(x))^\bullet$. For each $s$, we have (c.f. Proposition \ref{proposition:PoissonForWeight1}):
		\begin{align}
			\begin{split}
				\phi_s(\Gamma_s(x)\lambda_\emptyset(x_1,...,x_n)) &= \exp(\varphi_s(x-1))\prod_{1\leq i \leq n}\varphi_s(xx_i)\pl.
			\end{split}
		\end{align}
		For $x\in S^\omega_{ana}$, we can take the limit: $\lim_s\exp(\varphi_s(x-1))\prod_{1\leq i \leq n}\varphi_s(xx_i) = \exp(\omega(x-1))\prod_{1\leq i\leq n}\omega(xx_i)$ since $x_i\in\mathcal{U}^\omega_{ana}\ssubset m_\omega$ and hence $xx_i\in m_\omega$. Thus Equation \ref{equation:GammaLiftWeight} holds.
		
		For self-adjoint analytic element $x$, we have: $|\omega(e^{ix} - 1)| \leq \sum_{n\geq 1}\frac{|\omega(x^n)|}{n!} \leq \exp(|||x|||) < \infty$. Therefore $e^{ix} -1 \in m_\omega$. Hence $\{\exp(ix):x^* = x,\text{ }x\in\mathcal{U}^\omega_{ana}\}\ssubset S^\omega_{ana}$. In addition, using the ultraproduct construction, $\Gamma(e^{ix})$ is the ultraproduct $(\Gamma_s(e^{ix}))^\bullet$ where each component $\Gamma_s(e^{ix})$ acts on $L_2(\mathbb{P}_{\varphi_s} N, \phi_s)$ (c.f. Proposition \ref{proposition:PoissonForWeight1}). Hence $\Gamma: S^\omega_{ana}\rightarrow \mathbb{P}_\omega N^\times: x\mapsto \Gamma(x) = (\Gamma_s(x))^\bullet$ is an extension of Equation \ref{equation:GammaUnitary}.
		
		Finally we observe that $S^\omega_{ana}$ is a multiplicative group (We have already did this calculation in the proof of Corollary \ref{corollary:modularAutoWeight}):
		\begin{equation*}
			xy - 1 = (x - 1)(y - 1) + (x - 1) + (y - 1)
		\end{equation*}
		where $x,y\in S^{\omega}_{ana}$. Since $x -1 , y-1 $ are both in the Tomita algebra, their product is in the Tomita algebra as well. Hence $xy\in S^{\omega}_{ana}$. In addition, for each $s$ using the symmetric tensor product model, it is clear that $\Gamma_s(x)\Gamma_s(y) = \Gamma_s(xy)$ for $x,y\in S^\omega_{ana}$. Then passing to the ultraproduct, we have $\Gamma(x)\Gamma(y) = \Gamma(xy)$. Thus $\Gamma$ is a group homomorphism.
	\end{proof}
	\begin{lemma}\label{lemma:contractionLiftWeight}
		Let $(N,\omega_N)$ and $(M,\omega_M)$ be two von Neumann algebras with fixed n.s.f. weights. Let $T:N\rightarrow M$ be a normal unital contraction such that $\omega_M\circ T = \omega_N$. Assume there exists a normal unital linear map: $\hat{T}:M\rightarrow N$ such that $(T,\hat{T})$ forms a dual pair: $\omega_M(T(x)y) = \omega_N(x\hat{T}(y))$. Then there exists a normal contraction:
		\begin{equation}
			\Gamma(T):\mathbb{P}_\omega N\rightarrow \mathbb{P}_\omega M: \Gamma(a)\mapsto \Gamma(Ta)
		\end{equation}
		where $a$ is in the group $S^{\omega_N}_{ana} $. In addition $\phi_{\omega_M}\circ\Gamma(T) = \phi_{\omega_N}$ and the following formula holds in the predual:
		\begin{equation}
			\Gamma(T)^*\lambda_\emptyset(x_1,...,x_n)d_{\phi_M} = \lambda_\emptyset(\hat{T}x_1,...,\hat{T}x_n)d_{\phi_N}\pl.
		\end{equation}
	\end{lemma}
	\begin{proof}
		This is a direct generalization of Proposition \ref{proposition:contractionLift}. The proof of Proposition \ref{proposition:contractionLift} holds verbatim because all calculations are well-defined thanks to Corollary \ref{corollary:GammaAction}.
	\end{proof}
	Using this lemma, we can now collect a number of functorial properties of Poissonization with n.s.f. weights.
	\begin{proposition}\label{proposition:collectionFunctorialWeight}
		Using the same notation as Lemma \ref{lemma:contractionLiftWeight}, the following statements are true:
		\begin{enumerate}
			\item Let $\pi:N\rightarrow M$ be a normal $*$-homomorphism that preserves the weight: $\omega_M\circ\pi = \omega_N$. Then there exists a normal $*$-homomorphism:
			\begin{equation}
				\Gamma(\pi):\mathbb{P}_\omega N\rightarrow \mathbb{P}_\omega M: \Gamma(x)\mapsto\Gamma(\pi(x))
			\end{equation}
			where $x\in S^{\omega_N}_{ana}$. In addition, $\phi_{\omega_M}\circ\Gamma(\pi) = \phi_{\omega_N}$.
			\item Let $N'\ssubset N$ be a von Neumann subalgebra such that the restricted weight $\omega|_{N'}$ is semifinite. In addition, assume $N'$ is invariant under the modular automorphism $\sigma^\omega_t(N') = N'$. Then $\mathbb{P}_{\omega}N'\ssubset \mathbb{P}_{\omega}N$ and there exists a normal conditional expectation: $\mathcal{E}:\mathbb{P}_\omega N \rightarrow\mathbb{P}_\omega N'$ such that $\phi_\omega\circ\mathcal{E} = \phi_\omega$. Moreover $\mathcal{E} = \Gamma(E)$ where $E:N\rightarrow N'$ is the conditional expectation and $\Gamma(E)$ is the lift constructed in Lemma \ref{lemma:contractionLiftWeight}.
		\end{enumerate}
	\end{proposition}
	\begin{proof}
		For the first statement, using Lemma \ref{lemma:contractionLiftWeight}, the proof can be copied verbatim from Corollary \ref{corollary:homomorphismLift}.
		
		For the second statement, by Corollary \ref{corollary:modularAutoWeight} and Lemma \ref{lemma:contractionLiftWeight}, the proof can be copied verbatim from Corollary \ref{corollary:conditionalExpLift}. 
	\end{proof}
	To end this subsection, we collect all the facts about Poissonization in the following theorem. This is the main result of this paper.
	\begin{theorem}\label{theorem:Poisson}
		Let $\textbf{vNa}_w$ be the category where the objects are $(N,\omega)$ where $N$ is a von Neumann algebra and $\omega$ is a normal faithful semifinite weight and the morphisms are normal $*$-homomorphisms that preserve the n.s.f. weights. Let $\textbf{vNa}_s$ be the category where the objects are $(N',\varphi')$ where $N'$ is a von Neumann algebra and $\varphi'$ is a normal faithful state and the morphisms are normal $*$-homomorphisms that preserve the states. Then Poissonization is a functor:
		\begin{equation}
			\textbf{Poiss}:\textbf{vNa}_w\Rightarrow\textbf{vNa}_s: (N,\omega)\rightarrow (\mathbb{P}_\omega N, \phi_\omega)
		\end{equation}
		such that the following diagram holds:
		\begin{center}
			\begin{tikzcd}
				N \arrow[r, "\pi"] \arrow[d, "\textbf{Poiss}"] & M \arrow[d, "\textbf{Poiss}"] \\
				\mathbb{P}_{\omega_N} N \arrow[r, "\Gamma(\pi)"] & \mathbb{P}_{\omega_M} M
			\end{tikzcd}
		\end{center}
		where $\omega_M\circ\pi = \omega_N$ and $\phi_{\omega_M}\circ\Gamma(\pi) = \phi_{\omega_N}$.
		
		In addition, Poissonization induces a functor on the category of Haagerup $L_2$-spaces $\textbf{Hilb}_{vNa}$:
		\begin{equation}
			\textbf{Poiss}_{Hilb}: \textbf{Hilb}_{vNa} \Rightarrow \textbf{Hilb}_{vNa}: L_2(N, \omega)\rightarrow L_2(\mathbb{P}_\omega N, \phi_\omega)
		\end{equation}
		such that $\textbf{Poiss}_{Hilb} = \textbf{Fock}_{Sym}$ where $\textbf{Fock}_{Sym}$ is the functor on the category of Hilbert spaces that implements the symmetric Fock space construction.
	\end{theorem}
	\subsection{Poissonization and Normal Unital Completely Positive Maps}
	In this section, we generalizes the functor $\textbf{Poiss}$ to the category of von Neumann algebras where the morphisms are now weight-preserving unital completely positive maps. We shall denote this category as $\textbf{vNa}_{OpSys}$ emphasizing the fact that this is a subcategory of the category of operator systems.
	\begin{proposition}\label{proposition:quantumChannelLift}
		Let $T:N\rightarrow M$ be a unital completely positive map that preserves the reference functional: $\omega_M\circ T = \omega_N$. Then there exists a normal unital completely positive map:
		\begin{equation}
			\Gamma(T):\mathbb{P}_{\omega_N}N\rightarrow\mathbb{P}_{\omega_M}M
		\end{equation}
		such that for all $x\in S^{\omega_N}_{ana}$ (c.f. Lemma \ref{lemma:GammaLiftWeight}), we have: $\Gamma(T)\Gamma(x) = \Gamma(T(x))$. In addition, $\Gamma(T)$ preserves the Poisson states: $\phi_{\omega_M}\circ\Gamma(T) = \phi_{\omega_N}$.
	\end{proposition}
	\begin{proof}
		The proof relies on an extension of Stinespring dilation due to Paschke \cite{Pash1,Lance}. Given a unital completely positive map $T:N\rightarrow M$, there exists a Hilbert $W^*$-bimodule $H_M$ over the von Neumann algebra $M$ and a normal $*$-homomorphism: $\pi: N\rightarrow \mathcal{L}(H_M)$ where $\mathcal{L}(H_M)$ is the von Neumann algebra of adjointable bimodule maps on $H_M$ such that the following diagram commutes:
		\begin{center}
			\begin{tikzcd}
				N \arrow[rd,"T"] \arrow[r, "\pi"]  & \mathcal{L}(H_M)\arrow[d,"E"] \\ &M
			\end{tikzcd}
		\end{center}
		where there exists a projection $e\in \mathcal{L}(H_M)$ such that $E(S) = eSe$. 
		
		Using this commutative diagram, $\omega_M\circ E$ defines a normal weight on $\mathcal{L}(H_M)$. The difficulty in lifting this diagram via Poissonization is the fact that this weight $\omega_M\circ E$ needs not be faithful. To fix this, fix a normal faithful state $\varphi$ on $\mathcal{L}(H_M)$ and consider the perturbed weight:
		\begin{equation}
			\omega_\epsilon(S):= \omega_M(eSe) + \epsilon \phi((1-e)S(1-e))
		\end{equation}
		where $\epsilon >0$ is a perturbative parameter. Note $\omega_\epsilon$ is a normal faithful semifinite weight on $\mathcal{L}(H_M)$. And $\omega_\epsilon\circ\pi = \omega_M\circ E\circ\pi = \omega_M\circ T = \omega_N$. Hence for each $\epsilon > 0$, Poissonization gives us lifts:
		\begin{align*}
			\begin{split}
				&\Gamma(\pi):\mathbb{P}_{\omega_N}N\rightarrow \mathbb{P}_{\omega_\epsilon}\mathcal{L}(H_M)\\
				&\Gamma(E):\mathbb{P}_{\omega_\epsilon}\mathcal{L}(H_M)\rightarrow \mathbb{P}_{\omega_M} M
			\end{split}
		\end{align*}
		where the lift $\Gamma(E)$ is well-defined by Lemma \ref{lemma:contractionLiftWeight}. Since for any $\epsilon>0$ the projection $e$ is in the centralizer algebra $\mathcal{L}(H_M)_{\omega_\epsilon}$, then for all $x$ in the definition ideal of $\omega_\epsilon$ we have:
		\begin{equation}
			\omega_\epsilon(ex) = \omega_\epsilon(xe)\pl.
		\end{equation}
		In particular, for any $\epsilon > 0$, for all $x\in m_{\omega_M}$ we have $exe\in m_{\omega_\epsilon}$ and for all $y\in m_{\omega_N}$ we have $\pi(y)\in m_{\omega_\epsilon}$. Hence $\omega_\epsilon(ex\pi(y)) = \omega_\epsilon(x\pi(y)e)$. 
		
		Moreover, using the $L_2$-adjoint of $T$ we can define a completely positive contraction:
		\begin{equation}
			\Gamma(T^\dagger):L_2(\mathbb{P}_{\omega_M}M, \phi_M)\rightarrow L_2(\mathbb{P}_{\omega_N}N, \phi_N):d_{\phi_M}^{1/4}\lambda_\emptyset(x_1,...,x_n)d_{\phi_M}^{1/4}\mapsto d_{\phi_N}^{1/4}\lambda_\emptyset(T^\dagger x_1,...,T^\dagger x_n)d_{\phi_N}^{1/4}\pl.
		\end{equation}
		Combining these observations, we now consider the adjoint action of $\Gamma(\pi)_*\circ\Gamma(E)_*$ on the predual of the Poisson algebra:
		\begin{align}\label{equation:predualLiftPart1}
			\begin{split}
				\big(&\Gamma(\pi)_*\circ\Gamma(E)_*d_{\phi_M}^{1/2}\lambda_\emptyset(x_1,...,x_n)d_{\phi_M}^{1/2}, \Gamma(y)\big) = \big(d_{\phi_{\omega_\epsilon}}^{1/2}\lambda_\emptyset(ex_1e,...,ex_ne)d_{\phi_{\omega_\epsilon}}^{1/2}, \Gamma(\pi(y))\big) \\&=\exp(\omega_\epsilon(\pi(y) - 1))\prod_{1\leq i \leq n}\omega_\epsilon(ex_ie\pi(y))
				=\exp(\omega_N(y - 1))\prod_{1\leq i \leq n}\omega_\epsilon(ex_i\pi(y))
				\\&=\exp(\omega_N(y - 1))\prod_{1\leq i \leq n}\omega_M(x_i T(y)) = \exp(\omega_N(y - 1))\prod_{1\leq i \leq n}\langle T^\dagger\widehat{x_i}, \widehat{y}\rangle 
				\\
				&=\langle \Gamma(T^\dagger)d_{\phi_M}^{1/4}\lambda_\emptyset(x_1,...,x_n)d_{\phi_M}^{1/4}, d_{\phi_N}^{1/4}\Gamma(y)d_{\phi_N}^{1/4}\rangle = \big(d_{\phi_N}^{1/4}(\Gamma(T^\dagger)d_{\phi_M}^{1/4}\lambda_\emptyset(x_1,...,x_n)d_{\phi_M}^{1/4})d_{\phi_N}^{1/4}, \Gamma(y)\big)
			\end{split}
		\end{align}
		where $x_i\in \mathcal{U}_{ana}^{\omega_M}$, and $y\in S^{\omega_N}_{ana}\ssubset N$ (c.f. Lemma \ref{lemma:GammaLiftWeight}) and $\big(\cdot,\cdot\big)$ denotes the canonical pairing between the Poisson algebra and its predual while $\langle\cdot,\cdot\rangle$ denotes the $L_2$-inner product. Therefore the adjoint of $\Gamma(\pi)_*\circ\Gamma(E)_*$ extends to the adjoint map $\Gamma(T^\dagger)^\dagger = \Gamma(T): L_2(\mathbb{P}_{\omega_N}N, \phi_N)\rightarrow L_2(\mathbb{P}_{\omega_M}M, \phi_M)$. Restricted to the Poisson algebra, we have the desired unital completely positive map. 
	\end{proof}
	\section{Some Properties of Poisson Algebra}\label{section:properties}
	In this section, we discuss some interesting properties of Poisson algebras. First we determine the type of a Poisson algebra $\mathbb{P}_\omega N$ using the Arveson spectrum of the modular automorphism group $\{\sigma^\omega_t\}_{t\in \mathbb{R}}$ on $N$. Using this result, we construct type $III_\lambda$ factors from unitary principal series representations of real semisimple Lie groups. The value of $\lambda\in[0,1]$ depends on the root of the restricted representation on the maximal Abelian subgroup \cite{Knapp}. 
	
	Second, we show that under some general conditions, inclusions of a pair of type III Poisson algebras are split \cite{DL1} if and only if the Poisson algebras are hyperfinite. The splitting property is an important property in the algebraic formulation of quantum field theory \cite{haag}. Typically, the splitting property holds in algebraic quantum field theories because certain "partition function" is finite for all "temperature" \cite{BDL1, BDL2}. In the case of Poisson algebras, the analogs of these "partition functions" are rarely finite, but splitting property hold in general when the algebras are hyperfinite. This is due to the special properties of Poisson algebras. 
	
	Lastly, we calculate the quantum relative entropy \cite{Ara1, Ara2} of some states on the Poisson algebra. When the algebra is type III, it is generally impossible to calculate the relative entropy for a pair of arbitrary states. Here our calculation holds when one state is the Poisson state and the other is a bounded perturbation of the Poisson state. The formula relates the (Araki) relative entropy of the Poisson algebra $\mathbb{P}_\omega N$ to the Lindblad relative entropy \cite{Lin} of the original von Neumann algebra $N$. Since the original von Neumann algebra $N$ is typically semi-finite in actual applications, the Lindblad entropy on $N$ can be easily calculated using well-known formula. This result will be extended and studied in depth in a companion paper.
	\subsection{Factoriality and Type of Poisson Algebras}\label{subsection:type}
	In this subsection, we determine the type of some Poisson algebras. The key observation is the following proposition.
	\begin{proposition}\label{proposition:typeII1}
		The Poisson algebra $\mathbb{P}_{\tr}\mathbb{B}(\ell_2)$ is the hyperfinite $II_1$ factor where $\ell_2$ denotes the separably infinite dimensional Hilbert space $\ell_2(\mathbb{Z})$. 
	\end{proposition}
	\begin{proof}
		From straight-forward calculations, it is clear that the Poisson state $\phi_{\tr}$ is a tracial state on $\mathbb{P}_{\tr} \mathbb{B}(\ell_2)$. Therefore we only need to show that the algebra is a factor. 
		
		To see this we need to consider the following sequence of subalgebras: $\{\mathbb{P}_{\tr_k}\mathbb{B}(\ell_2^k)\}_{k\geq 1}$ where $\ell_2^k\ssubset \ell_2$ is the $k$-dimensional subspace of $\ell_2$ and $\tr_k$ is the canonical trace on the subalgebra $\mathbb{B}(\ell^k_2)$. Let $\{e_k\in\mathbb{B}(\ell_2)\}_k$ be the increasing family of orthogonal projections such that $e_k\ell_2 = \ell^k_2$. Then for each pair $k < n$, there exists a normal $*$-homomorphism:
		\begin{equation}
			\pi_{n,k}:\mathbb{P}_{\tr_k}\mathbb{B}(\ell^k_2) \rightarrow\mathbb{P}_{\tr_n}\mathbb{B}(\ell_2^n): \Gamma(x)\mapsto \Gamma(x + e_n - e_k)\pl.
		\end{equation}
		In addition, $\phi_{\tr_n}\circ \pi_{n,k} = \phi_{\tr_k}$:
		\begin{equation}
			\phi_{\tr_n}(\pi_{n,k}(\Gamma(x))) = \phi_{\tr_n}(\Gamma(x + e_n - e_k)) = e^{\tr_n(x + e_n - e_k - e_n)} = e^{\tr_k(x - e_k)} = \phi_{\tr_k}(\Gamma(x))\pl.
		\end{equation}
		Hence the Poisson algebra $\mathbb{P}_{\tr}\mathbb{B}(\ell_2)$ is isomorphic to the inductive limit of $\{(\mathbb{P}_{\tr_k}\mathbb{B}(\ell^k_2), \phi_{\tr_k})\}_{k\geq 1}$. 
		
		On $\mathbb{B}(\ell_2)$, we fix a sequence of unitaries weakly converging to 0: $\{u_{2n}:=\id_n \otimes \begin{pmatrix}0 & 1 \\ 1 & 0\end{pmatrix}\in \mathbb{B}(\ell^{2n}_2)\}_{n\geq 1}$. Then on the weakly dense $*$-subalgebra $\cup_{k\geq 1}\mathbb{P}_{\tr_k}\mathbb{B}(\ell^k_2)\ssubset \mathbb{P}_{\tr}\mathbb{B}(\ell_2)$, we consider the following sesquilinear form:
		\begin{equation}
			\big(\Gamma(x), \Gamma(y))\big)_u:=\lim_{n\rightarrow \infty}\phi_{\tr_{2n}}(\Gamma(x^*)\Gamma(u^*_{2n})\Gamma(y)\Gamma(u_{2n}))\pl.
		\end{equation}
		For any pair of $\Gamma(x),\Gamma(y)$ in the dense subalgebra, by the upwards closedness of the inductive sequence, there exists $m\in\mathbb{N}$ such that $\Gamma(x),\Gamma(y)\in\mathbb{P}_{\tr_m}\mathbb{B}(\ell_2^m)$. Then for all $n\in\mathbb{N}$ such that $n > m$, we have the following factorization:
		\begin{align}
			\begin{split}
				\phi&_{\tr_{2n}}(\Gamma(x^*)\Gamma(u^*_{2n})\Gamma(y)\Gamma(u_{2n})) = \exp\tr_{2n}\big((x^* + e_{2n} - e_m)u^*_{2n}(y + e_{2n} - e_m)u_{2n} - e_{2n}\big)\\
				&=\exp\tr_{2n}\big(\begin{pmatrix}
					x^* + e_n - e_m & 0 \\ 0 & e_n
				\end{pmatrix}\begin{pmatrix}
				0 & e_n \\ e_n & 0
			\end{pmatrix}\begin{pmatrix}
			y + e_n - e_m & 0 \\ 0 & e_n
		\end{pmatrix}\begin{pmatrix}
		0 & e_n \\ e_n & 0
	\end{pmatrix} - e_{2n}\big)
			\\
			&=\exp\big(\tr_m(x^* - e_m)\big)\exp\big(\tr_m(y - e_m)\big) = \phi_{\tr_m}(\Gamma(x^*))\phi_{\tr_m}(\Gamma(y)) = \phi_{\tr}(\Gamma(x^*))\phi_{\tr}(\Gamma(y))
			\end{split}
		\end{align}
		where the last equation follows from the state-preserving property of the maps $\{\pi_{n,m}\}_{n\in\mathbb{N}}$. Hence on the dense subalgebra $\cup_{k\geq 1}\mathbb{P}_{\tr_k}\mathbb{B}(\ell^k_2)$, the sesquilinear form factorizes:
		\begin{equation}
			\big(\Gamma(x), \Gamma(y)\big)_u = \phi_{\tr}(\Gamma(x^*))\phi_{\tr}(\Gamma(y))\pl.
		\end{equation}
		Now suppose $\Gamma(y)$ is in the center of the Poisson algebra $\mathbb{P}_{\tr}\mathbb{B}(\ell_2)$, we can calculate the sesquilinear form in another way. For each $n\in\mathbb{N}$, we have:
		\begin{equation}
			\phi_{\tr_{2n}}(\Gamma(x^*)\Gamma(u_{2n}^*)\Gamma(y)\Gamma(u_{2n})) = \phi_{\tr_{2n}}(\Gamma(x^*)\Gamma(y)\Gamma(u^*_{2n})\Gamma(u_{2n})) =\phi_{\tr_{2n}}(\Gamma(x^*)\Gamma(y))\pl.
		\end{equation}
		Again by the state-preserving property of the embeddings $\{\pi_{n,m}\}_{n,m}$, we have:
		\begin{equation}
			\big(\Gamma(x), \Gamma(y)\big)_u = \phi_{\tr}(\Gamma(x^*)\Gamma(y))\pl.
		\end{equation}
		Since the equation $\phi_{\tr}(\Gamma(x^*)\Gamma(y)) = \phi_{\tr}(\Gamma(x^*))\phi_{\tr}(\Gamma(y))$ holds for all $\Gamma(x^*)$ in the dense subalgebra, $\Gamma(y) = \phi_{\tr}(\Gamma(y))1$. Hence the center of the Poisson algebra is trivial.
	\end{proof}
	\begin{corollary}\label{corollary:typeIII}
		Using the same notation as Proposition \ref{proposition:typeII1} and let $\omega$ be a n.s.f. weight on a von Neumann algebra $N$. Then the Poisson algebra $\mathbb{P}_{\omega\otimes\tr}(N\overline{\otimes}\mathbb{B}(\ell_2))$ is a factor and its centralizer is also a factor $\mathbb{P}_{\omega\otimes\tr}(N_\omega\overline{\otimes}\mathbb{B}(\ell_2))$. Consequently, the type of the Poisson algebra $\mathbb{P}_{\omega\otimes\tr}(N\overline{\otimes}\mathbb{B}(\ell_2))$ is determined by the Arveson spectrum $Sp(\Delta_{\phi_\omega})$.
	\end{corollary}
	\begin{proof}
		Using the same construction as in Proposition \ref{proposition:typeII1} and replacing $\{u_{2n}\}$ by $\{1\otimes u_{2n}\}$, we have that $\mathbb{P}_{\omega\otimes\tr}(N\overline{\otimes}\mathbb{B}(\ell_2))$ is a factor. The centralizer of the Poisson algebra is a von Neumann algebra generated by $\{\Gamma(e^{ix}): e^{ix}\in (N\overline{\otimes}\mathbb{B}(\ell_2))_{\omega\otimes\tr}, e^{ix}\in \mathcal{U}^{\omega\otimes\tr}_{ana}\}$ where $e^{ix}$ is a unitary in the centralizer $(N\overline{\otimes}\mathbb{B}(\ell_2))_{\omega\otimes\tr}\cong N_\omega\overline{\otimes}\mathbb{B}(\ell_2)$. Hence it is isomorphic to $\mathbb{P}_{\omega\otimes\tr}(N_\omega\overline{\otimes}\mathbb{B}(\ell_2))$ and it is a factor. 
		
		It is well-known that the Arveson spectrum $Sp(\Delta_{\phi_{\omega\otimes\tr}})$ is minimal if the centralizer is a factor \cite{T}. Hence the Connes invariant of the Poisson algebra is given by the group generated by the Arveson spectrum of the Poisson state:
		\begin{equation*}
			S(\mathbb{P}_{\omega\otimes\tr}(N\overline{\otimes}\mathbb{B}(\ell_2))) = \Gamma\big(Sp(\Delta_{\phi_{\omega\otimes\tr}}) - \{0\}\big)\pl.
		\end{equation*}
		Since the modular automorphism group of the Poisson state $\phi_{\omega\otimes\tr}$ is a tensor product: $\sigma^{\phi_\omega}_t\otimes id = \Gamma(\sigma^\omega_t)\otimes id$, the Arveson spectrum equals to its restriction to the first component: $Sp(\Delta_{\phi_\omega})$. Lastly, since the $L_2$-space of the Poisson algebra is isomorphic to the symmetric Fock space of the $L_2$-space of the original algebra, the multiplicative subgroup generated by $Sp(\Delta_{\phi_\omega})$ equals to the multiplicative subgroup generated by $Sp(\Delta_\omega)$. Therefore the type of the Poisson algebra $\mathbb{P}_{\omega\otimes\tr}(N\overline{\otimes}\mathbb{B}(\ell_2))$ is determined by the spectrum $Sp(\Delta_\omega)$.
	\end{proof}
	\subsection{Type III Algebras from Principal Series Representation}\label{subsection:representation}
	Using the result on the type of Poisson algebra, in this section we construct examples of type III$_\lambda$ Poisson factors from unitary principal series representations. In addition, we will present a framework to construct algebraic quantum field theories using unitary principal series representations and Poissonization. The examples from this construction satisfy certain axioms in AQFT \cite{haag}. A dedicated study of Poissonization and AQFT with more constructions will come in a separate paper.
	
	\subsubsection{Unitary Principal Series Representations and Parabolic Induction}\label{subsubsection:induction}We first recall the construction of unitary principal series representation as an induced representation \cite{Knapp1, bern, CH, Rep}. Let $G$ be a real semisimple Lie group and let $G\cong K A  N$ be the Iwasawa decomposition of $G$. Let $M$ be the centralizer of $A$ in $K$, then $P = MAN$ is a minimal parabolic subgroup of $G$. In addition, let $\mathfrak{a}$ be the Lie algebra of $A$ such that $A = \exp(\mathfrak{a})$. Let $\Sigma^+$ be the set of positive roots of $\mathfrak{a}$ in the Lie algebra $\mathfrak{g}$. Fix a purely imaginary root $\nu\in\mathfrak{a}^*$ and fix a unitary irreducible representation $\sigma:M\rightarrow U(H_\nu)$ (since $K$ is compact, $M$ is compact and $H_\nu$ is a finite dimensional Hilbert space), then there exists a unitary representation of the parabolic group $P$ on $H_\nu$:
	\begin{equation}
		\mu_P:=\sigma\otimes e^\nu \otimes id: P = MAN\rightarrow U(H_\nu): man\mapsto \sigma(m)e^{\nu(\log a)}
	\end{equation}
	where $\log a\in\mathfrak{a}$ is the generator of $a\in A$. Let $\rho_G := \frac{1}{2}\sum_{\alpha \in\Sigma^+}\alpha$ be the half-sum of the roots in $\Sigma^+$. Then the induced unitary principal series representation of $G$ acts on the Hilbert space of half-densities over $G/P$ with value in the Hilbert space $H_\nu$:
	\begin{equation}
		\Gamma^2(G/P, H_\nu) \cong \{F: G\rightarrow H_\nu: F(gman) = \sigma(m^{-1})e^{-(\nu + \rho_G)(\log a)}F(g), \int_{G/P}|F|^2_{H_\nu} < \infty\}
	\end{equation} 
	where $G/P \cong K/M\cap K$ is a compact manifold and $|\cdot|_{H_\nu}$ is the $P$-invariant inner-product on $H_\nu$. The integration on $G/P$ is over the density $|F|^2_{H_\nu}$. The group acts on the half-density sections by left-multiplication: $\mathcal{P}^{\sigma, \nu}(g)F(h):=F(g^{-1}h)$. If the representation $\sigma$ is irreducible, then $\mathcal{P}^{\sigma,\nu}$ is also irreducible. This procedure of constructing irreducible unitary representations is called parabolic induction.  
	\subsubsection{Type III Poisson Factors from Parabolic Induction}\label{subsubsection:PoissonParabolic}
	Given a irreducible unitary principal series representation $\mathcal{P}^{\sigma,\nu}$ of a real semisimple Lie group $G$, we construct a type III$_\lambda$ Poisson factor with strongly continuous $G$-action. Let $\mathcal{P}^{\sigma,\nu}$ be an irreducible unitary principal series representation of $G$ induced by the finite dimensional unitary representation: $\mu_P:=\sigma\otimes e^\nu \otimes id$ (i.e. $\mathcal{P}^{\sigma,\nu} = Ind^G_P(\mu_P)$). Fix an element $a\in A$ and let $\log\delta_\nu$ be the generator of the unitary $\mu_P(a)$ (i.e. $\delta_\nu^{it} = \mu_P(e^{t\log a})$ and hence $\log \delta_\nu = |\nu(\log a)|$). And consider the weight $\omega_\nu = \tr(\delta\cdot)$ on $\mathbb{B}(H_\nu)$. Then we have:
	\begin{lemma}\label{lemma:PoissonParabolic}
		 The following statements are true:
		 \begin{enumerate}
		 	\item $\mathbb{P}_{\omega_\nu\otimes\tr}\mathbb{B}(H_\nu\otimes\ell_2)$ is the hyperfinite II$_1$ factor;
		 	\item Fix a unitary irreducible representation $\sigma:M\rightarrow U(H_\sigma)$. For two different purely imaginary roots $\nu,\mu\in\mathfrak{a}^*$, $\mathbb{B}(H_\nu\oplus H_\mu) = \mathbb{B}(H_\sigma)\otimes M_2(\mathbb{C})$ and the normalization of the weight $\tr(\delta_\nu \oplus\delta_\mu\cdot)$ is given by $ \tr\otimes \tr_2(\begin{pmatrix}
		 			\frac{e^{|\nu(\log a)|}}{e^{|\nu(\log a )|} + e^{|\mu(\log a)|}} & 0 \\ 0 & \frac{e^{|\mu(\log a)|}}{e^{|\nu(\log a )|} + e^{|\mu(\log a)|}}
		 		\end{pmatrix}\cdot)$. Denote this weight $\omega_{\nu,\mu}$. Then if $\lambda = \min\{e^{|\nu\log a| - |\mu\log a|}, e^{|\mu\log a| - |\nu\log a|}\} \in (0,1)$, the Poisson algebra $\mathbb{P}_{\omega_{\nu,\mu}\otimes\tr}\mathbb{B}\big((H_\nu\oplus H_\mu)\otimes\ell_2\big)$ is a type III$_\lambda$ factor;
	 		\item For three different purely imaginary roots $\nu,\mu,\eta\in\mathfrak{a}^*$, $\mathbb{B}(H_\nu\oplus H_\mu\oplus H_\eta)  = \mathbb{B}(H_\sigma)\otimes M_3(\mathbb{C})$ and the normalization of the weight $\tr(\delta_\nu\oplus\delta_\mu\oplus\delta_\eta\cdot)$ is given by
	 		\begin{equation*}
	 			\tr\otimes\tr_3(\begin{pmatrix}
	 				\frac{e^{|\nu\log a|}}{e^{|\nu\log a|} + e^{|\mu \log a|} + e^{|\eta\log a|}} & 0 &0 \\ 0 & \frac{e^{|\mu\log a|}}{e^{|\nu\log a|} + e^{|\mu \log a|} + e^{|\eta\log a|}} & 0 \\ 0 & 0  & \frac{e^{|\eta\log a|}}{e^{|\nu\log a|} + e^{|\mu \log a|} + e^{|\eta\log a|}}
	 			\end{pmatrix}\cdot)\pl.
	 		\end{equation*}
 			And the Poisson algebra $\mathbb{P}_{\omega_{\nu,\mu,\eta}\otimes\tr}\mathbb{B}\big((H_\nu\oplus H_\mu\oplus H_\eta)\otimes \ell_2\big)$ is of type III$_1$ if the triplets $\lambda_{\nu,\mu} := \min\{e^{|\nu\log a| - |\mu\log a|}, e^{|\mu\log a| - |\nu\log a|}\}$, $\lambda_{\nu,\eta} := \min\{e^{|\nu\log a| - |\eta\log a|}, e^{|\eta\log a| - |\nu\log a|}\}$ and $\lambda_{\mu,\eta} := \min\{e^{|\eta\log a| - |\mu\log a|}, e^{|\mu\log a| - |\eta\log a|}\}$ generate the multiplicative group $\mathbb{R}_+$.
		 \end{enumerate}
	\end{lemma}
	\begin{remark}
		Therefore for generic pairs of irreducible unitary principal series representations, the induced Poisson algebras are of type III$_\lambda$. And for generic multiplets of irreducible unitary principal series representations, the induced Poisson algebras are of type III$_1$.
		
		The construction we gave here is nothing but an adaptation of the Araki-Woods factor to the Poissonization functor.
	\end{remark}
	\begin{proof}
		The key observation is that the weight $\omega_\nu$ is invariant under $P$-action. Using the Levi decomposition of $P = LN$, $L$-action preserves $\omega_\nu$ since $A$ is Abelian and $M$ commutes with $A$. Since $N$ acts trivially on $H_\nu$, $\omega_\nu$ is preserved by the entire $P$-action. Therefore the automorphism on $\mathbb{B}(H_\nu\otimes\ell_2)$ given by the adjoint action $\mu_P(g)\cdot\mu_P(g)^*$ (where $g\in P$) preserves the weight $\omega_\nu$. Hence by Proposition \ref{proposition:collectionFunctorialWeight}, there exists a strongly continuous $P$-action on the Poisson algebra $\mathbb{P}_{\omega_\nu\otimes\tr}\mathbb{B}(H_\nu\otimes\ell_2)$. The unitary implementation of this automorphism gives a unitary representation of $P$ on the Haagerup $L_2$-space of the Poisson algebra:
		\begin{equation}
			U_{\Gamma(\mu_P)}:P\rightarrow U(L_2(\mathbb{P}_{\omega_\nu\otimes\tr}\mathbb{B}(H_\nu\otimes\ell_2), \phi_{\omega_\nu\otimes\tr}))\pl.
		\end{equation}
		In addition, by Corollary \ref{corollary:typeIII} the Poisson algebra is a factor. Then all three statements follow directly from a simple calculation of the corresponding Arveson spectrums.
	\end{proof}
	\begin{corollary}\label{corollary:PoissonParabolic}
		The unitary representation $U_{\Gamma(\mu_P)}$ constructed in the proof of Lemma \ref{lemma:PoissonParabolic} induces a unitary representation of $G$ via parabolic induction. In addition, this group action induces a strongly continuous automorphism on the Poisson algebra:
		\begin{equation}
			\mathbb{B}(L^2(G/P, \Omega_P^{1/2}))\overline{\otimes}\mathbb{P}_{\omega_{\nu,\mu}\otimes\tr}\mathbb{B}((H_\nu\oplus H_\mu)\otimes\ell_2) \cong \mathbb{P}_{\omega_{\nu,\mu}\otimes\tr}\mathbb{B}((H_\nu\oplus H_\mu)\otimes\ell_2)
		\end{equation}
		where the Hilbert space $L^2(G/P, \Omega_P^{1/2}) := \{f:G \rightarrow \mathbb{C}: f(gman) = e^{-\rho_G(\log a)}f(g), \int_{G/P}|f|^2 < \infty\}$ is the space of half-densities over $G/P$. The isomorphism of the von Neumann algebras follows from the absorption property of type III factors.
	\end{corollary}
	\begin{proof}
		To see that we have an induced representation, we need to calculate that the unitary representation $U_{\Gamma(\mu_P)}$:
		\begin{align}
			\begin{split}
				U_{\Gamma(\mu_P)}(man)\lambda_{\emptyset\emptyset}(x_1,...,x_n)\phi_{\omega_\nu\otimes\tr}^{1/2} &= \lambda_{\emptyset\emptyset}(Ad_{\mu_P(man)}x_1,...,Ad_{\mu_P(man)}x_n)\phi_{\omega_\nu\otimes\tr}^{1/2}
				\\
				&=\lambda_{\emptyset\emptyset}(Ad_{\sigma(m)e^{\nu(\log a)}}x_1, ..., Ad_{\sigma(m)e^{\nu(\log a)}}x_n)\phi_{\omega_\nu\otimes\tr}^{1/2}
				\\
				& =U_{\Gamma(\mu_P)}(ma)\lambda_{\emptyset\emptyset}(x_1,...,x_n)\phi_{\omega_\nu\otimes\tr}^{1/2}\pl.
			\end{split}
		\end{align}
		Hence the nilpotent subgroup $N$ does not act on the Haagerup $L_2$-space of the Poisson algebra. Denote this Hilbert space as $\mathcal{H}_\mathbb{P}$, we consider the space of half-densities on $G/P$ with value in $\mathcal{H}_\mathbb{P}$: 
		\begin{equation}
			\Gamma^2(G/P, \mathcal{H}_\mathbb{P}) := \{F:G\rightarrow \mathcal{H}_\mathbb{P}: F(gman) = U_{\Gamma(\mu_P)}(ma)^*e^{-\rho_G(\log a)}F(g), \int_{G/P}|F|^2_{\mathcal{H}_\mathbb{P}} < \infty\}\pl.
		\end{equation}
		The type I factor $\mathbb{B}(L^2(G/P, \Omega^{1/2}_P))$ acts on $\Gamma^2(G/P, \mathcal{H}_\mathbb{P})$ be left regular action, while the Poisson factor acts on $\mathcal{H}_\mathbb{P}$ by standard representation. Finally, the unitary representation of $G$ on $\Gamma^2(G/P, \mathcal{H}_\mathbb{P})$ induces a strongly continuous automorphism on $\mathbb{B}(L^2(G/P, \Omega^{1/2}_P))\overline{\otimes}\mathbb{P}_{\omega_\nu\otimes\tr}\mathbb{B}(H_\nu\otimes\ell_2)$. 
	\end{proof}
	\begin{example}\label{example:ConfGroup}
		Let $G = SO_0(n+1,1)$ be the (connected) conformal group on $\mathbb{R}^n$. Then the Iwasawa decomposition of $G$ is given by $K = SO(n)$ (the rotation), $A = SO_0(1,1)\cong \{\begin{pmatrix}
			\cosh\theta & \sinh\theta \\ \sinh\theta & \cosh\theta
		\end{pmatrix}:\theta\in \mathbb{R}\}$ (the boost or dilation), and $N = \mathbb{R}^n$ (the translation). The centralizer of $A$ in $K$ is given by $M = SO(n-1)$. Since the real Lie algebra of $A$ is $\mathfrak{a}_\mathbb{R} = \mathbb{R}\begin{pmatrix}0&1\\1&0\end{pmatrix}$, a purely imaginary root on $\mathfrak{a}$ is given by an imaginary number $\nu = it_\nu \in i\mathbb{R}$. And for an arbitrary element $a = \begin{pmatrix}
		\cosh\theta & \sinh\theta \\ \sinh\theta &\cosh\theta
		\end{pmatrix}\in SO_0(1,1)\cong A$, the pairing can be calculated explicity: $\nu(\log a) = it_\nu \theta$. In this example, we fix the boost: $\theta = 2\pi$. Then the key parameter $\nu(\log a) = 2\pi i t_\nu$ where $t_\nu\in\mathbb{R}$. The quotient space is the sphere $G/P = K/ M\cap K = SO(n) / SO(n-1) \cong \mathbb{S}^n$. Then for any unitary irreducible representation $U_\nu:SO(n-1)\rightarrow U(H_\nu)$, there exists a Poisson algebra $\mathbb{B}(L^2(\mathbb{S}^n, \Omega^{1/2}_P))\overline{\otimes}\mathbb{P}_{\omega_\nu\otimes\tr}\mathbb{B}(H_\nu\otimes\ell_2)$ where the weight on $\mathbb{B}(H_\nu)$ is given by $\omega_\nu = \tr(e^{2\pi t_\nu}\cdot)$. For a pair $t_\nu, t_\mu\in\mathbb{R} $ such that $\lambda_{\nu,\mu} = \min\{e^{2\pi(|t_\nu| - |t_\mu|)}, e^{2\pi (|t_\mu| - |t_\nu|)}\} \in (0,1)$, then $\mathbb{P}_{\omega_{\nu,\mu}\otimes\tr}\mathbb{B}((H_\nu\oplus H_\mu)\otimes \ell_2)$ is a type III$_\lambda$ factor. Similar results hold for general multiplets of $\{t_\nu\}$'s.
	\end{example}
	\subsubsection{Nets of Poisson Algebras from Unitary Principal Series Representations}\label{subsubsection:toyAQFT}
	As an application of the constructions so far, we briefly discuss algebraic quantum field theories (AQFT) \cite{haag}. The AQFT approach studies the local algebras of observables in a quantum field theory. It is an axiomatic framework that emphasizes on the necessary properties these algebras must satisfy. Briefly these properties are isotony, covariance under a (strongly-continuous) group action, the existence of a group invariant state, and locality. The axioms are often modified to accommodate different physical situations. For example, one would often require there to be a unique group-invariant state in order to study the so-called vacuum sector of the underlying AQFT. As another example, one would often require the group representation to satisfy certain spectral conditions (i.e. positive energy condition) in order to ensure the existence of well-defined "Hamiltonian". \footnote{Note however, there are natural examples where the positive energy condition is not satisfied \cite{BDH,KLM} even in the case of conformal nets (i.e. $G = SU(1,1)$). The positive energy condition requires the group action on the local algebras to come from discrete series representation. In our examples, although the group action is not discrete series representation, the modular Hamiltonian is still well-defined.} 
	
	Here we only consider the case of conformal group $G = SO(n+1,1)$ which is the group of conformal transformations on $\mathbb{R}^n$. Using principal series representation and Poissonization, we construct nets of type III$_1$ local Poisson algebras and check that the nets satisfy the desired axioms. Our local algebras will be defined on the one-point compactification of $\mathbb{R}^n$, namley the sphere $\mathbb{S}^n \cong G/P$ (c.f. Example \ref{example:ConfGroup}). Mathematically, the local algebras form a pre-cosheaf over the category of connected open subsets on $\mathbb{S}^n$. \footnote{If the group $G$ is the conformal group acting on the Minkowski space $\mathbb{R}^{n,1}$, then we must consider local algerbas on certain distinguished open subsets (i.e. proper open double cones). One can still define such cones on $\mathbb{R}^{n+1}$. But since our example is a "Euclidean" field theory model, there is no causal constraint that prevents us from considering more general open subsets. } To state the proposition, we fix a unitary representation of the minimal parabolic group $\mu_P: P \rightarrow U(H_\mu)$ where the spectrum of the generator $\log a $ generates the multiplicative group $\mathbb{R}_+$.\footnote{This assumption is not mathematically unnecessary. We make this assumption only to ensure the local algebras are all type III$_1$ factors. It is believed that the local algebras of QFT should be of type III$_1$.} Denote the Poisson algebra $\mathbb{P}_{\omega_\mu\otimes\tr}\mathbb{B}(H_\mu\otimes\ell_2)$ as $\mathbb{P}_\mu$ with Poisson state $\phi_\mu$. And denote its Haagerup $L_2$-space as $\mathcal{H}_\mathbb{P}$. Then we have:
	\begin{proposition}\label{proposition:PoissonNet}
		Consider the twisted Poisson algebra bundle $G\times_P \mathbb{P}_\mu$ over $G/P = \mathbb{S}^n$. The smooth section of this bundle is given by:
		\begin{equation}
			C^\infty(G/P,\mathbb{P}_\mu):=\{f\in C^\infty(G,\mathbb{P}): f(gp) = \alpha_p^{-1}f(g)\}
		\end{equation}
		where $\alpha:P\rightarrow Aut(\mathbb{P}_\mu)$ is the strongly continuous induced action of $P$ on the Poisson algebra. Analogously consider the twisted Hilbert bundle $G\times_P \mathcal{H}_\mathbb{P} $ over $G/P = \mathbb{S}^n$. Then the $L^2$-section of the tensor bundle of $\big(G\times_P\mathcal{H}_\mathbb{P}\big)\otimes\Omega^{1/2}_P$ is given by:
		\begin{equation}
			\Gamma^2(G/P, \mathcal{H}_\mathbb{P}) :=\{f\in L^2(G, \mathcal{H}_\mathbb{P}): f(gman) = U_{\mu_P}(ma)^*e^{-\rho_G(\log a)}f(g), \int_{G/P}|f|^2_{\mathcal{H}_\mathbb{P}}<\infty\}
		\end{equation}
		where $U_{\mu_P}:P\rightarrow U(\mathcal{H}_\mathbb{P})$ is the unitary implementation of the automorphism $\alpha$. This is induced by the unitary representation $\mu_P$. $C^\infty(G/P, \mathbb{P}_\mu)$ is a unital $*$-algebra acting on $\Gamma^2(G/P,\mathcal{H}_\mathbb{P})$. Denote its weak closure as $\mathcal{A}(\mathbb{S}^n)$. 
		
		Then for each connected open set $\mathcal{O}$ there exists a von Neumann algebra $\mathcal{A}(\mathcal{O}):= \{f\in\mathcal{A}(\mathbb{S}^n): supp(f) \ssubset\mathcal{O}\}$ acting on $\Gamma^2(G/P,\mathcal{H}_\mathbb{P})$. The net of local algebras $\{\mathcal{A}(\mathcal{O})\}_{\mathcal{O}\subset\mathbb{S}^n}$ satisfies the following properties:
		\begin{enumerate}
			\item For any $\mathcal{O}$, $\mathcal{A}(\mathcal{O})$ is type III$_1$ with nontrivial center $L^\infty(\mathcal{O})$;
			\item(Isotony) $\mathcal{A}(\mathcal{O}_1)\ssubset\mathcal{A}(\mathcal{O}_2)$ if $\mathcal{O}_1\ssubset\mathcal{O}_2$;
			\item($G$-covariance) There exists a unitary induced representation given by left regular action:
			\begin{equation}
				V:G\rightarrow U(\Gamma^2(G/P, \mathcal{H}_\mathbb{P})): (V_gF)(h) = F(g^{-1}h)\pl.
			\end{equation}
			By conjugation, this unitary representation induces a strongly continuous action $\widehat{\alpha}:G\rightarrow Aut(\mathcal{A}(\mathbb{S}^n))$ such that for any connect open set $\mathcal{O}$: $\alpha_g(\mathcal{A}(\mathcal{O})) = \mathcal{A}(g\mathcal{O})$ where $g\mathcal{O} = \{\xi\in\mathbb{S}^n: g^{-1}\xi\in\mathcal{O}\}$;
			\item(Existence of vacuum state) The normal state $\varphi:=\int_{G/P}d\xi \otimes\phi_\mathbb{P}$ where $d\xi$ is the unique $K$-invariant measure on $G/P = K/(M\cap K) = \mathbb{S}^n$\footnote{Note there is no $G$-invariant measure on $G/P$.} is invariant under the group action and it is represented by the vector $1\otimes\phi_\mathbb{P}^{1/2}\in \Gamma^2(G/P,\mathcal{H}_\mathbb{P})$. In addition, the expectation under this vacuum state does not split into a product of two-point functions (i.e. the vacuum state is not quasi-free and the consequently the model cannot be regarded as "generalized free fields");
			\item(Locality) $[\mathcal{A}(\mathcal{O}_1), \mathcal{A}(\mathcal{O}_2)] = 0$ if $\mathcal{O}_1\cap \mathcal{O}_2 = \emptyset$.
		\end{enumerate}
	\end{proposition}
	\begin{proof}
		The statements are all straight-forward based on previous discussions. Using the strongly continuous automorphism $\alpha:P\rightarrow Aut(\mathbb{P}_\mu)$, the Poisson algebra bundle is formed as the quotient: $G\times_P \mathbb{P}_\mu:= \{[(g,T)]: (gp, T) \sim (g,\alpha_p(T))\text{ for }g\in G\text{ , }T\in\mathbb{P}_\mu\}$.
		
		Then the smooth section of this bundle acts on $\Gamma^2(G/P, \mathcal{H}_\mathbb{P})$ by left multiplication. We only need to check that this action is well-defined:
		\begin{align}
			\begin{split}
				(f\xi)(gman) &= f(gman)\xi(gman) = \alpha_{man}^{-1}f(g)\big(U_{\mu_P}(ma)^*e^{-\rho_G(\log a)}\xi(g)\big) 
				\\&= U_{\mu_P}(ma)^*e^{-\rho_G(\log a)}f(g)\xi(g)
			\end{split}
		\end{align}
		where $f\in C^\infty(G/P,\mathbb{P}_\mu)$ and $\xi\in\Gamma^2(G/P, \mathcal{H}_\mathbb{P})$. Using the same proof that shows the usual principal series representations are unitary and the fact that the Poisson state $\phi_\mu$ is invariant under $P$-action, it follows that the left multiplication action is also unitary. All the properties are straight-forward. The state $\varphi$ is not quasi-free because the Poisson state $\phi_\mu$ is not quasi-free.  
	\end{proof}
	\subsection{Split Inclusion and Hyperfiniteness of Poisson Algebras}\label{subsection:split}
	In this subsection, we study the structure of Poisson algebra inclusions. Specifically, we will focus on a pair of von Neumann algebras $N\ssubset M$ with a normal semifinite faithful weight $\omega$ on $M$ such that its restriction to $N$ remains semifinite. Then Poissonization constructs a pair of Poisson algebras $(\mathbb{P}_\omega N \ssubset \mathbb{P}_\omega M, \phi_\omega)$, where the Poisson state on $\mathbb{P}_\omega M$ restricts to the Poisson state on $\mathbb{P}_\omega N$. For simplicity, we focus on the case where both Poisson algebras are factors. There are two situations to consider. If $\phi_\omega$ is tracial, then the Poisson algebras are type II$_1$ and we are in the classical set-up of the type II subfactor theory \cite{Jones1, PP1, Popa1, Ocn1}. On the other hand, if $\phi_\omega$ is not tracial, then the Poisson algebras are properly infinite (either containing a type I$_\infty$ factor or are of type III). By a classical theorem of Dixmier and Maréchal \cite{DM}, there exists a dense subset of vectors in $L_2(\mathbb{P}_\omega M, \phi_\omega)$ such that each vector is simultaneously cyclic and separating for both Poisson algebras. This is the classical set-up of type III subfactor theory \cite{DL1, Longo3, Kos1, KL1}. Unfortunately, for the Poisson algebra inclusion, the subfactor index is typically infinite. There is no definitive construction for the standard invariant of infinite index subfactors, although some possible answers have been provided \cite{Burns, P1, JP1} by considering the Jones construction with type II$_\infty$ algebras. Following this approach, we calculate directly the centralizer algebras and the central vectors \cite{Burns, P1, JP1} of type II$_1$ Poisson subfactors. A planar algebra calculus has also been developed in the case of infinite index \cite{P1}. We shall reserve this and a general study of bimodules over Poisson algebras for a future paper.
	
	The key observation for Poisson subfactors is the following lemma.
	\begin{lemma}\label{lemma:centralProj}
		Let $M$ be a von Neumann algebras and let $\omega$ be a n.s.f. weight on $M$. In addition, let $e\in M_\omega$ be a projection in the centralizer. Then the inclusion of the corner $N:= M_e \ssubset M$ induces Poisson algebra inclusion $\mathbb{P}_\omega N\ssubset \mathbb{P}_\omega M$ such that $\phi_\omega$ on $\mathbb{P}_\omega M$ restricts to the Poisson state on $\mathbb{P}_\omega N$. If $\mathbb{P}_\omega N$ is a factor, then the relative commutant is given by $\mathbb{P}_\omega M \cap (\mathbb{P}_\omega N)' = \mathbb{P}_\omega (M_{1-e})$. In addition, we have split isomorphism:
		\begin{equation}
			\mu_e:\mathbb{P}_\omega N \overline{\otimes}\mathbb{P}_\omega(M_{1-e}) \rightarrow \mathbb{P}_\omega N\vee\mathbb{P}_\omega(M_{1-e})\cong \mathbb{P}_\omega M: \Gamma(e^{ix})\otimes\Gamma(e^{iy})\mapsto \Gamma(e^{i(x+y)})\pl.
		\end{equation}
		In particular, the Poisson state $\phi_\omega$ splits:
		\begin{equation}
			\phi_\omega(\Gamma(e^{ix})\Gamma(e^{iy})) = \phi_\omega(\Gamma(e^{ix}))\phi_\omega(\Gamma(e^{iy}))
		\end{equation}
		where $e^{ix}\in N$ and $e^{iy}\in M_{1-e}$.
	\end{lemma}
	\begin{proof}
		We first show that the Haagerup $L_2$ space of $\mathbb{P}_\omega M$ factorizes into a tensor product. This is a simple consequence of the well-known result that the symmetric Fock space construction turns direct sum to tensor product. Let $n_\omega$ be the left ideal of $\omega$-finite elements in $M$. Then $en_\omega$ is the left ideal of $\omega$-finite elements in $N$:
		\begin{equation*}
			\omega(ex^*xe) < \infty
		\end{equation*}
		where we used the fact that the definition ideal $m_\omega = n_\omega^*n_\omega$ is a bimodule over the centralizer $M_\omega$ \cite{T}. In addition, we have orthogonality:
		\begin{equation*}
			\omega(ex^*y(1-e)) = \omega(x^*y(1-e)e) = 0
		\end{equation*}
		where $x,y\in n_\omega$ \cite{T}. Then by the Fock space model, we have:
		\begin{align}
			\begin{split}
				L_2(\mathbb{P}_\omega M, \phi_\omega) &\cong \mathcal{F}_s(L_2(M, \omega)) = \mathcal{F}_s(\overline{n_\omega\eta_\omega}) = \mathcal{F}_s(\overline{en_\omega \eta_\omega \oplus (1-e)n_\omega\eta_\omega}) \\&\cong \mathcal{F}_s(\overline{en_\omega\eta_\omega}) \otimes\mathcal{F}_s(\overline{(1-e)n_\omega \eta_\omega}) = \mathcal{F}_s(L_2(N, \omega))\otimes \mathcal{F}_s(L_2(M_{1-e}, \omega)) \\&\cong L_2(\mathbb{P}_\omega N, \phi_\omega)\otimes L_2(\mathbb{P}_\omega M_{1-e}, \phi_\omega)
			\end{split}
		\end{align}
		where $\eta_\omega$ is the representing vector in the semi-cyclic representation of $\omega$. Using the same calculation, we can directly show that the Poisson state factorizes:
		\begin{align}
			\begin{split}
				\phi_\omega(\Gamma(e^{ix})\Gamma(e^{iy})) &= \phi_\omega(\Gamma(e^{i(x+y)})) = \exp(\omega(e^{i(x+y)} - 1))
				\\
				&=\exp(\omega\big((e^{ix} - 1)(e^{iy} - 1)\big))\exp(\omega(e^{ix} - 1))\exp(\omega(e^{iy} - 1))\\& = \exp(\omega(e^{ix} - 1))\exp(\omega(e^{iy} - 1)) = \phi_\omega(\Gamma(e^{ix}))\phi_\omega(\Gamma(e^{iy}))
			\end{split}
		\end{align}
		where $e^{ix}-1\in em_\omega e = m_{\omega_e}\ssubset N$ and $e^{iy} - 1\in (1-e)m_\omega(1-e) = m_{\omega_{1-e}}\ssubset M_{1-e}$. Then the factorization follows by density. Combining these two calculations, we see that the multiplication map is a split isomorphism and it is unitarily implemented by the Fock space isomorphism. Finally, when $\mathbb{P}_\omega N$ is a factor, using the Fock space isomorphism we have: $(\mathbb{P}_\omega N)'\cap \mathbb{P}_\omega M = (\mathbb{P}_\omega N)' \cap (\mathbb{P}_\omega N\vee\mathbb{P}_\omega M_{1-e}) \cong J_N\mathbb{P}_\omega N J_N\otimes \mathbb{B}(L_2(\mathbb{P}_\omega M_{1-e}, \phi_\omega))\cap \mathbb{P}_\omega N\otimes \mathbb{P}_\omega M_{1-e} \cong \mathbb{P}_\omega M_{1-e}$ where $J_N$ is the modular conjugation of the Poisson algebra $\mathbb{P}_\omega N$ in its standard form. \qedhere
	\end{proof}
	\begin{corollary}\label{corollary:infIndex}
		Using the same set-up as Lemma \ref{lemma:centralProj}. If both $\mathbb{P}_\omega M$ and $\mathbb{P}_\omega N$ are factors, then the index $[\mathbb{P}_\omega M: \mathbb{P}_\omega N] = \infty$.
	\end{corollary}
	\begin{proof}
		If the index is finite, then the relative commutant is finite dimensional \cite{Longo3}. However in our case, the relative commutant is a Poisson algebra which is always infinite dimensional. Therefore the index is always infinite.
	\end{proof}
	We pause to comment on the infinite index type II$_1$ subfactors. By Proposition \ref{proposition:typeII1} and Corollary \ref{corollary:infIndex}, for any projection $e \sim id\in \mathbb{B}(\ell_2)$ that is Murray-von-Neumann equivalent to the identity, we have an infinite index type II$_1$ subfactor $\mathbb{P}_{\tr}\mathbb{B}(e\ell_2)\ssubset\mathbb{P}_{\tr}\mathbb{B}(\ell_2)$. It is possible to carry out Jones' construction on infinite index type II$_1$ subfactors. But it is not clear that the construction produces the standard invariant as in the finite index case. One candidate for the standard invariant of the infinite index subfactor was proposed in \cite{Burns, JP1, P1} in terms of the centralizer algebras and the central $L_2$-vectors and their corresponding planar calculi.
	\begin{definition}\label{definition:infIndexStdInv}
		Given an infinite index type II$_1$ subfactor $(N\ssubset M,\tr)$ where $\tr$ is the canonical trace. For $n\geq 0$, the centralizer algebras are defined by:
		\begin{equation}
			\mathcal{Q}_n:= N'\cap (N^{op})'\cap \mathbb{B}(\otimes^n_N L_2(M))
		\end{equation}
		where $\otimes^n_N L_2(M)$ is the Connes fusion tensor product of $N$-bimodule $L_2(M)$ and $N'$ is the commutant of the left $N$-action on the bimodule $\otimes^n_N L_2(M)$ while $(N^{op})'$ is the commutant of the right $N$-action. 
		
		The central $L_2$-vectors are defined by:
		\begin{equation}
			\mathcal{P}_n:=N'\cap \otimes^n_N L_2(M)
		\end{equation}
		where $N'$ is the commutant of the left $N$-action.
	\end{definition}
	For infinite index subfactors, the centralizer algebras do not coincide with the central vectors \cite{Burns}. However, they admit compatible planar calculi \cite{Jones2, P1}. For type II$_1$ Poisson subfactors, we can calculate these "invariants" explicitly.
	\begin{corollary}\label{corollary:planarPoisson}
		Let $e\sim id\in\mathbb{B}(\ell_2)$ be a projection that is Murray-von-Neumann equivalent to the identity. Then as an $\mathbb{P}_{\tr}\mathbb{B}(e\ell_2)$-bimodule, the Connes fusion product is given by:
		\begin{equation}\label{equation:PoissonBimod}
			\otimes^n_{\mathbb{P}_{\tr}\mathbb{B}(e\ell_2)} L_2(\mathbb{P}_{\tr}\mathbb{B}(\ell_2), \phi_{\tr}) \cong L_2(\mathbb{P}_{\tr}\mathbb{B}(e\ell_2), \phi_{\tr})\otimes\bigg(\otimes^n_\mathbb{C} \mathcal{F}_s(S_2((1-e)\ell_2))\bigg)
		\end{equation}
		where $S_2((1-e)\ell_2)$ is the Schatten 2-class on $(1-e)\ell_2$. On the right hand side, the Poisson algebra only acts on the bimodule $L_2(\mathbb{P}_{\tr}\mathbb{B}(e\ell_2), \phi_{\tr})$ while the symmetric Fock space is the multiplicity. The tensor product on the right hand side is the usual tensor product over $\mathbb{C}$.
		
		In addition, the centralizer algebras are the matrix algebra on the symmetric Fock space:
		\begin{equation}
			\mathcal{Q}_n = \mathbb{B}\bigg(\otimes^n_\mathbb{C}\mathcal{F}_s(S_2((1-e)\ell_2))\bigg)\pl.
		\end{equation}
		And the central vectors are given by:
		\begin{equation}
			\mathcal{P}_n = \widehat{\mathbb{P}_{\tr}\mathbb{B}(e\ell_2)'}\otimes\bigg(\otimes^n_\mathbb{C}\mathcal{F}_s(S_2((1-e)\ell_2))\bigg)
		\end{equation}
		where $\widehat{\mathbb{P}_{\tr}\mathbb{B}(e\ell_2)'}\ssubset L_2(\mathbb{P}_{\tr}\mathbb{B}(e\ell_2), \phi_{\tr})$ are the bounded vectors that commute with the left Poisson algebra action on its standard representation.
	\end{corollary} 
	\begin{proof}
		The claims follow easily once we proved the isomorphism \ref{equation:PoissonBimod}. However this is a simple consequence of Lemma \ref{lemma:centralProj}. As $\mathbb{P}_{\tr}\mathbb{B}(e\ell_2)$-bimodule, $L_2(\mathbb{P}_{\tr}\mathbb{B}(\ell_2), \phi_{\tr})\cong L_2(\mathbb{P}_{\tr}\mathbb{B}(e\ell_2), \phi_{\tr})\otimes \mathcal{F}_s(S_2((1-e)\ell_2))$ where the tensor product on the right hand side is the usual tensor product over $\mathbb{C}$ and the Poisson algebra only acts on the $L_2(\mathbb{P}_{\tr}\mathbb{B}(e\ell_2), \phi_{\tr})$ component. Since $L_2(\mathbb{P}_{\tr}\mathbb{B}(e\ell_2), \phi_{\tr})\otimes_{\mathbb{P}_{\tr}\mathbb{B}(e\ell_2)}L_2(\mathbb{P}_{\tr}\mathbb{B}(e\ell_2), \phi_{\tr}) = L_2(\mathbb{P}_{\tr}\mathbb{B}(e\ell_2), \phi_{\tr})$, then the claim follows.
	\end{proof}
	If $rk(1-e):=r < \infty$ has finite rank, then the tensor product can be written as $\otimes^n\mathcal{F}_s(S^{r}_2) = \mathcal{F}_s(\oplus^nS^{r}_2)$. In the work of \cite{JP1}, the authors showed that there exists a faithful embedding of the gauge-invariant CAR algebras into the centralizer algebras. In the case of Poisson subfactors, Corollary \ref{corollary:planarPoisson} gives faithful representations of "fermionic" algebras on "bosonic" (symmetric) Fock space. The relation to the Jordan-Wigner transformation and a more detailed study of the type II$_1$ Poisson subfactors will be conducted in a seperate paper. 
	
	In the final part of this subsection, we turn our attention to the type III Poisson subfactors. To fix the notation, we consider a pair of von Neumann algerbas $(N\ssubset M, \omega)$ where the group generated by the Arveson spectrum $Sp(\omega_N)$ in $\mathbb{R}_+$ is either $\mathbb{R}_+$ or $\langle\lambda^n\rangle$. Then by Corollary \ref{corollary:typeIII}, both $\mathbb{P}_{\omega\otimes\tr}(N\overline{\otimes}\mathbb{B}(\ell_2))$ and $\mathbb{P}_{\omega\otimes\tr}(M\overline{\otimes}\mathbb{B}(\ell_2))$ are type III factors. One easy corollary of Lemma \ref{lemma:centralProj} is the following observation:
	\begin{corollary}\label{corollary:PoissonAbsorption}
		Given a type III Poisson factor of the form $\mathbb{P}_{\omega\otimes\tr}(M\overline{\otimes}\mathbb{B}(\ell_2))$, then for any $k\in \mathbb{N}$, we have normal isomorphism:
		\begin{equation}
			\mathbb{P}_{\omega\otimes\tr}(M\overline{\otimes}\mathbb{B}(\ell_2))^{\otimes k}\cong \mathbb{P}_{\omega\otimes\tr}(M\overline{\otimes}\mathbb{B}(\ell_2))\pl.
		\end{equation}
	\end{corollary}
	\begin{proof}
		By induction, it suffices to prove the statement for $k = 2$. Fix a projection $e\in \mathbb{B}(\ell_2)$ such that $e\sim 1 - e \sim 1$. Then since $e$ is in the centralizer $(M\overline{\otimes}\mathbb{B}(\ell_2))_{\omega\otimes\tr}$, by Lemma \ref{lemma:centralProj} we have normal isomorphism: $\mathbb{P}_{\omega\otimes\tr}(M\overline{\otimes}\mathbb{B}(e\ell_2))\overline{\otimes}\mathbb{P}_{\omega\otimes\tr}(M\overline{\otimes}\mathbb{B}((1-e)\ell_2))\cong \mathbb{P}_{\omega\otimes\tr}(M\overline{\otimes}\mathbb{B}(\ell_2))$. Since $\mathbb{B}(e\ell_2) \cong \mathbb{B}((1-e)\ell_2)\cong\mathbb{B}(\ell_2)$, then we have the desired isomorphism.
	\end{proof}
	Recall the definition of a split inclusion \cite{DL1, Longo2}:
	\begin{definition}\label{definition:split}
		Let $N\ssubset M$ be two von Neumann algebras acting on the same Hilbert space $\mathcal{H}$. Then the inclusion split if there exists a normal isomorphism: $N\vee M' \cong N\overline{\otimes}M'$. 
	\end{definition}
	The split inclusion is an important concept in algebraic quantum field theory. Briefly, splitting property can be understood as a mathematical characterization of physical locality. Here we show that type III Poisson subfactors of the form $\mathbb{P}_{\omega\otimes\tr}(N\overline{\otimes}\mathbb{B}(\ell_2))\ssubset\mathbb{P}_{\omega\otimes\tr}(M\overline{\otimes}\mathbb{B}(\ell_2))$. 
	\begin{lemma}\label{lemma:PoissonSplit}
		Let $N\ssubset M$ be two hyperfinite von Neumann algebras and let $\omega$ be a n.s.f. weight on $M$ such that the group generated of the Arveson spectrum $Sp(\omega_N)$ is either $\mathbb{R}_+$ or $\langle\lambda^n\rangle$. Then the inclusion $\mathbb{P}_{\omega\otimes\tr}(N\overline{\otimes}\mathbb{B}(\ell_2))\ssubset\mathbb{P}_{\omega\otimes\tr}(M\overline{\otimes}\mathbb{B}(\ell_2))$ is a split inclusion. In particular, there exists an intermediate type I subfactor:
		\begin{equation}
			\mathbb{P}_{\omega\otimes\tr}(N\overline{\otimes}\mathbb{B}(\ell_2))\ssubset\mathbb{B}\bigg(L_2\big(\mathbb{P}_{\omega\otimes\tr}(M\overline{\otimes}\mathbb{B}(\ell_2)), \phi_{\omega\otimes\tr}\big)\bigg)\ssubset\mathbb{P}_{\omega\otimes\tr}(M\overline{\otimes}\mathbb{B}(\ell_2))\pl.
		\end{equation}
	\end{lemma}
	\begin{proof}
		The normal isomorphism in Corollary \ref{corollary:PoissonAbsorption} $\mathbb{P}_{\omega\otimes\tr}(M\overline{\otimes}\mathbb{B}(\ell_2))^{\otimes 2}\cong\mathbb{P}_{\omega\otimes\tr}(M\overline{\otimes}\mathbb{B}(\ell_2)) $ is implemented by the unitary isomorphism:
		\begin{equation}
			U: L_2(\mathbb{P}_{\omega\otimes\tr}(M\overline{\otimes}\mathbb{B}(\ell_2)),\phi_{\omega\otimes\tr})^{\otimes 2} \rightarrow L_2(\mathbb{P}_{\omega\otimes\tr}(M\overline{\otimes}\mathbb{B}(\ell_2)), \phi_{\omega\otimes\tr})\pl.
		\end{equation}
		Notice that the tensor product is over $\mathbb{C}$ not the Connes fusion product. Then on the left hand side, the faithful left action of the Poisson subalgebra $\mathbb{P}_{\omega\otimes\tr}(N\overline{\otimes}\mathbb{B}(\ell_2))$ is contained in the type I factor:
		\begin{equation}
			\pi_l\bigg(\mathbb{P}_{\omega\otimes\tr}(N\overline{\otimes}\mathbb{B}(\ell_2))\bigg) \cong \mathbb{P}_{\omega\otimes\tr}(N\overline{\otimes}\mathbb{B}(\ell_2))\otimes \mathbb{C} \ssubset \mathbb{B}\bigg(L_2\big(\mathbb{P}_{\omega\otimes\tr}(M\overline{\otimes}\mathbb{B}(\ell_2)), \phi_{\omega\otimes\tr}\big)\bigg)\otimes\mathbb{C}\pl.
		\end{equation}
		Similarly on the tensor product, the faithful right action of the Poisson algebra $\mathbb{P}_{\omega\otimes\tr}(M\overline{\otimes}\mathbb{B}(\ell_2))$ has a commutant that contains the same type I factor:
		\begin{equation}
			\pi_r\bigg(\mathbb{P}_{\omega\otimes\tr}(M\overline{\otimes}\mathbb{B}(\ell_2))\bigg)' \cong \mathbb{B}\bigg(L_2\big(\mathbb{P}_{\omega\otimes\tr}(M\overline{\otimes}\mathbb{B}(\ell_2)), \phi_{\omega\otimes\tr}\big)\bigg)\otimes \mathbb{P}_{\omega\otimes\tr}(M\overline{\otimes}\mathbb{B}(\ell_2))\pl.
		\end{equation}
		Since $U\pi_r\bigg(\mathbb{P}_{\omega\otimes\tr}(M\overline{\otimes}\mathbb{B}(\ell_2))\bigg)U^* = J\mathbb{P}_{\omega\otimes\tr}(M\overline{\otimes}\mathbb{B}(\ell_2)) J$ where $J$ is the modular conjugation of the Poisson state $\phi_{\omega\otimes\tr}$, then $U\mathbb{B}\bigg(L_2\big(\mathbb{P}_{\omega\otimes\tr}(M\overline{\otimes}\mathbb{B}(\ell_2)), \phi_{\omega\otimes\tr}\big)\bigg)\otimes\mathbb{C}U^*$ is the intermediate type I factor we are looking for. 
		
		To finish the proof, the relative commutant $\mathbb{P}_{\omega\otimes\tr}(N\overline{\otimes}\mathbb{B}(\ell_2))'\cap \mathbb{P}_{\omega\otimes\tr}(M\overline{\otimes}\mathbb{B}(\ell_2))$ is properly infinite. Again using the unitary isomorphism $U$, both the commutant $\mathbb{P}_{\omega\otimes\tr}(N\overline{\otimes}\mathbb{B}(\ell_2))'$ and the Poisson algebra $\mathbb{P}_{\omega\otimes\tr}(M\overline{\otimes}\mathbb{B}(\ell_2))$ absorb the Poisson algebra $\mathbb{P}_{\omega\otimes\tr}(M\overline{\otimes}\mathbb{B}(\ell_2))$ itself. Hence the relative commutant absorbs the type III Poisson factor as well. Hence as an infinite von Neumann algebra, the relative commutant must be properly infinite. Hence by a result in \cite{DL1}, when the relative commutant is properly infinite, an inclusion is split if and only if there exists an intermediate type I factor. Hence we have the desired result.
	\end{proof}
	
	\subsection{Relative Entropy in Poisson Algebras}\label{subsection:relEnt}
	As a last application of Poissonization, we show how to calculate quantum relative entropy for certain states in Poisson algebras. When the Poisson algebra is of type III, this calculation gives rigorous formula for the relative entropy in a type III algebra.
	
	To state the main mathematical result in full generality, we first recall Araki's definition of relative entropy and Lindblad's generalization. 
	\begin{definition}[Araki \cite{Ara1, Ara2}]
		Let $N$ be a von Neumann algebra and $\rho,\psi$ be two states on $N$. Then the relative entropy is given by:
		\begin{equation}\label{equation:relEnt}
			D(\rho | \psi) := \langle\eta_\psi, -\log\Delta_{\rho, \psi}\eta_\psi\rangle	
		\end{equation}
		where $\eta_\psi\in L_2(N,\psi)$ is the vector representation of the state $\psi$ and $\Delta_{\rho,\psi}$ is the relative modular operator between the state $\rho$ and $\psi$.
	\end{definition}
	\begin{definition}[Lindblad \cite{Lin}]
		Let $N$ be a finite dimensional von Neumann algebra and $\rho,\psi\in L_1(N)$ be two positive densities (not necessarily normalized). Then the Lindblad relative entropy is given by:
		\begin{equation}\label{equation:LindEnt}
			D_{Lin}(\rho|\psi):= Tr(\rho(\log\rho - \log\psi)) + Tr(\psi) - Tr(\rho)
		\end{equation}
		where $Tr(\cdot)$ is the canonical trace on $N$.
	\end{definition}
	In the case where $N$ is finite dimensional and both $\rho$ and $\psi$ are normalized, the Lindblad relative entropy reduces to the Araki's definition. It is not clear how to define relative entropy when $\rho$ and $\psi$ are unbounded weights. However in the case where one weight is a "small" perturbation of the other, the Lindblad relative entropy can be generalized:
	\begin{definition}\label{definition:genLindbladEntropy}
		Let $N$ be a semifinite von Neumann algebra and let $\rho,\psi$ be two n.s.f weigths such that $\rho \leq \psi$ and $\psi - \rho\in N_*$ is a normal bounded linear functional. Then the generalized Lindblad relative entropy is given by:
		\begin{equation}\label{equation:genLindEnt}
			D_{Lin}(\rho | \psi) := \langle\eta_\psi, -\log\Delta_{\rho, \psi}\eta_\psi\rangle + (\psi - \rho)(1)\pl.
		\end{equation}
	\end{definition}
	When $N$ is finite dimensional and $\rho,\psi$ are bounded functionals, the generalized Lindblad relative entropy reduces to the familiar Lindblad relative entropy. Just as compact operators are treated as small perturbations around bounded operators in noncommutative geometry, the perturbation $\psi - \rho$ can be understood as a small bounded perturbation around an unbounded weight. 
	
	Using this generalized notion of relative entropy, we will calculate the relative entropy between two Poisson states $\phi_\rho$ and $\phi_\psi$. One difficulty in defining this relative entropy is that on a Poisson algebra $\mathbb{P}_\psi N$, it is not clear that $\phi_\rho$ is a normal state because Poissonization only constructs $\phi_\rho$ as a state on another Poisson algebra, namely $\mathbb{P}_\rho N$. However, when $\rho$ is a bounded perturbation around $\psi$, we can define both states on the Poisson algebra $\mathbb{P}_\psi N$. \footnote{In general, when the two weights $\rho, \psi$ do not satisfy the conditions of Proposition \ref{proposition: PoissonRelEnt}, there is no guarantee that the Poisson states $\phi_\rho, \phi_\psi$ can be simultaneously defined on the same Poisson algebra. In fact, if $\psi = \lambda\rho$ for some constant $\lambda > 0$ and assume both are well-defined states on the Poisson algebra $\mathbb{P}_\psi N$, then using Powers-Størmer inequality and a net of approximating states, one can show that these two states are necessarily orthogonal.}
	\begin{proposition}\label{proposition: PoissonRelEnt}
		Let $N$ be a semifinite von Neumann algebra and let $\rho, \psi$ be two n.s.f. weights on $N$. In addition let $\mathbb{P}_\psi N$ be the Poisson algebra and $\phi_{\psi}$ be the corresponding Poisson state. Assume $\rho \leq \psi$ and $\psi - \rho$ is a bounded positive functional. Then on $\mathbb{P}_\psi N$, there exists a n.s.f. weight $\widetilde{\phi}_\rho$ such that:
		\begin{equation}
			(D\widetilde{\phi}_\rho : D\phi_\psi)_t  = \Gamma((D\rho: D\psi)_t)
		\end{equation}
		where $(D\rho:D\psi)_t \in N$ is Connes' cocycle \cite{T} between $\rho$ and $\psi$ and similarly for $(D\widetilde{\phi}_\rho:D\phi_\psi)_t$. Moreover the Poisson state $\phi_\rho$ is a well-defined state on $\mathbb{P}_\psi N$ and the cocycle between $\widetilde{\phi_{\rho}}$ and $\phi_\rho$ is given by: $(D\phi_\rho: D\widetilde{\phi_\rho})_t = \exp(-it(\psi- \rho)(1))$. Then the generalized Lindblad entropy is preserved
		\begin{equation}
			D(\phi_\rho | \phi_\psi) = D_{Lin}(\rho | \psi)\pl.
		\end{equation}
	\end{proposition} 
	\begin{proof}
		The proof is based on the simple observation that Poissonization naturally lifts a cocycle in $N$ to a cocycle in the Poisson algebra $\mathbb{P}_\psi N$. But the integrated weight is not the correct Poisson state. So the assumption on bounded perturbation is necessary to ensure that there exists another cocycle on the Poisson algebra to produce the correct Poisson state.
		
		Let $u_t:=(D\rho:D\psi)_t\in N$ be the cocycle between $\rho, \psi$ in N. Poissonization produces a unitary $\Gamma(u_t)\in \mathbb{P}_\psi N$ such that the cocycle condition holds:
		\begin{equation}
			\Gamma(u_t)\sigma_t^{\phi_\psi}(\Gamma(u_s)) = \Gamma(u_t)\Gamma(\sigma^\psi_t(u_s)) = \Gamma(u_t\sigma^\psi_t(u_s)) = \Gamma(u_{t+s})
		\end{equation} 
		where $\sigma^{\phi_\psi}_t$ is the modular automorphism of $\phi_\psi$ and similarly for $\sigma^\psi_t$. Hence by Connes-Matsuda theorem, there exists a n.s.f. weight $\widetilde{\phi_\rho}$ on the Poisson algebra $\mathbb{P}_\psi N$ such that: $(D\widetilde{\phi_\rho}:D\phi_\psi)_t = \Gamma(u_t)$.
		
		To calculate the $\widetilde{\phi_\rho}$, we first observe that $\Gamma(u_{-i/2})$ is a well-defined contraction on $L_2(\mathbb{P}_\psi N, \phi_\psi)$. Since $\rho \leq \psi$, $||u_{-i/2}|| \leq 1$ \cite{T} and using the isomorphism $\mathcal{F}_s(L_2(N,\psi))\cong L_2(\mathbb{P}_\psi N, \phi_\psi)$, the contraction $u_{-i/2}$ can be lifted to a contraction $\Gamma(u_{-i/2})$ such that:
		\begin{align}
			\begin{split}
				\big(\Gamma(u_{-i/2})\cdot\phi_\psi\cdot\Gamma(u_{-i/2}), \Gamma(x)\big) &= \phi_\psi(\Gamma(u_{-i/2})^*\Gamma(x)\Gamma(u_{-i/2})) = \exp(\psi(u_{-i/2}^*xu_{-i/2} - 1))\\& = \exp(\rho(x - 1))\exp((\rho-\psi)(1))
			\end{split}
		\end{align}
		where $x\in S^{\psi}_{ana} = \{y\in\mathcal{U}^{\psi}_{ana}: y - 1\in m_\psi, ||y||\leq 1\}$ and $\big(\cdot,\cdot\big)$ is the canonical pairing between the Poisson algebra and its predual. Since $\widetilde{\phi_\rho} = \Gamma(u_{-i/2})\cdot\phi_\psi\cdot\Gamma(u_{-i/2})$, then by density we see that the twisted version of $\widetilde{\phi_\rho}$ coincides with the Poisson state $\phi_\rho$:
		\begin{equation}
			\exp((\rho - \psi)(1))\widetilde{\phi_\rho}(\Gamma(x)) = \exp(\rho(x - 1)) =  \phi_\rho(\Gamma(x))\pl.
		\end{equation}
		Hence $\phi_\rho$ is well-defined state on the Poisson algebra $\mathbb{P}_\psi N$ and we have cocycle: $(D\phi_\rho:D\widetilde{\phi_\rho})_t = \exp(it(\rho - \psi)(1))$. In particular, $(D\phi_\rho:D\phi_\psi)_t = \exp(it(\rho - \psi)(1))\Gamma(u_t)$. Therefore by the definition of Araki's relative entropy, we have:
		\begin{align}
			\begin{split}
				D(\phi_\rho | \phi_\psi) &= -\langle\phi_\psi^{1/2}, \log\Delta_{\phi_\rho, \phi_\psi}\phi_\psi^{1/2}\rangle = -\langle\phi_\psi^{1/2}, \lambda(\log\Delta_{\rho,\psi})\phi_\psi^{1/2}\rangle + (\psi - \rho)(1)\\&
				=-\langle\psi^{1/2}, \log\Delta_{\rho, \psi}\psi^{1/2}\rangle + (\psi - \rho)(1) = D_{Lin}(\rho |\psi)\pl.
			\end{split}
		\end{align}
	\end{proof}
	A similar calculation also works for the Renyi entropies. In a separate paper, we will use this observation to study several other quantum information quantities on the Poisson algebras. It turns out that the Poisson algebra is a good toy model to study the so-call chaotic quantum systems.
	\nocite{*}
	\bibliographystyle{alpha}
	\bibliography{abstractPoiss}
\end{document}